\documentclass[12pt]{amsart}

\usepackage[paperheight=11in,paperwidth=8.5in,left=1in,top=1.25in,right=1in,bottom=1.25in,]{geometry}
\usepackage[linktocpage=false,colorlinks=true,linkcolor=black,citecolor=black,
filecolor=black,urlcolor=black,pagebackref=false,pdfstartpage={1},
pdfstartview={FitH},pdftitle={},pdfauthor={},pdfsubject={},pdfcreator={},
pdfproducer={},pdfkeywords={}]{hyperref}

\usepackage[lofdepth,lotdepth]{subfig}

\usepackage{comment}
\usepackage{afterpage}
\usepackage{amsfonts}
\usepackage{amsmath}
\allowdisplaybreaks[4]
\usepackage{amssymb}
\usepackage{amsthm}
\usepackage{bbm}
\usepackage[justification=centering]{caption}
\usepackage[capitalize,nameinlink,noabbrev,nosort]{cleveref}
\usepackage{enumerate}
\usepackage{enumitem}
\usepackage{float}
\restylefloat{figure}
\usepackage{kbordermatrix}
\usepackage{graphicx}
\usepackage{mathrsfs}
\usepackage{mathtools}
\usepackage[framemethod=tikz,ntheorem,xcolor]{mdframed}
\usepackage{parcolumns}
\usepackage{tabularx}
\usepackage{tikz}\pdfpageattr{/Group <</S /Transparency /I true /CS /DeviceRGB>>}
\usetikzlibrary{arrows,calc,intersections,matrix,positioning,through}
\usepackage{tikz-cd}
\tikzset{commutative diagrams/.cd,every label/.append style = {font = \normalsize}}
\usepackage[all]{xy}
\usepackage{rotating}
\usepackage[aligntableaux=center]{ytableau}
\usepackage{epstopdf}
\usepackage{diagbox}

\numberwithin{equation}{section}

\newtheorem{thm}[equation]{Theorem}
\AfterEndEnvironment{thm}{\noindent\ignorespaces}
\newtheorem{theorem}[equation]{Theorem}
\AfterEndEnvironment{theorem}{\noindent\ignorespaces}

\newtheorem{corollary}[equation]{Corollary}
\AfterEndEnvironment{cor}{\noindent\ignorespaces}
\newtheorem{lemma}[equation]{Lemma}
\AfterEndEnvironment{lemma}{\noindent\ignorespaces}
\newtheorem{lem}[equation]{Lemma}
\AfterEndEnvironment{lem}{\noindent\ignorespaces}
\newtheorem{proposition}[equation]{Proposition}
\AfterEndEnvironment{proposition}{\noindent\ignorespaces}

\AfterEndEnvironment{prop}{\noindent\ignorespaces}
\newtheorem{conj}[equation]{Conjecture}
\AfterEndEnvironment{conj}{\noindent\ignorespaces}

\AfterEndEnvironment{prob}{\noindent\ignorespaces}
\newtheorem{question}[equation]{Question}
\AfterEndEnvironment{question}{\noindent\ignorespaces}

\theoremstyle{definition}
\newtheorem{defn}[equation]{Definition}
\newtheorem{definition}[equation]{Definition}
\AfterEndEnvironment{defn}{\noindent\ignorespaces}
\AfterEndEnvironment{definition}{\noindent\ignorespaces}
\newtheorem*{pf_no_qed}{Proof}

\AfterEndEnvironment{pf}{\noindent\ignorespaces}
\newtheorem{eg_no_qed}[equation]{Example}

\AfterEndEnvironment{eg}{\noindent\ignorespaces}
\newenvironment{example}[1][]{\begin{eg_no_qed}[#1]\pushQED{\qed}}{\popQED\end{eg_no_qed}}
\AfterEndEnvironment{example}{\noindent\ignorespaces}
\newtheorem{rmk}[equation]{Remark}
\AfterEndEnvironment{rmk}{\noindent\ignorespaces}
\newtheorem{remark}[equation]{Remark}
\AfterEndEnvironment{remark}{\noindent\ignorespaces}

\theoremstyle{remark}

\AfterEndEnvironment{claim}{\noindent\ignorespaces}
\newtheorem*{claimpf_no_qed}{Proof of Claim}

\AfterEndEnvironment{claimpf}{\noindent\ignorespaces}

\newcommand{\Le}{\textup{\protect\scalebox{-1}[1]{L}}}

\font\pipefont=lcircle10
\def\elbow{\smash{\raise3pt\hbox{\pipefont\rlap{\rlap{\char'014}\char'016}}}}

\def\halfelbow{\smash{\raise2pt\hbox{\pipefont\rlap{\rlap{\rlap{\char'015}\phantom{\char'017}}}}}}
\def\cross{\smash{\lower5pt\hbox{\rlap{\vrule height16pt}}\raise3pt\hbox{\rlap{\hskip-8pt \vrule height0.4pt depth0pt width16pt}}}}

\ytableausetup{boxsize=16pt}

\def\O{\mathcal{O}}
\def\B{\mathcal{B}}

\def\PP{\mathbb{P}}
\DeclareMathOperator{\convex}{convex}
\DeclareMathOperator{\dual}{dual}
\DeclareMathOperator{\Dr}{Dr}

\DeclareMathOperator{\Mat}{Mat}

\newcommand{\Gr}{Gr_{k,n}^{\geq 0}}

\newcommand{\R}{\mathbb{R}}
\newcommand{\RR}{\mathbb{R}}
\newcommand{\Z}{\mathbb{Z}}

\newcommand{\x}{\mathbf{x}}

\newcommand{\U}{\widetilde{U_{k,n}}}
\DeclareMathOperator{\pre}{pre}
\DeclareMathOperator{\inc}{inc}
\newcommand{\ipre}{\mathfrak{i}_{\pre}}
\newcommand{\iinc}{\mathfrak{i}_{\inc}}
\newcommand{\Ipre}{\iota_{\pre}}
\newcommand{\Iinc}{\iota_{\inc}}

\newcommand{\rf}[1]{\hyperref[#1]{(\ref*{#1})}}

\DeclareMathOperator{\spn}{span}

\DeclareMathOperator{\Trop}{Trop}


\title[The positive tropical Grassmannian
and 
 the $m=2$ amplituhedron]{The positive tropical Grassmannian, the 
hypersimplex, and the 
$m=2$ amplituhedron}

\author{Tomasz {\L}ukowski}
\author{Matteo Parisi}
\author{Lauren K.\ Williams}
\address{}
\email{\href{mailto:t.lukowski@herts.ac.uk}{t.lukowski@herts.ac.uk}}
\email{\href{mailto:matteo.parisi@maths.ox.ac.uk}{matteo.parisi@maths.ox.ac.uk}}
\email{\href{mailto:williams@math.harvard.edu}{williams@math.harvard.edu}}

\begin{document}

\begin{abstract}
The study of the moment map 
from the Grassmannian to the hypersimplex,
and the relation between torus orbits and matroid polytopes,
dates back to the foundational 1987 
 work of Gelfand-Goresky-MacPherson-Serganova \cite{GGMS}.
On the other hand, the \emph{amplituhedron} is a very new object,
defined by Arkani-Hamed--Trnka  \cite{arkani-hamed_trnka} in connection 
with scattering amplitudes in $\mathcal{N}=4$ super Yang Mills theory.
In this paper we discover a striking duality  between
the {moment map} $\mu:Gr^{\geq 0}_{k+1,n} \to \Delta_{k+1,n}$ from 
	the positive Grassmannian $Gr^{\geq 0}_{k+1,n}$ 
to the hypersimplex,
and the {amplituhedron map} 
$\tilde{Z}: Gr^{\geq 0}_{k,n} \to \mathcal{A}_{n,k,2}(Z)$ from 
	$Gr^{\geq 0}_{k,n}$ 
to the $m=2$
amplituhedron.  We consider the 
\emph{positroid dissections} of both objects, which informally,
are subdivisions of $\Delta_{k+1,n}$ (respectively, 
$\mathcal{A}_{n,k,2}(Z)$) into a disjoint union 
of images of  positroid cells of the positive Grassmannian.
At first glance,
$\Delta_{k+1,n}$ and $\mathcal{A}_{n,k,2}(Z)$ seem very different -- the former is an $(n-1)$-dimensional polytope, while 
the latter is a $2k$-dimensional non-polytopal subset of $Gr_{k,k+2}$. 
Nevertheless,  
we conjecture that positroid dissections of $\Delta_{k+1,n}$
are in bijection with positroid dissections of $\mathcal{A}_{n,k,2}(Z)$ via a map we call \emph{T-duality}. 
We prove this conjecture for the (infinite) class of 
\emph{BCFW dissections} and give additional experimental evidence. 
Moreover, we prove that the \emph{positive tropical Grassmannian}
is the secondary fan for the regular positroid subdivisions
of the hypersimplex, and propose that it also 
controls the T-dual positroid subdivisions of the amplituhedron.
Along the way, we 
	prove that a matroid 
polytope is a positroid polytope if and only if all two-dimensional faces
are positroid polytopes.  Towards the goal of generalizing T-duality
for higher $m$, we also define the \emph{momentum
amplituhedron} for any even $m$.
\end{abstract}

\maketitle
\setcounter{tocdepth}{1}
\tableofcontents

\section{Introduction}\label{sec_intro}

In 1987, the foundational work of Gelfand-Goresky-MacPherson-Serganova \cite{GGMS}
initiated the study of the Grassmannian and torus orbits in the Grassmannian via the 
\emph{moment map} and \emph{matroid polytopes}, 
which arise as moment map images of (closures of) 
torus orbits.  Classifying points of the Grassmannian based on the moment map images
of the corresponding torus orbits leads naturally to the \emph{matroid stratification} of the 
Grassmannian.
The moment map image of the entire Grassmannian $Gr_{k+1,n}$ is the $(n-1)$-dimensional
hypersimplex
$\Delta_{k+1,n} \subseteq \R^n$, the convex hull of the indicator vectors $e_I\in \R^n$ where $I \in {[n] \choose k+1}$.
Over the last decades 
there has been a great deal of work on 
matroid subdivisions of the hypersimplex
\cite{Kapranov, Lafforgue, Speyer}; these are closely connected to the 
\emph{tropical Grassmannian} \cite{tropgrass, Speyer, Dressian} 
and the \emph{Dressian} \cite{Dressian}, 
which parametrizes regular matroidal subdivisions of the hypersimplex.

The matroid stratification of the real Grassmannian is notoriously complicated: Mnev's universality
theorem says that the topology of the matroid strata can be as bad as that of any algebraic variety.
However, there is a subset of the Grassmannian called the \emph{totally nonnegative 
Grassmannian} or (informally) the \emph{positive Grassmannian} \cite{lusztig, postnikov}, where 
these difficulties disappear: the restriction of the matroid stratification to the positive 
Grassmannian gives a cell complex \cite{postnikov, rietsch, PSW}, whose cells $S_{\pi}$
are called \emph{positroid cells} and labelled by (among other things) \emph{decorated permutations}.
Since the work of Postnikov \cite{postnikov}, there has been an extensive study of \emph{positroids} \cite{Suho, ARW, ARW2} -- 
the matroids associated to the positroid cells.  The moment map images of 
positroid cells are precisely the positroid polytopes 
\cite{tsukerman_williams}, and as we will discuss in this paper, the 
\emph{positive tropical Grassmannian} \cite{troppos} 
(which equals the \emph{positive Dressian} \cite{SW})
parametrizes the regular positroid subdivisions of the hypersimplex.

Besides the moment map, there is another interesting map on the positive 
Grassmannian, which was recently introduced by Arkani-Hamed and Trnka \cite{arkani-hamed_trnka} 
in the context of \emph{scattering amplitudes} in $\mathcal{N}=4$ SYM.
In particular, any $n \times (k+m)$ 
matrix $Z$ with maximal minors positive induces
a map $\tilde{Z}$ from $Gr^{\geq 0}_{k,n}$ to the Grassmannian
$Gr_{k,k+m}$, whose image has full dimension $mk$ and is called
the \emph{amplituhedron} $A_{n,k,m}$ \cite{arkani-hamed_trnka}.
The case $m=4$ is most relevant to physics: in this case, the \emph{BCFW recurrence} (named for Britto, Cachazo, Feng, and Witten \cite{BCFW}) gives rise to
collections of $4k$-dimensional cells in $Gr^{\geq 0}_{k,n}$, whose images conjecturally 
\emph{tile} or \emph{triangulate} the amplituhedron.  

Given that the hypersimplex and the amplituhedron are images of the positive Grassmannian, 
which has a decomposition into positroid cells, one can ask the following questions.
When does a collection of positroid cells give -- via the moment map -- a \emph{positroid dissection} of the hypersimplex?
By \emph{dissection}, we mean that the images of these cells 
are disjoint and cover a dense subset of the hypersimplex (but we do not
put any constraints on how their boundaries match up).
When does a collection of positroid cells give -- via the $\widetilde{Z}$-map -- a dissection of the amplituhedron?
We can also ask about \emph{positroid tilings}, which are dissections 
coming from cells on which 
the moment map (respectively, the $\widetilde{Z}$-map) is injective.

The combinatorics of positroid tilings for both the hypersimplex and the amplituhedron
is very interesting: 
Speyer's $f$-vector theorem \cite{Speyer, Ktheory} gives an upper bound
on the number of 
matroid polytopes of each dimension in a matroidal subdivision coming
from the tropical Grassmannian.
In particular, it 
says that the number of 
top-dimensional matroid polytopes in such a subdivision of 
$\Delta_{k+1,n}$ is at most 
${n-2 \choose k}$.  
This number is in particular achieved by finest positroid subdivisions \cite{SW}.
Meanwhile, the third author together with Karp and Zhang 
\cite{Karp:2017ouj} conjectured 
that the number of cells in a tiling of the amplituhedron
 $\mathcal{A}_{n,k,m}(Z)$ for even $m$ is precisely 
$M(k, n-k-m, \frac{m}{2})$, where 
$$M(a,b,c) := \prod_{i=1}^a\prod_{j=1}^b\prod_{k=1}^c\frac{i+j+k-1}{i+j+k-2}
$$
 is the number of plane partitions contained in an $a \times b \times c$ box.
Note that when $m=2$, this conjecture says that the number of cells in a tiling of 
$\mathcal{A}_{n,k,2}(Z)$ equals ${n-2 \choose k}$.

What we show in this paper is that the appearance of the number ${n-2 \choose k}$ 
in the context of both the hypersimplex $\Delta_{k+1,n}$ and the amplituhedron $\mathcal{A}_{n,k,2}(Z)$
is not a coincidence!  Indeed, we can obtain tilings of the amplituhedron from 
tilings of the hypersimplex, by applying a \emph{T-duality map}.
This T-duality map sends loopless positroid cells $S_{\pi}$ of $Gr^{\geq 0}_{k+1,n}$ to coloopless
positroid cells
$S_{\hat{\pi}}$ of $Gr^{\geq 0}_{k,n}$ via a simple operation on the decorated permutations, 
see 
\cref{T-duality}.
T-duality sends \emph{tiles} for the hypersimplex (cells
where the moment map is injective) to tiles for the amplituhedron (cells
where $\widetilde{Z}$ is injective), see 
\cref{prop:injective}, and moreover
it sends dissections of 
the hypersimplex to dissections of the amplituhedron, see 
\cref{thm:shift}
and \cref{conj:dis}.
This explains the two appearances of the number ${n-2 \choose k}$ on the two
sides of the story. 

The fact that dissections of $\Delta_{k+1,n}$ and $\mathcal{A}_{n,k,2}(Z)$
 are in bijection is a rather surprising statement.
 Should there be a map from $\Delta_{k+1,n}$ to $\mathcal{A}_{n,k,2}(Z)$ or 
 vice-versa?
We have $\dim \Delta_{k+1,n} = n-1$ and $\dim \mathcal{A}_{n,k,2}(Z) = 2k$,
with no relation between $n-1$ and $2k$ (apart from $k \leq n$) so it is 
not obvious that a nice map between them should exist.  Nevertheless we do show
  that T-duality descends from a certain map
 that can be defined directly on positroid cells of $Gr^{\geq 0}_{k+1,m}$.

The T-duality map provides a handy tool for studying the amplituhedron $\mathcal{A}_{n,k,2}(Z)$: 
we can try to 
understand properties of the amplituhedron (and its dissections)
by studying the hypersimplex and applying T-duality.
For example, we show in \cref{sec:parity} that the rather mysterious 
\emph{parity duality}, which relates dissections of $\mathcal{A}_{n,k,2}(Z)$
with dissections of $\mathcal{A}_{n,n-k-2,2}$, can be obtained by composing
the hypersimplex duality $\Delta_{k+1,n} \simeq \Delta_{n-k-1,n}$ (which comes
from the Grassmannian duality $Gr_{k+1,n} \simeq Gr_{n-k-1,n}$) with T-duality on both sides.
As another example, we can try to obtain ``nice'' dissections of the amplituhedron from 
correspondingly nice dissections of the hypersimplex.
In general, dissections of $\Delta_{k+1,n}$ and $\mathcal{A}_{n,k,2}(Z)$ may have 
unpleasant properties, with images of cells intersecting badly at their boundaries, see \cref{sec:good}. 
However, the 
 \emph{regular subdivisions} of $\Delta_{k+1,n}$ are very nice polyhedral subdivisions.
By \cref{prop:positroid}, the 
regular positroid dissections of $\Delta_{k+1,n}$ come precisely from 
the positive Dressian $Dr^+_{k+1,n}$ (which equals the positive tropical Grassmannian 
$\Trop^+Gr_{k+1,n}$).  And moreover the images of these subdivisions 
under the T-duality map are very nice subdivisions of the amplituhedron $\mathcal{A}_{n,k,2}(Z)$, see 
\cref{sec:regularA}.   We speculate that $\Trop^+Gr_{k+1,n}$ plays the role of secondary fan
for the regular positroid subdivisions of $\mathcal{A}_{n,k,2}(Z)$, see
\cref{conj:fan}.

One step in proving \cref{prop:positroid} is the following new characterization of positroid
polytopes 
(see \cref{thm:char}):
a matroid polytope is a positroid polytope if and only if all of its two-dimensional faces are positroid polytopes.

Let us now explain how  
the various geometric objects in our story are related to scattering amplitudes in  
supersymmetric fields theories.
The main emphasis so far has been on the so-called ``planar limit'' of $\mathcal{N}=4$ super Yang-Mills.
In 2009, the works of Arkani-Hamed--Cachazo--Cheung--Kaplan \cite{ArkaniHamed:2009dn} and Bullimore--Mason--Skinner \cite{Bullimore:2009cb} introduced beautiful Grassmannian formulations for scattering amplitudes in this theory. Remarkably, this led to the discovery that the positive Grassmannian encodes most of the physical properties of amplitudes \cite{abcgpt}. 
Building on these developments and on Hodges' idea that scattering amplitudes might be `volumes' of some geometric object \cite{Hodges:2009hk}, Arkani-Hamed and Trnka arrived at the definition of the \emph{amplituhedron} $\mathcal{A}_{n,k,m}(Z)$ \cite{arkani-hamed_trnka} in 2013.

The $m=4$ amplituhedron $\mathcal{A}_{n,k,4}$ is the object most relevant
to physics: it encodes the geometry of (tree-level) scattering amplitudes in planar $\mathcal{N}=4$ SYM.  However, the amplituhedron is a well-defined and interesting mathematical object for any $m$. 
For example, the $m=1$ amplituhedron $\mathcal{A}_{n,k,1}$ can be identified with the complex of bounded faces of a cyclic hyperplane arrangement \cite{karpwilliams}.  
The $m=2$ amplituhedron $\mathcal{A}_{n,k,2}(Z)$, which is a main subject of this paper,
also has a beautiful combinatorial structure, and 
has been recently studied e.g. in \cite{Arkani-Hamed:2017vfh, Karp:2017ouj,BaoHe,Lukowski:2019sxw,Lukowski:2019kqi}.  
From the point of view of physics, $\mathcal{A}_{n,k,2}(Z)$ is often considered as a toy-model for the $m=4$ case.
However it has applications to physics as well:
 $\mathcal{A}_{n,2,2}$ governs the geometry of scattering amplitudes in $\mathcal{N}=4$ SYM at the subleading order in perturbation theory for the so-called `MHV' sector of the theory, and remarkably,
the $m=2$ amplituhedron $\mathcal{A}_{n,k,2}(Z)$ 
is also relevant for the `next to MHV' sector, enhancing its connection with the geometries of loop amplitudes \cite{Kojima:2020tjf}.

Meanwhile, in recent years physicists have been increasingly interested in understanding how \emph{cluster algebras} encode the analytic properties of scattering amplitudes, both at tree- and loop- level \cite{Golden:2013xva}. This led them to explore the connection between cluster algebras and the \emph{positive tropical Grassmannian} which was observed in \cite{troppos}.
In particular, the positive tropical Grassmannian has been increasingly playing a role in different areas of scattering amplitudes: from bootstrapping loop amplitudes in $\mathcal{N}=4$ SYM \cite{Drummond:2019cxm, Arkani-Hamed:2019rds,Henke:2019hve} to computing scattering amplitudes in certain scalar theories \cite{Cachazo:2019ngv}. 

Finally, physicists have already observed a duality between the formulations of scattering amplitudes $\mathcal{N}=4$ SYM in momentum space\footnote{More precisely, it is `spinor helicity' space, or, equivalently (related by half-Fourier transform), in twistor space. See \cite[Section 8]{abcgpt}.} and in momentum twistor space.
This is possible because of the so-called `Amplitude/Wilson loop duality' \cite{Alday:2008yw}, which was shown to arise from a more fundamental duality in String Theory called `T-duality' \cite{Berkovits:2008ic}. 
The geometric counterpart of this fact is a duality 
between collections of $4k$-dimensional `BCFW' cells of $Gr^{\geq 0}_{k,n}$ which 
(conjecturally) tile 
the amplituhedron $\mathcal{A}_{n,k,4}$, and corresponding collections of $(2n-4)$-dimensional cells of $Gr^{\geq 0}_{k+2,n}$ which (conjecturally) tile the \emph{momentum amplituhedron} $\mathcal{M}_{n,k,4}$;
the latter object
was introduced very recently by the first two authors together with Damgaard and Ferro
\cite{mamp}.
In this paper we see that this duality, which we have evocatively called \emph{T-duality}, extends beyond $m=4$. 
In particular, for $m=2$, the \emph{hypersimplex} $\Delta_{k+1,n}$ and the $m=2$ amplituhedron $\mathcal{A}_{n,k,2}(Z)$  are somehow 
dual to each other, a phenomenon that we explore and employ to study 
properties of both objects. 
We believe that this duality holds for any (even) $m$: 
in \cref{sec:twistor} we introduce
a generalization $\mathcal{M}_{n,k,m}$ of the momentum amplituhedron $\mathcal{M}_{n,k,4}$, and a 
corresponding
notion of \emph{T-duality}.

\textsc{Acknowledgements:}
All three authors would like to thank
the Harvard Center for Mathematical Sciences and Applications (CMSA)
for its hospitality, 
and the first and second authors would like to thank the organizers of the 
``Spacetime and Quantum Mechanics master class workshop" at the CMSA
for providing an excellent environment for discussions.  
Additionally we would like to thank 
Nick Early, Pasha Galashin, Felipe Rincon, 
Mario Sanchez,
 Melissa Sherman-Bennett, and David Speyer for useful comments.
The second author would like to acknowledge the support of the ERC grant number 724638, and the
third author would  like to acknowledge the support of the National Science Foundation
under agreements No.\ DMS-1854316 and No.\ DMS-1854512.  Any opinions,
findings and conclusions or recommendations expressed in this material
are those of the authors and do not necessarily reflect  the
views of the National Science Foundation.

\section{The positive Grassmannian, the hypersimplex, and the amplituhedron}

In this section we introduce the three main geometric objects in this paper:
the positive Grassmannian, the hypersimplex,
and the amplituhedron.  The latter two objects are images of the positive Grassmannian  under
the \emph{moment map} and the \emph{$\widetilde{Z}$-map}.

\begin{defn} The {\itshape (real) Grassmannian} $Gr_{k,n}$ (for $0\le k \le n$) is the space of all $k$-dimensional subspaces of $\R^n$.  An element of
$Gr_{k,n}$ can be viewed as a $k\times n$ matrix of rank $k$ modulo invertible row operations, whose rows give a basis for the $k$-dimensional subspace.
\end{defn}

Let $[n]$ denote $\{1,\dots,n\}$, and $\binom{[n]}{k}$ denote the set of all $k$-element subsets of $[n]$. Given $V\in Gr_{k,n}$ represented by a $k\times n$ matrix $A$, for $I\in \binom{[n]}{k}$ we let $p_I(V)$ be the $k\times k$ minor of $A$ using the columns $I$. The $p_I(V)$ do not depend on our choice of matrix $A$ (up to simultaneous rescaling by a nonzero constant), and are called the {\itshape Pl\"{u}cker coordinates} of $V$.

\subsection{The positive Grassmannian and its cells}
\begin{defn}[{\cite[Section~3]{postnikov}}]
	\label{def:positroid}
We say that $V\in Gr_{k,n}$ is {\itshape totally nonnegative} if $p_I(V)\ge 0$ for all $I\in\binom{[n]}{k}$.  
	The set of all totally nonnegative $V\in Gr_{k,n}$ is the {\it totally nonnegative Grassmannian} $Gr_{k,n}^{\geq 0}$; abusing notation, we will often refer to 
	$\Gr$ as the \emph{positive Grassmannian}.
	For $M\subseteq \binom{[n]}{k}$, let $S_{M}$ be
the set of $V\in Gr_{k,n}^{\geq 0}$ with the prescribed collection of Pl\"{u}cker coordinates strictly positive (i.e.\ $p_I(V)>0$ for all $I\in M$), and the remaining Pl\"{u}cker coordinates
equal to zero (i.e.\ $p_J(V)=0$ for all $J\in\binom{[n]}{k}\setminus M$). If $S_M\neq\emptyset$, we call $M$ a \emph{positroid} and $S_M$ its \emph{positroid cell}.
\end{defn}

Each positroid cell $S_{M}$ is indeed a topological cell \cite[Theorem 6.5]{postnikov}, and moreover, the positroid cells of $Gr_{k,n}^{\ge 0}$ glue together to form a CW complex \cite{PSW}.

As shown in \cite{postnikov}, the cells of $Gr_{k,n}^{\geq 0}$
are in bijection 
with various combinatorial objects, including 
\emph{decorated permutations} $\pi$ on $[n]$ with $k$ anti-excedances
and equivalence classes of \emph{reduced plabic graphs} $G$ of type $(k,n)$.
In \cref{app} we review these objects and give bijections between them.  This gives a canonical way to label each positroid by a decorated permutation and an equivalence class of plabic graphs; we will correspondingly refer to positroid cells as $S_{\pi}$, 
$S_G$, etc.

\subsection{The moment map and the hypersimplex} \label{def:hyper}

The \emph{moment map} from the Grassmannian $Gr_{k,n}$ to $\R^n$
is defined as follows.
\begin{definition}\label{def:moment1}
	Let $A$ be a $k \times n$ matrix representing a point of 
	$Gr_{k,n}$.
	The \emph{moment map}\footnote{We remark that 
	there is another version of the moment map called the 
	\emph{algebraic moment map}, which we will briefly discuss
	later, see \cref{def:moment2}.}
	$\mu: Gr_{k,n} \to \R^n$ is defined by 
	$$\mu(A) = \frac{ \sum_{I \in \binom{[n]}{k}} |p_I(A)|^2 e_I}
	{\sum_{I \in \binom{[n]}{k}} |p_I(A)|^2},$$
where $e_I := \sum_{i \in I} e_i \in \R^n$, and $\{e_1, \dotsc, e_n\}$ is the standard basis of $\RR^n$.
\end{definition}

It is well-known that the image of the Grassmannian $Gr_{k,n}$ under the moment
map is the \emph{$(k,n)$-hypersimplex} 
$\Delta_{k,n}$, 
which is the convex hull of the points $e_I$ where $I$ runs over $\binom{[n]}{k}$.
If one restricts the moment map to $\Gr$ then the image is again the hypersimplex
$\Delta_{k,n}$ 
\cite[Proposition 7.10]{tsukerman_williams}.

We will consider the restriction of the moment map to 
positroid cells of $\Gr$.  
\begin{definition}\label{def:Gamma}
Given a positroid cell $S_{\pi}$ of $Gr^{\geq 0}_{k,n}$, we let 
$\Gamma^{\circ}_{\pi} = \mu(S_{\pi})$, and 
$\Gamma_{\pi} = \overline{\mu(S_{\pi})}$.  
\end{definition}

There are a number of natural questions to ask.
What do the $\Gamma_{\pi}$ look like,
and how can one characterize them?
On which positroid cells is the moment
map injective?
The images $\Gamma_{\pi}$ of (closures of) positroid cells 
are called 
\emph{positroid polytopes}; we will explore their nice properties 
 in \cref{sec:polytopes}.

One of our main motivations is to understand \emph{positroid dissections} 
 of the hypersimplex.

\begin{definition}\label{def:dissection1}
	A \emph{positroid dissection} (or simply a \emph{dissection})
	of $\Delta_{k,n}$ 
is a collection 
 $\mathcal{C} = \{S_{\pi}\}$ of positroid cells of $\Gr$, such that:
	\begin{itemize}
		\item $\dim \Gamma_{\pi} = n-1$ for each $S_{\pi}\in \mathcal{C}$
		\item the images $\Gamma^{\circ}_{\pi}$ and $\Gamma^{\circ}_{\pi'}$
			of two distinct cells in the collection are disjoint 
		\item $\cup_{\pi} \Gamma_{\pi} = \Delta_{k,n}$, i.e. 
			the union of the images of the cells is dense in $\Delta_{k,n}$.
	\end{itemize}
	We say that a positroid dissection 
 $\mathcal{C} = \{S_{\pi}\}$ 
	of $\Delta_{k,n}$ is a \emph{positroid tiling}
	of $\Delta_{k,n}$ if $\mu$ is injective on each $S_{\pi} \in \mathcal{C}$.
\end{definition}

\begin{question}
Let $\mathcal{C} = \{S_{\pi}\}$ be a collection of positroid cells of $\Gr$.
	When is $\mathcal{C}$ a positroid dissection of $\Delta_{k,n}$?  When is it a positroid tiling?
\end{question}

\subsection{The $\widetilde{Z}$-map and the amplituhedron}

Building on \cite{abcgpt},
Arkani-Hamed and Trnka \cite{arkani-hamed_trnka} recently introduced a beautiful new
mathematical object called the \emph{(tree) amplituhedron}, which
is the image of the positive Grassmannian under a map $\widetilde{Z}$
induced by a positive matrix $Z$.

\begin{defn}\label{def:amp}
For $a \le b$, define $\Mat_{a,b}^{>0}$ as the set of real $a \times b$ matrices whose $a\times a$ minors are all positive.
Let $Z\in \Mat_{n,k+m}^{>0}$. The \emph{amplituhedron map} $\tilde{Z}:Gr_{k,n}^{\ge 0} \to Gr_{k,k+m}$ is defined by $\tilde{Z}(C):=CZ$, where $C$ is a $k \times n$ matrix representing an element of $Gr_{k,n}^{\geq 0}$, and $CZ$ is a $k \times (k+m)$ matrix representing an element of $Gr_{k,k+m}$. The \emph{amplituhedron} $\mathcal{A}_{n,k,m}(Z) \subseteq Gr_{k,k+m}$ is the image $\tilde{Z}(Gr_{k,n}^{\ge 0})$.
\end{defn}

In special cases the amplituhedron recovers familiar objects. If $Z$ is a square matrix, i.e.\
$k+m=n$, then $\mathcal{A}_{n,k,m}(Z)$ is isomorphic to
the positive Grassmannian. If $k=1$, then it follows from
\cite{Sturmfels} that $\mathcal{A}_{n,1,m}(Z)$ is a {\itshape cyclic polytope} in projective space $\mathbb{P}^m$.
If $m=1$, then $\mathcal{A}_{n,k,1}(Z)$ can be identified with the complex of 
bounded faces of a cyclic hyperplane arrangement \cite{karpwilliams}.

We will consider the restriction of the $\widetilde{Z}$-map to 
positroid cells of $\Gr$. 

\begin{definition}\label{def:Z}
Given a positroid cell $S_{\pi}$ of $Gr^{\geq 0}_{k,n}$, we let 
	$Z^{\circ}_{\pi} = \widetilde{Z}(S_{\pi})$, and 
	$Z_{\pi} = \overline{\widetilde{Z}(S_{\pi})}$. We refer to $Z^{\circ}_{\pi}$ and $Z_{\pi}$ as \emph{open Grasstopes} and \emph{Grasstopes} respectively.
\end{definition}

As in the case of the hypersimplex, one of our main motivations 
is to understand \emph{positroid dissections}  of the amplituhedron $\mathcal{A}_{n,k,m}(Z)$.

\begin{definition}\label{def:dissection2}
Let $\mathcal{C} = \{Z_{\pi}\}$ be a collection 
of Grasstopes, with $\{S_{\pi}\}$ a collection of positroid cells  of $\Gr$.  We say that 
$\mathcal{C}$ is 
a \emph{positroid dissection} of $\mathcal{A}_{n,k,m}(Z)$ if we have that:
\begin{itemize}
	\item $\dim {Z}_{\pi} = mk$ for each $Z_{\pi}\in \mathcal{C}$
	\item pairs of distinct open Grasstopes $Z^{\circ}_{\pi}$ and $Z^{\circ}_{\pi'}$  in the collection are disjoint 
	\item $\cup_{\pi} Z_{\pi} = \mathcal{A}_{n,k,m}(Z)$.
	\end{itemize}
	We say that a positroid dissection 
 $\mathcal{C} = \{Z_{\pi}\}$ 
	of $\mathcal{A}_{n,k,m}(Z)$ is a \emph{positroid tiling}
	(or simply a \emph{tiling})
	of $\mathcal{A}_{n,k,m}(Z)$ if $\widetilde{Z}$ is injective on each $S_{\pi}$.
\end{definition}

\begin{remark}
	Let $\mathcal{S}$ be an index set for 
	 cells of $\Gr$.  It is expected that  
if $Z$ and $Z'$ both lie in $ \Mat_{k+m,n}^{>0}$, then
	$\{Z_{\pi}\}_{\pi \in \mathcal{S}}$ is a positroid tiling
	(respectively, dissection)
	of $\mathcal{A}_{n,k,m}(Z)$ if and only if 
	$\{Z'_{\pi}\}_{\pi \in \mathcal{S}}$ is a positroid tiling
	(respectively, dissection)
	of $\mathcal{A}_{n,k,m}(Z')$.

	The results we prove
	in this paper will be independent of $Z$. 
\end{remark}

\begin{question}
	Let $\mathcal{C} = \{Z_{\pi}\}$ be a collection of Grasstopes, with $\{S_{\pi}\}$ positroid cells of $\Gr$.
	When is $\mathcal{C}$ a positroid dissection of $\mathcal{A}_{n,k,m}(Z)$?  When is it a positroid tiling?
\end{question}

In this paper we will primarily focus on the case $m=2$ (with the exception of \cref{sec:twistor},
where we give some generalizations of our results and conjectures to general even $m$).
(positroid) tilings of the amplituhedron
have been studied in \cite{arkani-hamed_trnka}, \cite{Ferro:2015grk},
 \cite{Arkani-Hamed:2017tmz}, 
  \cite{Karp:2017ouj}, 
  \cite{Galashin:2018fri}, 
  \cite{Ferro:2018vpf}.  Very recently 
  the paper \cite{BaoHe} constructed (with proof) many
  tilings of the $m=2$ amplituhedron.
 The $m=2$ amplituhedron has also been studied in 
 \cite{Arkani-Hamed:2017vfh} (which gave an alternative description of it in terms of 
sign patterns; see also \cite{Karp:2017ouj}), in 
\cite{Lukowski:2019kqi} (which described
the boundary stratification of the amplituhedron $\mathcal{A}_{n,k,2}(Z)$), 
and in \cite{Lukowski:2019sxw} (which discussed
its relation to cluster algebras).
Note that our notion of dissection above is the same as the notion of subdivision from 
\cite[Definition 7.1]{Galashin:2018fri}.  (However, we prefer the word ``dissection,''
as the word ``subdivision'' is often used to indicate that there
are constraints on how the boundaries match up.)

\section{Positroid polytopes and the moment map}\label{sec:polytopes}

In this section we study \emph{positroid polytopes}, which are images
of positroid cells of $\Gr$ under the moment map $\mu:\Gr\to \R^n$.  
We recall some of the known properties of matroid and positroid polytopes, 
we give a new
characterization of positroid polytopes
(see \cref{thm:char}),
 and we describe when the moment map is an injection on a positroid cell,
or equivalently, when the 
moment map restricts to a homeomorphism from the closure of a positroid cell
to the corresponding positroid polytope (see \cref{prop:homeo}
and \cref{prop:homeo2}). 

\subsection{Matroid polytopes}
The torus $T = \R^n$ acts on $Gr_{k,n}$ by scaling the columns of 
a matrix representative $A$.  (This is really an $(n-1)$-dimensional torus
since the Grassmannian is a projective variety.)  
We let $TA$ denote the orbit of $A$ under the action of $T$, and 
$\overline{TA}$ its closure.  It follows from classical work of 
Atiyah \cite{A:82} and Guillemin-Sternberg \cite{GS} that the image
$\mu(\overline{TA})$ is a convex polytope, whose vertices
are the images of the torus-fixed points, i.e. the vertices are the
points $e_I$ such that $p_I(A) \neq 0$.

This motivates the notion of \emph{matroid polytope}.  
Note that any full
rank $k\times n$ matrix $A$ gives rise to a matroid 
$M(A)=([n],\B)$, where $\B = \{I \in \binom{[n]}{k} \ \vert \ p_I(A) \neq 0\}$.

\begin{definition}\label{def:mpolytope}
Given a matroid $M=([n],\B)$, the (basis) \emph{matroid polytope} $\Gamma_M$ of $M$ is the convex hull of the indicator vectors of the bases of~$M$:
\[
\Gamma_M := \convex\{e_B \mid B \in \B\} \subset \RR^n.
\]
\end{definition}

The following elegant characterization of matroid polytopes
is due to Gelfand,
Goresky, MacPherson, and Serganova.

\begin{theorem}[\cite{GGMS}]\label{r:GS} Let $\B$ be a collection of subsets of $[n]$ and let
$\Gamma_{\B}:=\convex\{e_B \mid B \in \B\} \subset \RR^n$. Then $\B$ is the collection of bases of a matroid if and only if every edge of  $\Gamma_\B$ is a parallel translate of $e_i-e_j$ for some $i,j \in [n]$.
\end{theorem}

The dimension of a matroid polytope is determined by the number of 
connected components of the matroid.
Recall that a matroid which cannot be written as the direct sum
of two nonempty matroids is called \emph{connected}.

\begin{proposition}[\cite{Oxley}]
\label{prop:equiv}
Let $M$ be a matroid on $E$.  For two elements $a, b \in E$, we set
$a \sim b$ whenever there are bases $B_1, B_2$ of $M$ such that
	$B_2 = (B_1 - \{a\}) \cup \{b\}$.  The relation $\sim$ is an equivalence
relation, and the equivalence classes are precisely the connected components of $M$.
\end{proposition}

\begin{proposition}[\cite{Coxeter}]\label{prop:dim}
For any matroid,
 the dimension of its matroid polytope is 
$\dim \Gamma_M = n-c$, where $c$ is the number of connected components of 
$M$.  
\end{proposition}

We note that there is an inequality description of any matroid polytope.
\begin{proposition}[\cite{welsh}]
\label{r:inequalitiesmatroids}
Let $M = ([n], \B)$ be any matroid of rank $k$, and let $r_M:2^{[n]} \to \Z_{\geq 0}$ be its rank function. Then the matroid polytope $\Gamma_M$ can be described as
\[
\Gamma_M = \left\{ {\bf x} \in \RR^n \mid \sum_{i \in [n]} x_i = k, \, \sum_{i \in A} x_i \leq r_M(A) \, \text{ for all $A \subset [n]$} \right\}.
\]
\end{proposition}

\subsection{Positroid polytopes}
In this paper we are interested in \emph{positroids}; these are the
matroids $M(A)$ associated to $k\times n$ matrices $A$ with 
maximal minors all nonnegative.  

In \cref{def:mpolytope}, we 
defined the matroid polytope $\Gamma_M$ to be the convex hull of the indicator vectors of 
the bases of the matroid $M$. We can of course apply the same definition to 
any positroid $M$, obtaining the \emph{positroid polytope} $\Gamma_M$.  
On the other hand, in \cref{def:Gamma}, for each positroid cell $S_{\pi}$,
we defined $\Gamma_{\pi} = \overline{\mu(S_{\pi})}$ to be the closure of the image of 
the cell under the moment map.  Fortunately these two objects coincide.

\begin{proposition}\cite[Proposition 7.10]{tsukerman_williams}\label{prop:moment}
Let $M$ be the positroid associated to the positroid cell $S_{\pi}$.  Then 
$\Gamma_{M} = \Gamma_{\pi} = \mu(\overline{S_{\pi}}) = \overline{\mu(S_{\pi})}$.
\end{proposition}

The first statement in \cref{thm:faces} below was proved in 
 \cite[Corollary 5.4]{ARW} 
 (and generalized to the setting of Coxeter 
 matroids in 
\cite[Theorem 7.13]{tsukerman_williams}.)
The second statement follows from the proof of \cite[Theorem 7.13]{tsukerman_williams}.
\begin{theorem}\label{thm:faces}
Every face of a positroid polytope is a positroid polytope.
	Moreover, every face $\Gamma_{\pi'}$ of a positroid polytope
	$\Gamma_{\pi}$ has the property that 
	$S_{\pi'} \subset \overline{S_{\pi}}$.
\end{theorem}

There is a simple inequality 
characterization
of positroid polytopes.

\begin{proposition}\label{prop:facets}\cite[Proposition 5.7]{ARW}
A matroid $M$ of rank $k$ on $[n]$ is a positroid if and only if its 
	matroid polytope $\Gamma_M$ can be described by the equality $x_1+ \dotsb + x_n = k$ and inequalities of the form
\[\sum_{\ell \in [i,j]} x_\ell \leq r_{ij}, \, \text{ with }i,j \in [n].\]
	Here $[i,j]$ is the cyclic interval given by 
	$[i,j]=\{i, i+1, \dots, j\}$ if $i<j$ and 
	$[i,j]=\{i, i+1,\dots,n,1, \dots, j\}$ if $i>j$.
\end{proposition}

We now give a new characterization of positroid polytopes. 
In what follows, we use $Sab$ as shorthand for $S \cup \{a,b\}$, etc.

\begin{theorem}\label{thm:char}
	Let $M$ be a matroid of rank $k$ on the ground set $[n]$, and 
	consider the matroid polytope $\Gamma_M$.  It is a positroid polytope
	(i.e. $M$ is a positroid) if and only if all of its two-dimensional
	faces are positroid polytopes.

	Moreover, if $M$ fails to be a positroid polytope, then $\Gamma_M$ 
	has a two-dimensional face $F$ with vertices
		$e_{Sab}, e_{Sad},
		e_{Sbc}, e_{Scd}$, for some $1 \leq a<b<c<d\leq n$ and $S$ 
		of size $k-2$ disjoint from $\{a,b,c,d\}$.
\end{theorem}

\begin{remark}
A different characterization of positroids in terms
of faces of their matroid polytopes
was given in \cite[Proposition 6.4]{RVY}, see also
\cite[Lemma 6.2 and Lemma 6.3]{RVY}.
There are also some related ideas in the proof of 
	\cite[Lemma 30]{Early}.
\end{remark}

By \cref{thm:faces}, every 
two-dimensional face of $\Gamma_M$ is a positroid
polytope.
To prove the other half of \cref{thm:char}, we use the following lemma.

\begin{lem}\label{lem:2}
	Let $M$ be a matroid of rank $k$
	on $[n]$ which has two connected components,
	i.e. $M = M_1 \oplus M_2$ such that the ground sets of $M_1$ and $M_2$
	are $S$ and $T = [n] \setminus S$.  Suppose that $\{S,T\}$ fails to be
	a noncrossing partition of $[n]$, in other words, there exists
	$a<b<c<d$ (in cyclic order) such that $a,c\in S$ and $b,d\in T$.
	Then $\Gamma_M$ has a two-dimensional face which is not a positroid
	polytope; in particular, that face is a square with vertices
		$e_{Sab}, e_{Sad},
		e_{Sbc}, e_{Scd}$, for some $1 \leq a<b<c<d\leq n$ and $S$ 
		of size $k-2$ disjoint from $\{a,b,c,d\}$.
\end{lem}

\begin{proof}
By \cref{prop:equiv}, we have bases $Aa$ and $Ac$ of $M_1$
	and also bases $Bb$ and $Bd$ of $M_2$.
	We can find a linear functional on $\Gamma_{M_1}$ given by a vector in $\R^S$
	 whose dot product is maximized on the convex hull of 
	 $e_{Aa}$ and 
	 $e_{Ac}$  (choose the vector $w$ such that $w_h = 1$ for $h\in A$,
	 $w_h = \frac{1}{2}$ for $h=a$ or $h=c$, and $w_h=0$ otherwise); therefore
	 there is an edge in $\Gamma_{M_1}$ between 
	 $e_{Aa}$ and 
	 $e_{Ac}$.
Similarly, 
	 there is an edge in $\Gamma_{M_2}$ between 
	 $e_{Bb}$ and 
	 $e_{Bd}$.
	 Therefore $\Gamma_M = \Gamma_{M_1} \times \Gamma_{M_2}$ has a two-dimensional
	 face whose vertices are 
		$e_{ABab}, e_{ABad},
		e_{ABbc}, e_{ABcd}$.
This is not a positroid polytope because 
	$\{ab, ad, bc, cd\}$ are not the bases of a rank 
	$2$ positroid.
\end{proof}

\begin{proposition}\label{prop:connected}
Let $M$ be a connected matroid.  If all of the two-dimensional
faces of $\Gamma_M$ are positroid polytopes,  then
	$\Gamma_M$ is a positroid polytope (i.e. $M$ is a positroid).
\end{proposition}

\begin{proof}
Suppose for the sake of contradiction that 
$\Gamma_M$ is not a positroid polytope.  
	
	Since $\Gamma_{M}$ is not a positroid polytope,
then by 
 \cref{r:inequalitiesmatroids}
and 	\cref{prop:facets},
	 it has a facet $F$ of the form
	$\sum_{i\in S} x_i = r_M(S)$, where $S$ is not a cyclic interval.
	In other words, $S$ and $T = [n]\setminus S$ fail to form a noncrossing partition.
Each facet of $\Gamma_M$ is the matroid polytope of a matroid with two
	connected components, so by the greedy algorithm for matroids (see e.g.
	\cite[Proposition 2.12]{ARW}), $F$ must be the matroid
	polytope of $M|S \oplus M/S$.  
	But now by \cref{lem:2}, $F$ has a two-dimensional face which is 
	not a positroid polytope.
\end{proof}

We now complete the proof
 of \cref{thm:char}.

\begin{proof}
We start by writing $M$ as a direct sum of connected matroids 
	$M = M_1 \oplus \dots \oplus M_{l}$.
	Let $S_1,\dots, S_l$ be the ground sets of 
	$M_1,\dots, M_l$.  By 
	\cite[Lemma 7.3]{ARW}, either one of the $M_i$'s fails to be a 
	positroid, or 
	$\{S_1,\dots,S_l\}$ fails to be a non-crossing 
	partition of $[n]$. 
        If one of the $M_i$'s fails to be a positroid, then by 
	\cref{prop:connected}, $\Gamma_{M_i}$
	has a two-dimensional face which fails to be a positroid.
	But then so does $\Gamma_M = \Gamma_{M_1} \times \dots \times \Gamma_{M_l}$.
On the other hand, if 
	$\{S_1,\dots,S_l\}$ fails to be a non-crossing 
	partition of $[n]$, then by 
\cref{lem:2}, $\Gamma_M$ has a two-dimensional face which fails to be a positroid.
This completes the proof.
\end{proof}

Our next goal is to use \cref{prop:dim} to  determine when the moment map restricted to a 
positroid cell is a homeomorphism.  To do so, 
we need to understand how to compute the number of connected components 
of a positroid.
The following result comes from 
\cite[Theorem 10.7]{ARW} and its proof.
We say that a permutation $\pi$ of $[n]$ is \emph{stabilized-interval-free 
(SIF)} if it does not stabilize any proper interval of $[n]$; that is,
$\pi(I) \neq I$ for all intervals $I \subsetneq [n]$.

\begin{proposition}\label{prop:conn}
Let $S_{\pi}$ be a positroid cell of $\Gr$ and 
let $M_{\pi}$ be the corresponding positroid.
	Then $M_{\pi}$ is connected if and only if 
	$\pi$ is a SIF 
	permutation of $[n]$.  
	More generally, the number of connected components  of 
	$M_{\pi}$ equals  the number of connected components
	of any reduced plabic graph associated to $\pi$.
\end{proposition}

\begin{example}
Consider the permutation 
$\pi = (5,3,4,2,6,7,1)$ (which in cycle notation is 
$(2 3 4) (1 5 6 7)$. 
Then there are two minimal-by-inclusion 
cyclic intervals such that $\pi(I)=I$,
namely $[2,4]$ and $[5,1]$, and hence 
the matroid $M_{\pi}$ has two connected components.
	(Note that $[1,7]$ is also a cyclic interval with $\pi([1,7]) = [1,7]$ but it is not minimal-by-inclusion.)
\end{example}

\begin{proposition}\label{prop:homeo}
Consider a positroid cell $S_{\pi} \subset \Gr$ and let $M_{\pi}$
be the corresponding positroid.  Then the following statements are 
equivalent:
\begin{enumerate}
	\item the moment map restricts to an injection on $S_{\pi}$
	\item 
the moment
map is a homeomorphism from 
$\overline{S_{\pi}}$ to $\Gamma_{{\pi}}$
		\item 
$\dim S_{\pi}
=\dim \Gamma_{{\pi}} = n-c,$ where $c$ is the number
of connected components of the matroid $M_{\pi}$.
	\end{enumerate}
	\end{proposition}
\begin{proof}
	Suppose that (1) holds, i.e. that 
	the moment map is an injection when 
restricted to a cell $S_{\pi}$.  Then 
$\dim \Gamma_{{\pi}} = \dim S_{\pi}$. 
By \cite[Proposition 7.12]{tsukerman_williams},
the positroid variety $X_{\pi}$ is a toric variety if and only if
$\dim \Gamma_{{\pi}} = \dim S_{\pi}$, 
so this implies that $X_{\pi}$ is a toric variety,
and $\overline{S_{\pi}}$ is its nonnegative part.
It is well-known that the moment map is a homeomorphism
when restricted to the nonnegative part of a toric variety
\cite[Section 4.2]{Fulton}, so it follows that 
	$\mu$ is a homeomorphism on $\overline{S_{\pi}}$.
	Therefore (1) implies (2).  But obviously (2) 
	implies (1).

	Now suppose that (2) holds.  Since $\Gamma_{\pi}$
	is the moment map image of $\overline{S_{\pi}}$,
	it follows that 
$\dim \Gamma_{{\pi}} = \dim S_{\pi}$, 
	and by \cref{prop:dim},
we have that 
	$\dim \Gamma_{{\pi}} = n-c,$ 
	where $c$ is the number of connected components of 
	the matroid $M_{\pi}$.  Therefore (2) implies (3).

	Now suppose  (3) holds.  
Then by \cite[Proposition 7.12]{tsukerman_williams},
	$X_{\pi}$ must be a toric variety, and so the 
	moment map restricts to a homeomorphism from 
	$\overline{S_{\pi}}$ to $\Gamma_{\pi}$.  So (3)
	implies (2).
\end{proof}

\begin{proposition}\label{prop:homeo2}
Consider a positroid cell $S_{\pi} \subset \Gr$ and let $M_{\pi}$
be the corresponding positroid.
Then the moment
map is a homeomorphism from 
$\overline{S_{\pi}}$ to $\Gamma_{{\pi}}\subset \R^n$ if and only if 
any reduced plabic graph associated to $\pi$ is a forest.
The $(n-1)$-dimensional cells $S_{\pi}$ on which 
the moment map is a homeomorphism to their image 
are precisely those cells whose reduced plabic graphs are trees.
\end{proposition}

\begin{proof}
This follows from \cref{prop:homeo} and 
\cref{prop:conn},
	together with the fact
that we can read off the dimension of a positroid cell
from any reduced plabic graph $G$ for it as the number of regions of 
$G$ minus $1$.
\end{proof}

\begin{remark}\label{rem:cycle}
	The connected $(n-1)$-dimensional positroid cells $S_{\pi}$
	of $\Gr$ are 
	precisely those $(n-1)$-dimensional cells where $\pi$ is 
	a single cycle of length $n$.
\end{remark}

As an alternative to the moment map from \cref{def:moment1}, 
we can also consider the 
 \emph{algebraic moment map} as in \cite{Sottile},
 defined as follows.\footnote{The reference \cite{Sottile}
 defines this map for toric varieties, but it makes sense for 
 $Gr_{k,n}$.}
\begin{definition}\label{def:moment2}
	Let $A$ be a $k \times n$ matrix representing a point of 
	$Gr_{k,n}$.
	The \emph{algebraic moment map} 
	$\tilde{\mu}: Gr_{k,n} \to \R^n$ is defined by 
	$$\tilde{\mu}(A) = \frac{ \sum_{I \in \binom{[n]}{k}} |p_I(A)| e_I}
	{\sum_{I \in \binom{[n]}{k}} |p_I(A)|}.$$
\end{definition}

\begin{lemma}\label{lem:alg}
\cref{prop:homeo} and 
\cref{prop:homeo2} hold verbatim after replacing \emph{moment map}
by \emph{algebraic moment map}.  In particular, if $S_{\pi}$
is a positroid cell whose reduced plabic graph is a tree,
then $\tilde{\mu}$ is an injection on $S_{\pi}$
and $\Gamma_{\pi} =\tilde{\mu}(\overline{S_{\pi}})$.
\end{lemma}
\begin{proof}
We note that both the moment map and the algebraic moment map
are homeomorphisms when restricted to the nonnegative part of 
a toric variety \cite[Theorem 8.5]{Sottile}, \cite[Section 4.2]{Fulton}. 
	Therefore the proofs of \cref{prop:homeo} and 
	\cref{prop:homeo2} hold when
we use the algebraic moment map.
\end{proof}

\begin{proposition}
We have 
	$\tilde{\mu}(Gr_{k,n}^{\geq 0}) = \Delta_{k,n}$.
\end{proposition}
\begin{proof}
It follows immediately from the definition that 
$\tilde{\mu}(A)$ will always be a convex combination 
of the points $e_I$ for $I\in {[n]\choose k}$ so 
 $\tilde{\mu}(Gr_{k,n}^{\geq 0}) \subseteq \Delta_{k,n}$.

 In the other direction, choose any positroid tiling
$\{S_{\pi}\}$ 
of $\Delta_{k,n}$, e.g. as in 
	\cref{prop:largeclass}.  Then by \cref{lem:alg}
	and the definition of positroid tiling,
	we have $\tilde{\mu}(\overline{S_{\pi}}) = \Gamma_{\pi}$
	and $\bigcup \Gamma_{\pi} = \Delta_{k,n}$.
	It follows that 
	$\tilde{\mu}(Gr_{k,n}^{\geq 0}) = \Delta_{k,n}$.
\end{proof}

\section{Dissecting the hypersimplex and the amplituhedron}

In this section we provide two 
recursive recipes for dissecting the hypersimplex 
$\Delta_{k+1,n}$, and dissecting the amplituhedron $\mathcal{A}_{n,k,2}(Z)$;
 the recipe for dissecting the $m=2$ amplituhedron  was proposed in \cite[Section 4.1]{Karp:2017ouj} and 
  proved in \cite{BaoHe}.  These recursive recipes are 
  completely parallel: as we will see in \cref{T-duality}, 
  the cells of corresponding dissections are in bijection with each 
  other via the \emph{T-duality map} on positroid cells.
Since these two recursions are analogous to the BCFW recurrence
(which conjecturally gives tilings of the $m=4$ amplituhedron),
 we refer to them as \emph{BCFW-style} recurrences.

\subsection{BCFW dissections of the hypersimplex}

\begin{definition}
Let $G$ (resp. $G'$) be a reduced plabic graph
with $n-1$ boundary vertices, associated to 
a positroid cell of 
$Gr^{\geq 0}_{k+1,n-1}$ (resp. $Gr^{\geq 0}_{k,n-1}$), which 
do not have a loop at vertex $n-1$.
We define $\ipre$ (resp. $\iinc$) to be the map which takes $G$ (resp. $G'$) and replaces the $(n-1)$st boundary vertex with a trivalent internal 
	white (resp. black) vertex attached to boundary vertices $n-1$ and $n$, as in the middle 
	(resp. rightmost) graph of 
\cref{fig:hyp}. 
\end{definition}

\begin{remark}
Using \cref{app}, it is straightforward to verify that both $\ipre(G)$ and $\iinc(G')$
	are reduced plabic graphs for cells of $Gr^{\geq 0}_{k+1,n}$. Moreover, we can in fact define
	$\ipre(G)$ (resp. $\iinc(G')$) on any reduced plabic graph for 
	$Gr^{\geq 0}_{k+1,n-1}$ (resp. $Gr^{\geq 0}_{k,n-1}$) which does not have a black (resp. white)
	lollipop at vertex $n-1$,
	and will again have that $\ipre(G)$ and $\iinc(G')$ represent cells of $Gr^{\geq 0}_{k+1,n}$.
\end{remark}

Abusing notation slightly, we also use $\ipre$ and $\iinc$ to denote the corresponding maps on 
positroid cells and polytopes, decorated permutations,  etc. 
Using \cref{def:rules},
 it is easy to determine the effect of $\ipre$ and $\iinc$ on decorated permutations.
We leave the proof of the following lemma as an exercise.

\begin{lemma}\label{lem:pre-inc}
If $\pi = (a_1, a_2,\dots,a_{n-1})$ is a decorated permutation such 
	that 
 $(n-1)\mapsto a_{n-1}$ is not a black fixed point, 
then $\ipre(\pi) = (a_1, a_2,\dots,a_{n-2},n,a_{n-1})$.

If $\pi = (a_1, a_2,\dots,a_{n-1})$ is a decorated permutation such that 
	$(n-1) \mapsto a_{n-1}$ is not a  white fixed point, 
then $\iinc(\pi) = 
(a_1,a_2,\dots, a_{j-1}, n, a_{j+1},\dots,a_{n-1},n-1)$ 
where $j =\pi^{-1}(n-1)$.
\end{lemma}

\begin{remark} 
\cref{lem:pre-inc} can be equivalently expressed in terms of 
 $\Le$-diagrams (see \cite{postnikov} or \cite[Section 2]{Karp:2017ouj}).
If $D$ is the $\Le$-diagram associated to $\pi$ as in the first paragraph
of \cref{lem:pre-inc}, then 
$\ipre(D)$ is obtained from $D$ by adding
a new column to the left of $D$,
where the new column consists of a single $+$ at the bottom. 
If $D$ is the $\Le$-diagram associated to $\pi$ as in the second
	paragraph of \cref{lem:pre-inc}, then 
$\iinc(D)$ is obtained from $D$ by adding
 a new row at the bottom of $D$,
 where the row consists of a single box containing a $+$.
\end{remark}

\begin{theorem}[BCFW recursion for the hypersimplex] \label{thm:dishyper}
	Let $\mathcal{C}_{k+1,n-1}$ (respectively $\mathcal{C}_{k,n-1}$)
	be a collection of positroid polytopes which dissects
	the hypersimplex $\Delta_{k+1,n-1}$ (resp. $\Delta_{k,n-1}$).  
	Then $$\mathcal{C}_{k+1,n} = \ipre(\mathcal{C}_{k+1,n-1}) \cup 
	    \iinc(\mathcal{C}_{k,n-1})$$ dissects $\Delta_{k+1,n}.$
\end{theorem}

We use the term \emph{BCFW dissection}  
(respectively, \emph{BCFW tiling}) to refer to any dissection or 
tiling 
that has the form $\mathcal{C}_{k,n}$ from \cref{thm:dishyper}.

Diagrammatically, \cref{thm:dishyper} is depicted in Fig.~\ref{fig:hyp}.
\begin{figure}[h]
\includegraphics[scale=0.32]{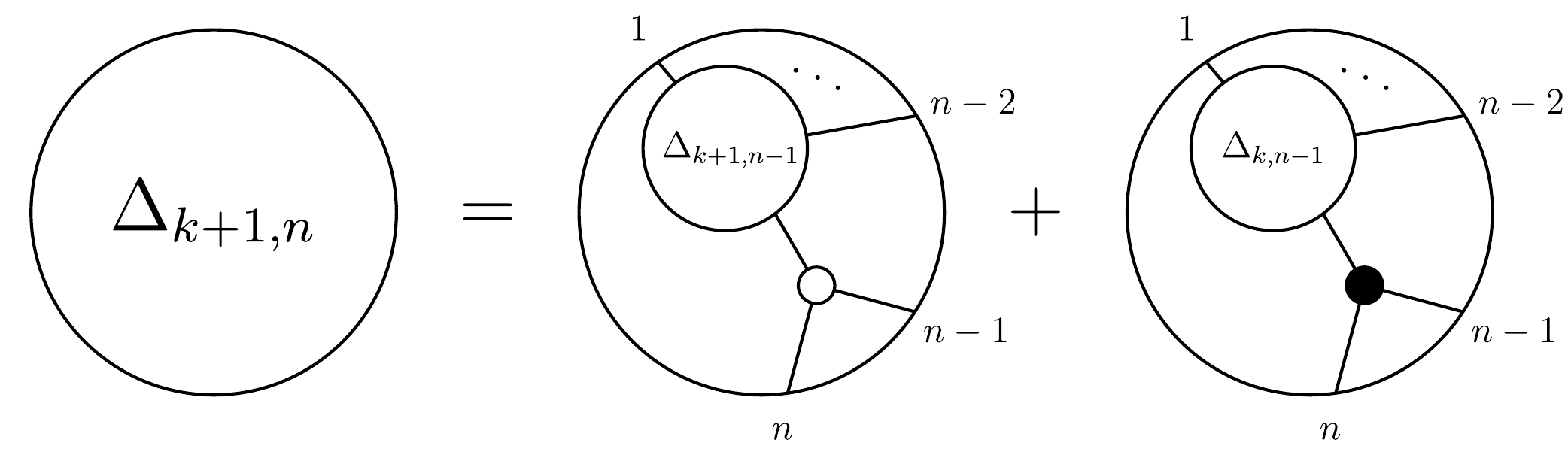}
	\caption{A BCFW-style 
	recursion for dissecting the hypersimplex.  There is 
	a parallel recursion obtained from this one by cyclically shifting
	all boundary vertices of the plabic graphs by $i$ (modulo $n$).}
	\label{fig:hyp}
\end{figure}

\begin{remark}\label{rem:cyc}
Because of the cyclic symmetry of 
the positive Grassmannian and the hypersimplex (see e.g. \cref{th:cyclicsymHS})
there are $n-1$ other versions of 
	\cref{thm:dishyper} (and \cref{fig:hyp}) in which all plabic graph labels
	get shifted by $i$  modulo $n$ (for $1 \leq i \leq n-1$).
\end{remark}

\begin{proof}
The hypersimplex $\Delta_{k+1,n}$ is cut out by the inequalities
$0 \leq x_i \leq 1$, as well as the equality $\sum_i x_i = k+1$.
We will show that \cref{fig:hyp} represents the partition of $\Delta_{k+1,n}$ into two pieces,
with the middle graph representing the piece cut out by $x_{n-1}+x_n \leq 1$,
and the rightmost graph representing the piece cut out by $x_{n-1}+x_n \geq 1$.

Towards this end, it follows from \cref{prop:perf} that if $G$ is a reduced plabic graph
	representing a cell of $Gr^{\geq 0}_{k+1,n-1}$, such that the positroid $M_G$
	has bases $\mathcal{B}$, then the bases of $M_{\ipre(G)}$
	are precisely $\mathcal{B} \sqcup\{(B\setminus \{n-1\}) 
	\cup \{n\}\ \vert \ B\in \mathcal{B}, n-1\in B\}$.
	In particular, each basis of $M_{\ipre(G)}$ may contain at most 
	one element of $\{n-1, n\}$.

	Meanwhile, it follows from \cref{prop:perf} that if $G$ is a reduced plabic graph
	representing a cell of $Gr^{\geq 0}_{k,n-1}$, such that the positroid $M_G$
	has bases $\mathcal{B}$, then the bases of $M_{\iinc(G)}$
	are precisely $\{B \cup \{n\} \ \vert \ B\in \mathcal{B}\} 
	\sqcup \{B \cup \{n-1\} 
	\ \vert \ B\in \mathcal{B}, n-1\notin B\}$. In particular, each basis of $M_{\iinc(G)}$
	must contain at least one element of $\{n-1,n\}$.

	It is now a straightforward exercise 
	(using e.g. \cite[Proposition 5.6]{ARW}) to determine that 
	if $\mathcal{C}_{k+1,n-1}$ 
	is a collection of cells in $Gr^{\geq 0}_{k+1,n-1}$  which dissects
	 $\Delta_{k+1,n-1}$ then 
	 $\ipre(\mathcal{C}_{k+1,n-1})$ dissects the 
	 subset of $\Delta_{k+1,n}$ cut out by the inequality
	 $x_{n-1}+x_n \leq 1$.
	 Similarly for $\iinc(\mathcal{C}_{k,n-1})$ and the 
	 subset of $\Delta_{k+1,n}$ cut out by 
	 $x_{n-1}+x_n \geq 1$.
\end{proof}

\begin{example}\label{ex:hyper}
Let $n=5$ and $k=2$.  
We will use \cref{thm:dishyper} to obtain a dissection of 
$\Delta_{k+1,n} = \Delta_{3,5}$.
We start with a dissection of $\Delta_{3,4}$ 
	coming from the plabic graph shown below (corresponding
	to the decorated permutation
	$(4,1,2,3)$),
	and a dissection of $\Delta_{2,4}$ (corresponding 
	to the permutations $(2,4,1,3)$ and $(3,1,4,2)$).
Applying the theorem leads to the three plabic graphs in the bottom
line, which correspond to the permutations 
	$(4,1,2,5,3), (2,5,1,3,4), (3,1,5,2,4)$.
\begin{center}
\begin{tabular}{c}
\raisebox{1.4cm}{$\Delta_{3,4}$:\,\,\,\,}$\includegraphics[scale=0.12]{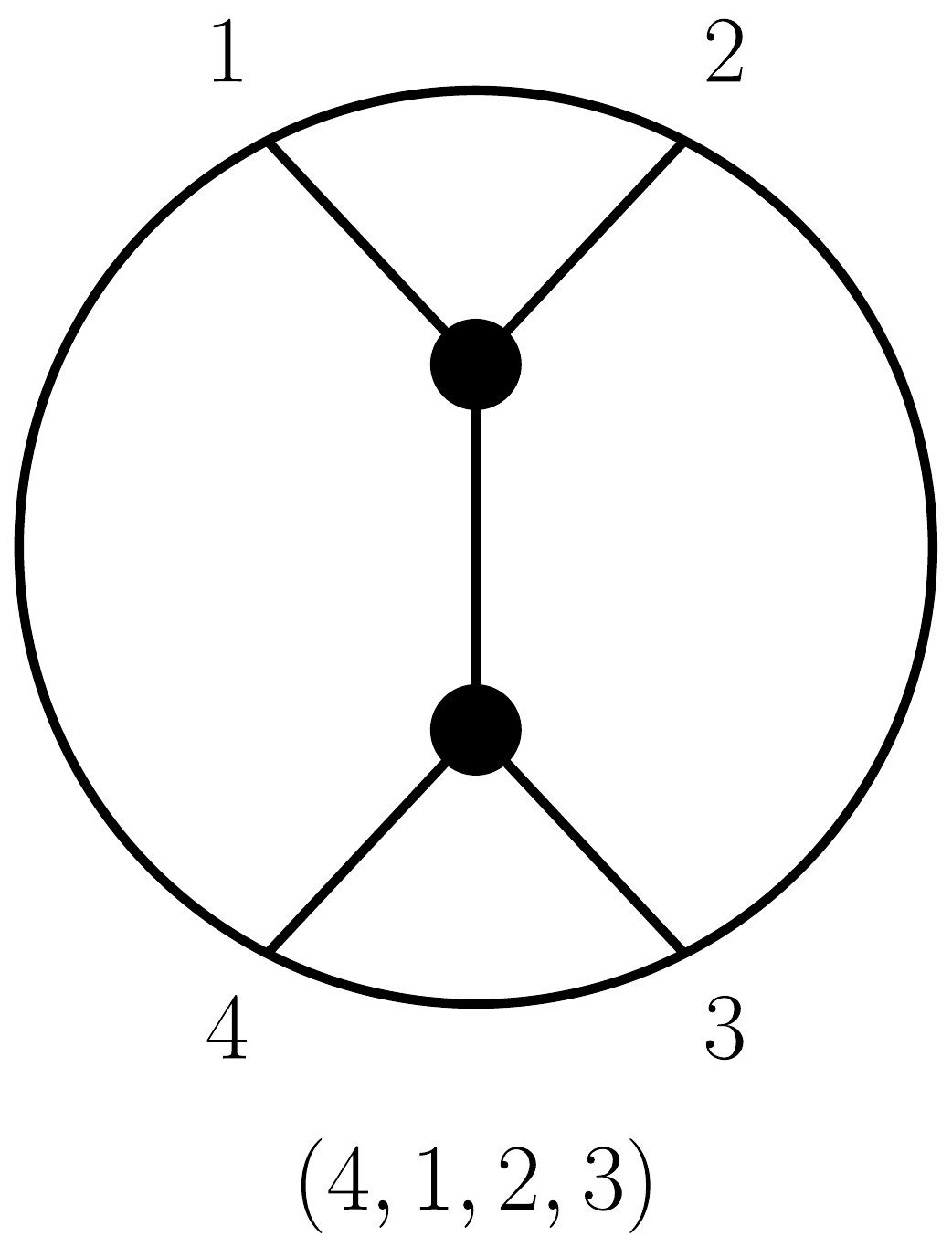}$\qquad
\raisebox{1.4cm}{ $\Delta_{2,4}$:\,\,\,\,}\includegraphics[scale=0.12]{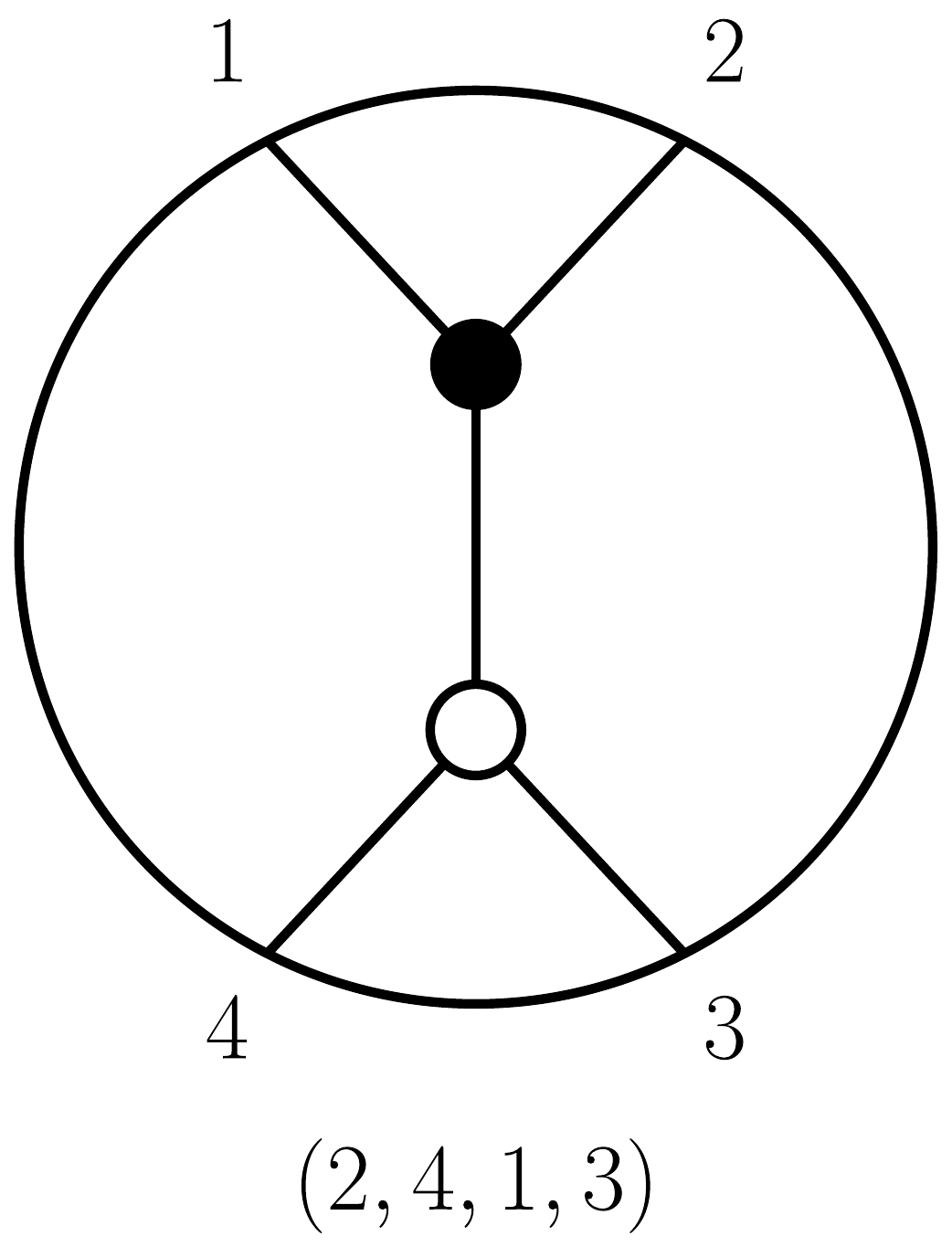}\quad\includegraphics[scale=0.12]{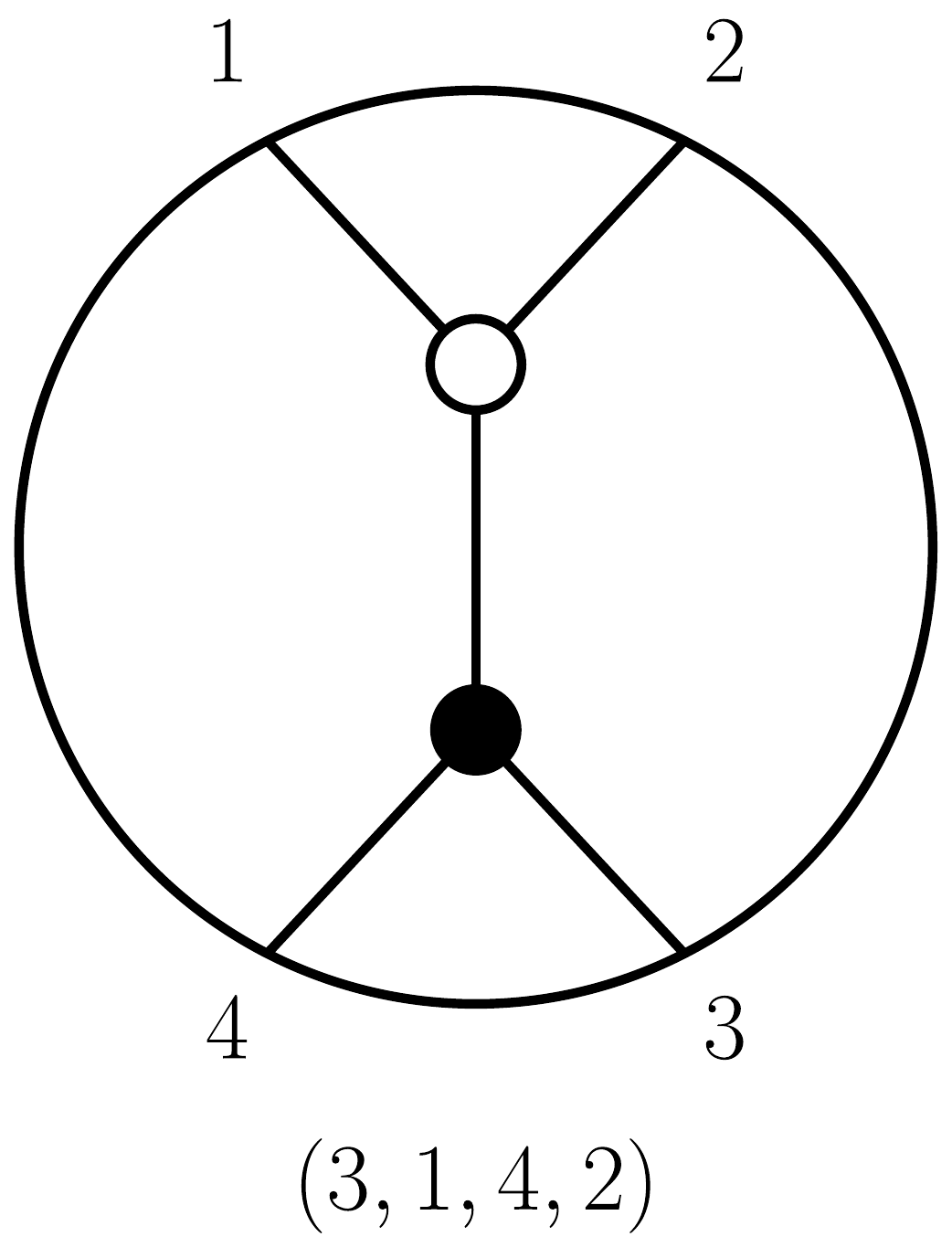}\\
\raisebox{1.4cm}{$\Delta_{3,5}$:}\,\,\,\,\,\includegraphics[scale=0.13]{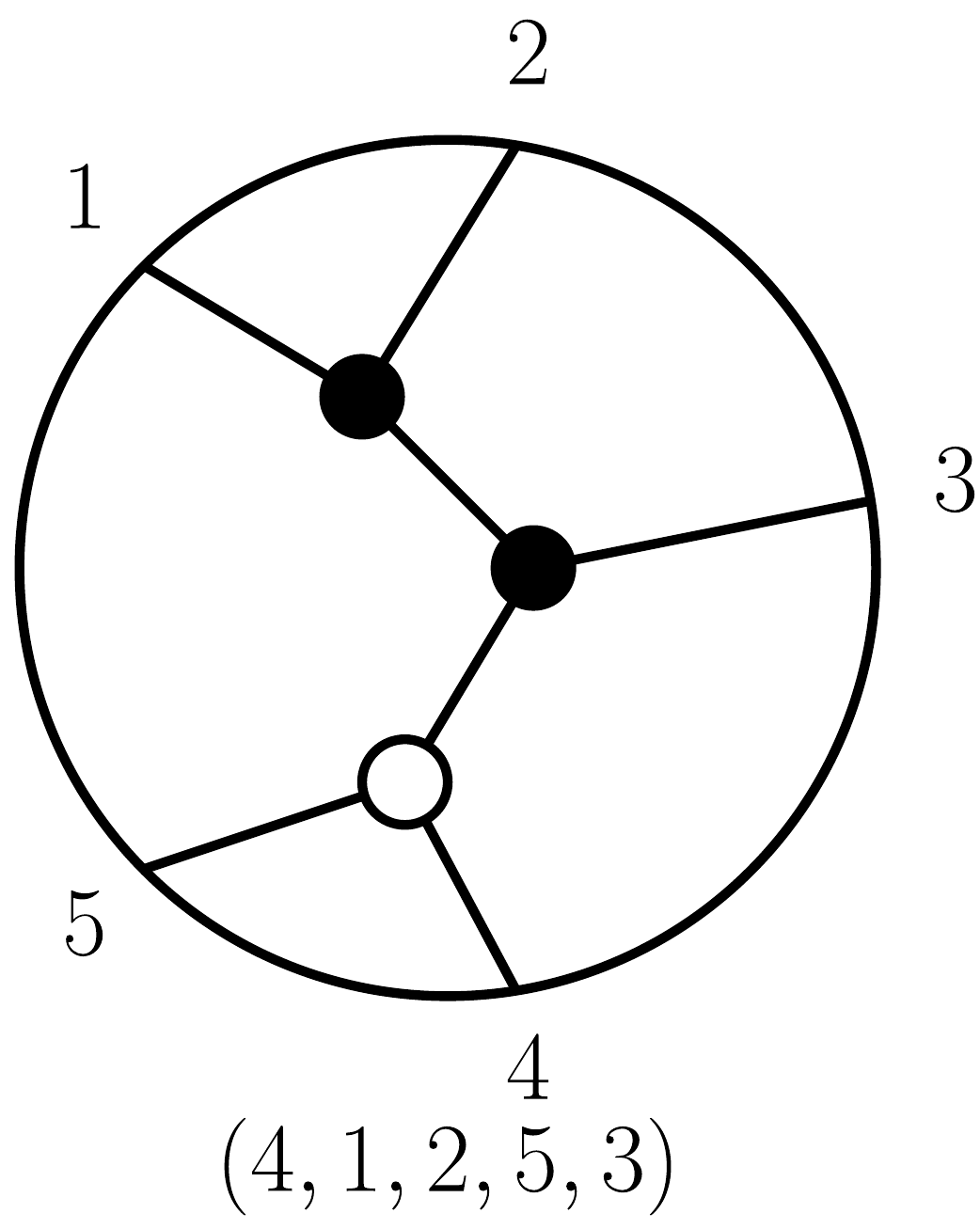}\quad\includegraphics[scale=0.13]{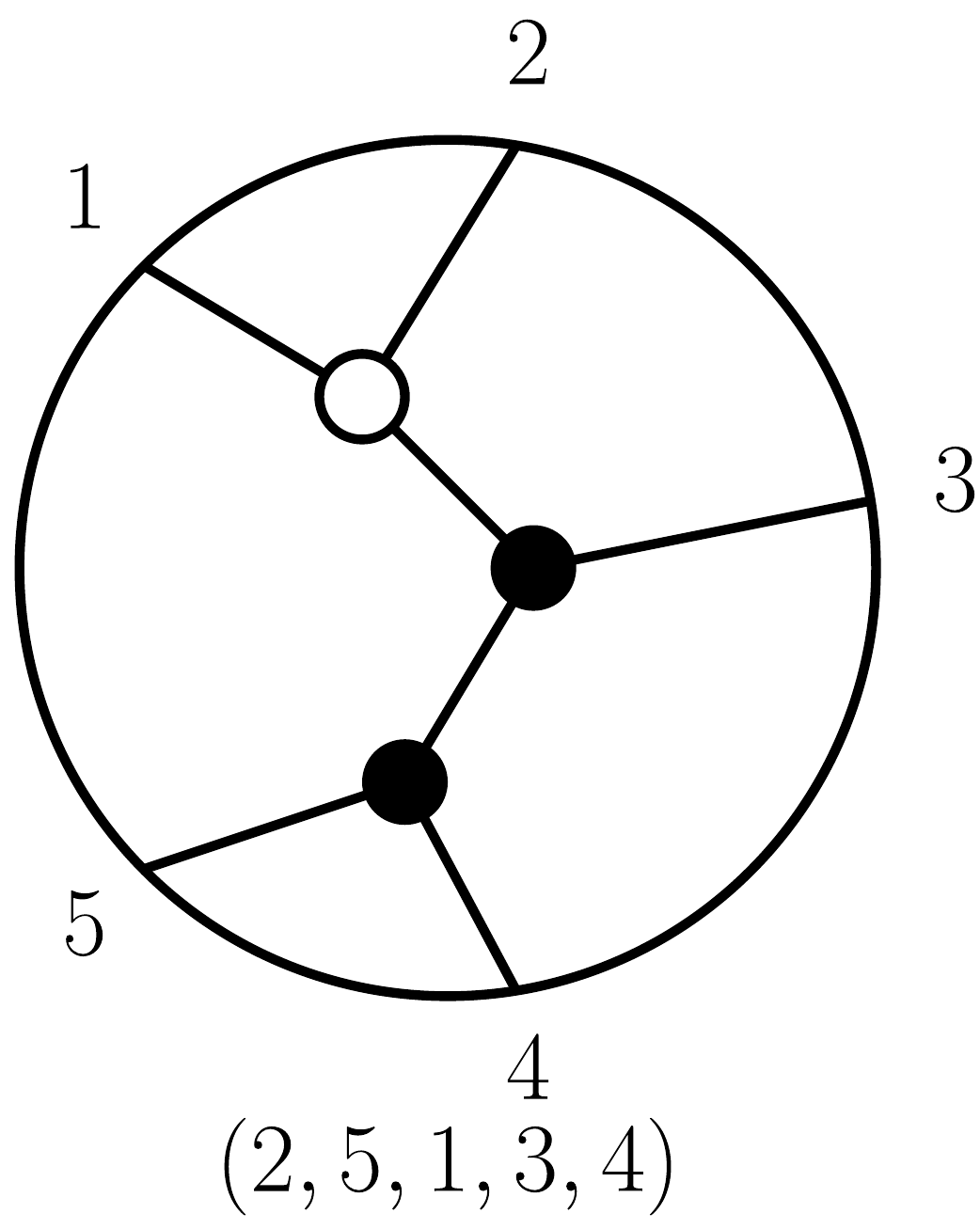}\quad\includegraphics[scale=0.13]{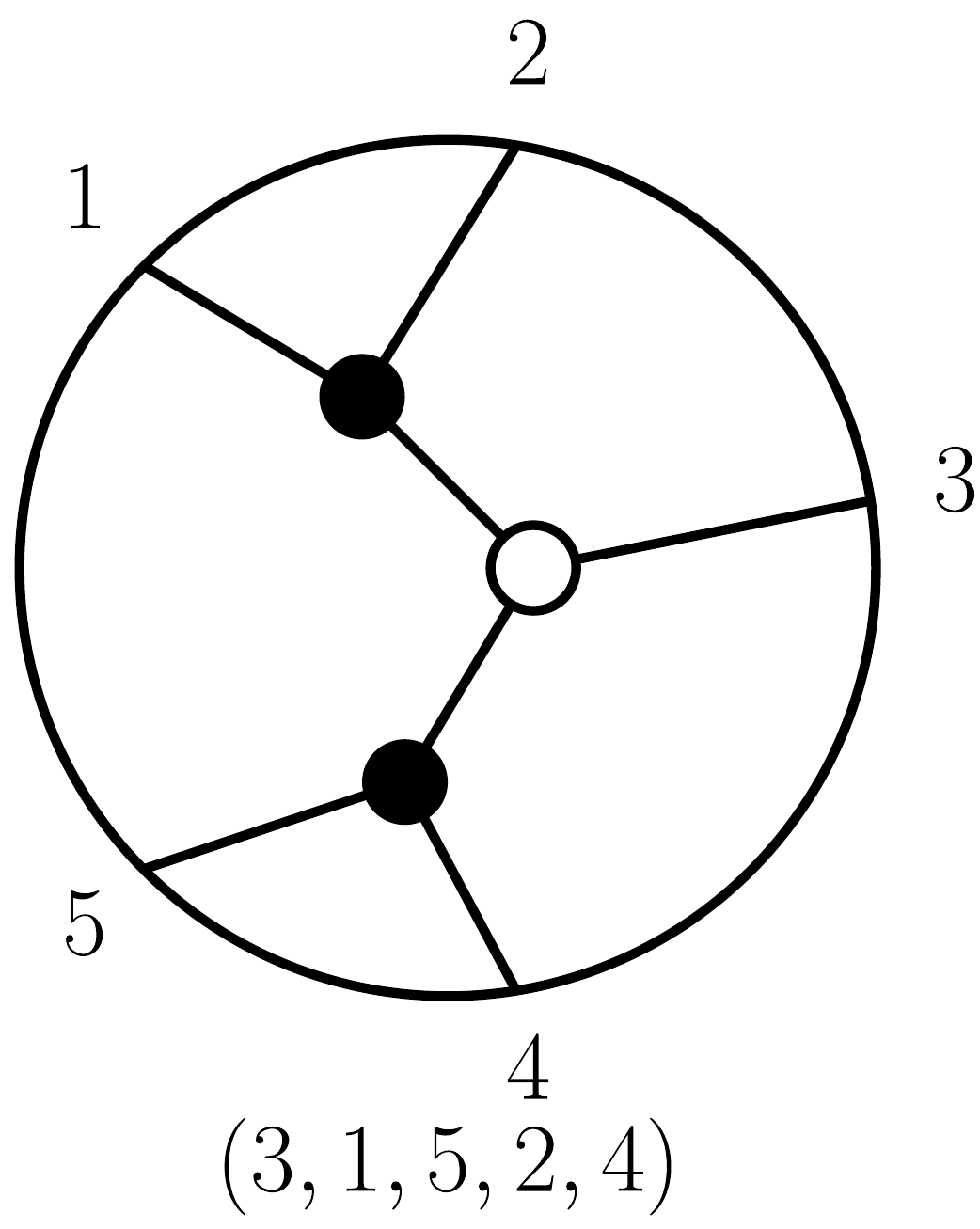}
\end{tabular}
\end{center}
\end{example}
\begin{remark}
It is worth pointing out that our BCFW-style recursion does not provide all possible dissections of the hypersimplex. This comes from the fact that in each step of the recursion we divide the hypersimplex into two pieces, while there are some dissections coming from 3-splits. The simplest example of a dissection which cannot be obtained from the recursion can be found already for $\Delta_{3,6}$ and is depicted in \cref{fig:hyp2}.

\begin{center}
\begin{figure}[h!]
\begin{tabular}{cccccc}
\includegraphics[scale=0.14]{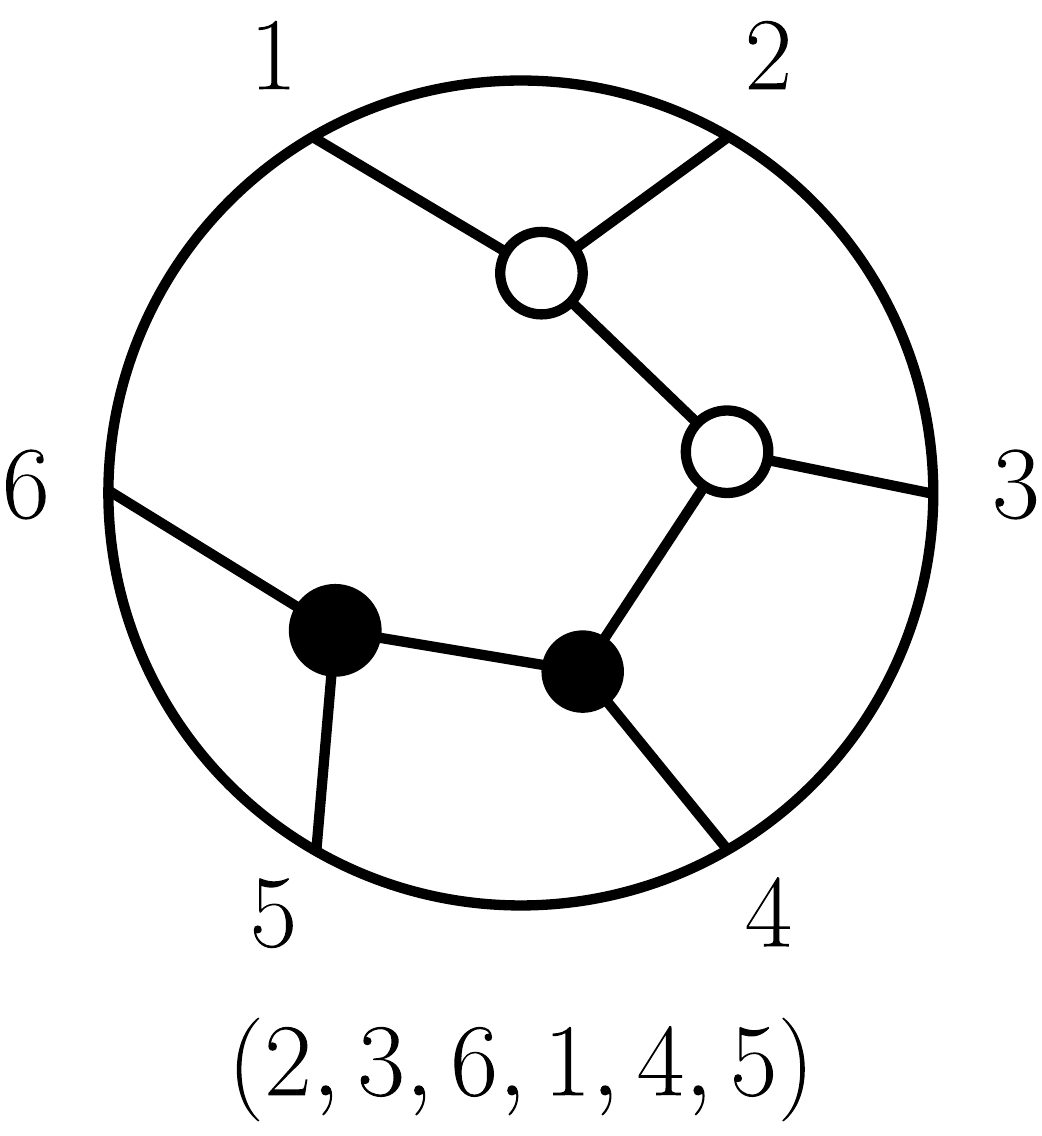}&
\includegraphics[scale=0.14]{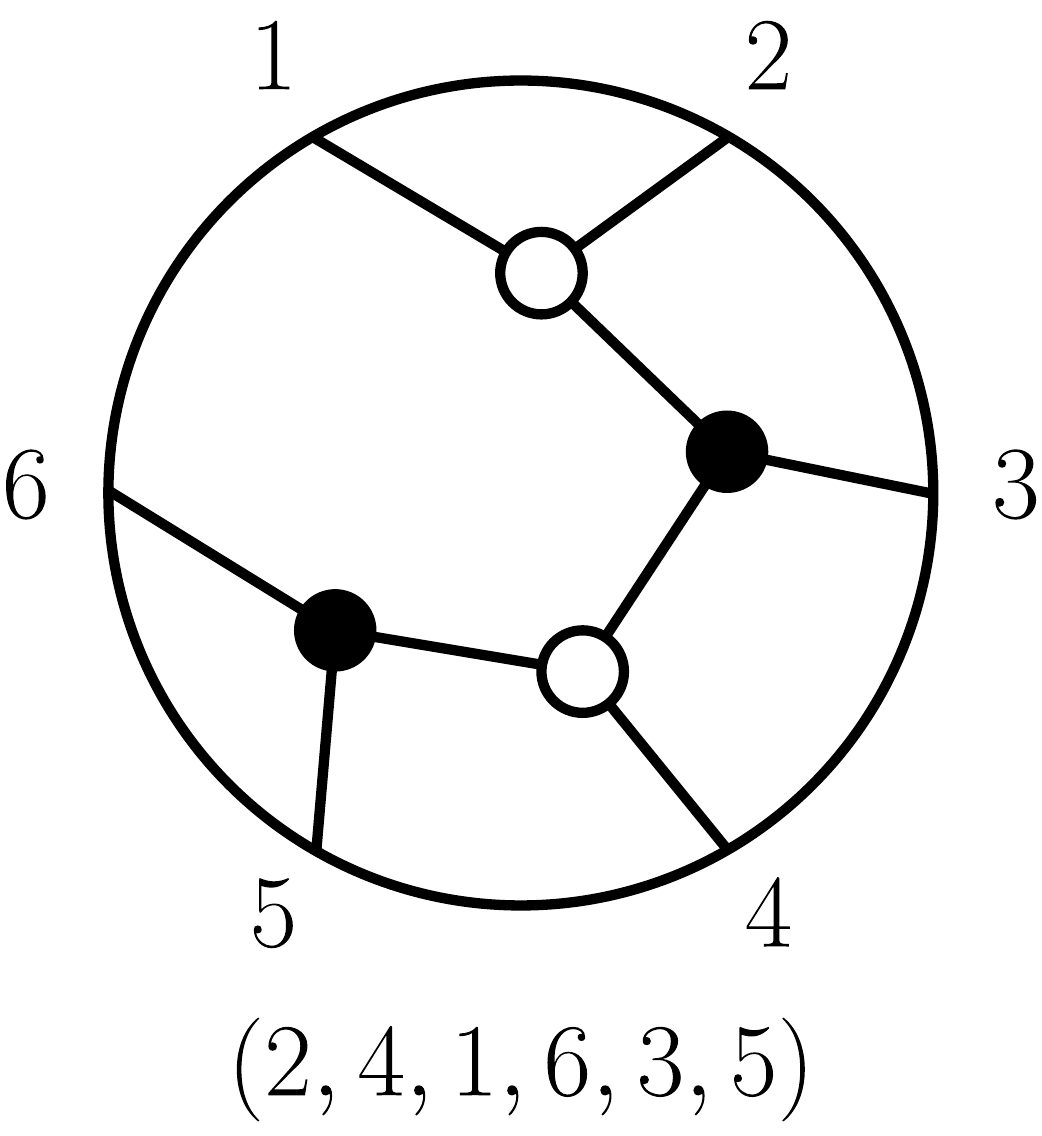}&
\includegraphics[scale=0.14]{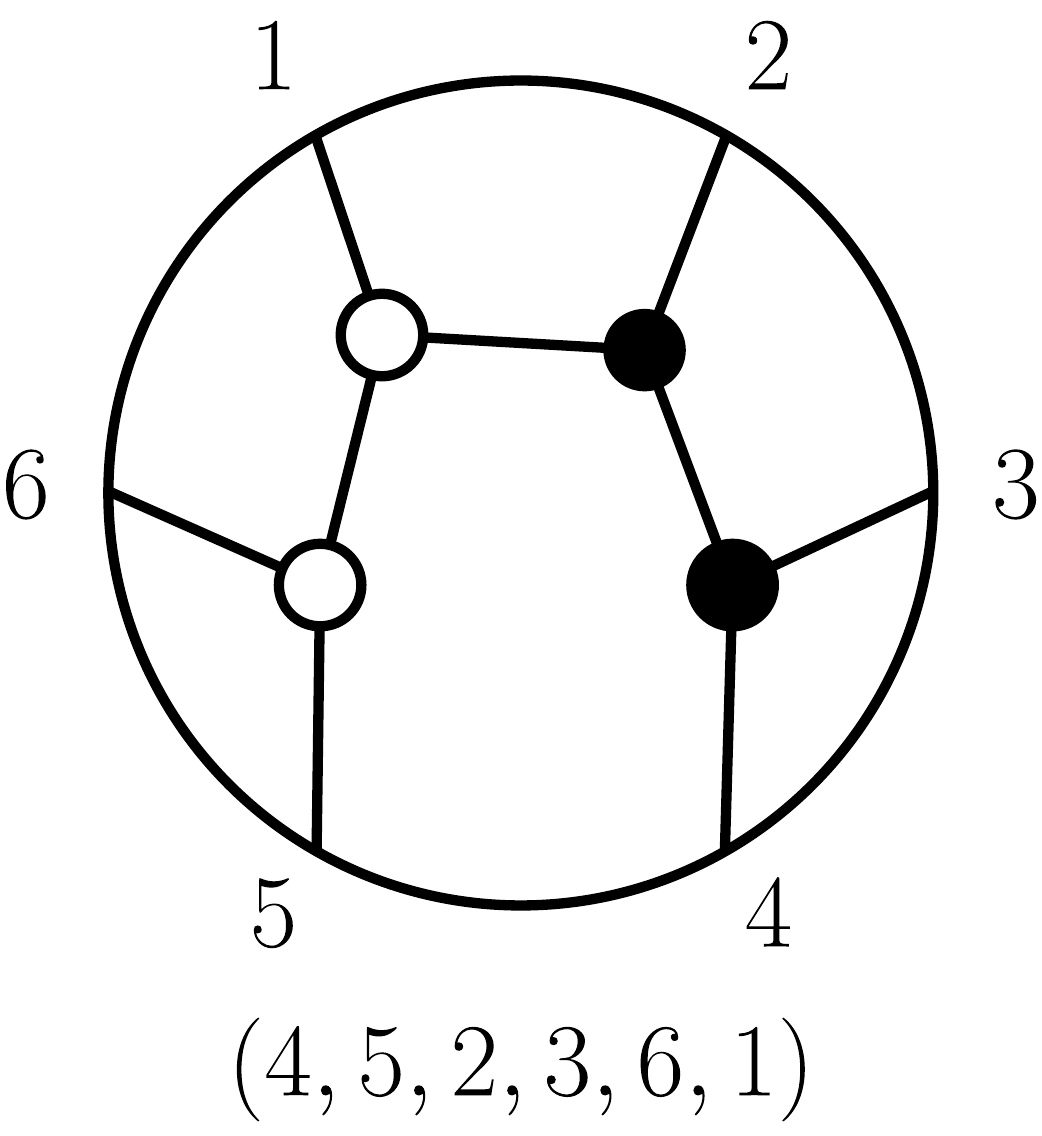}&
\includegraphics[scale=0.14]{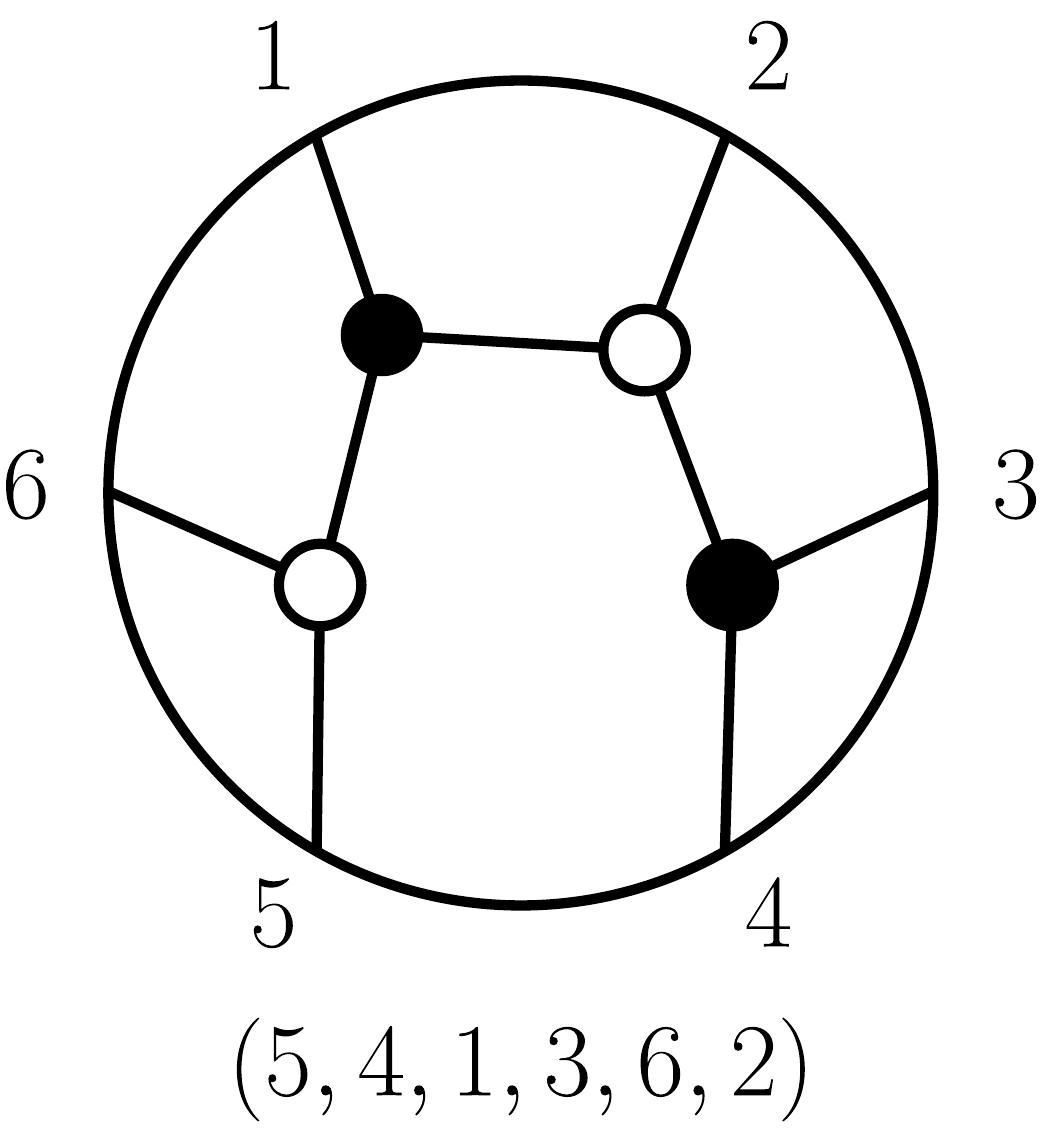}&
\includegraphics[scale=0.14]{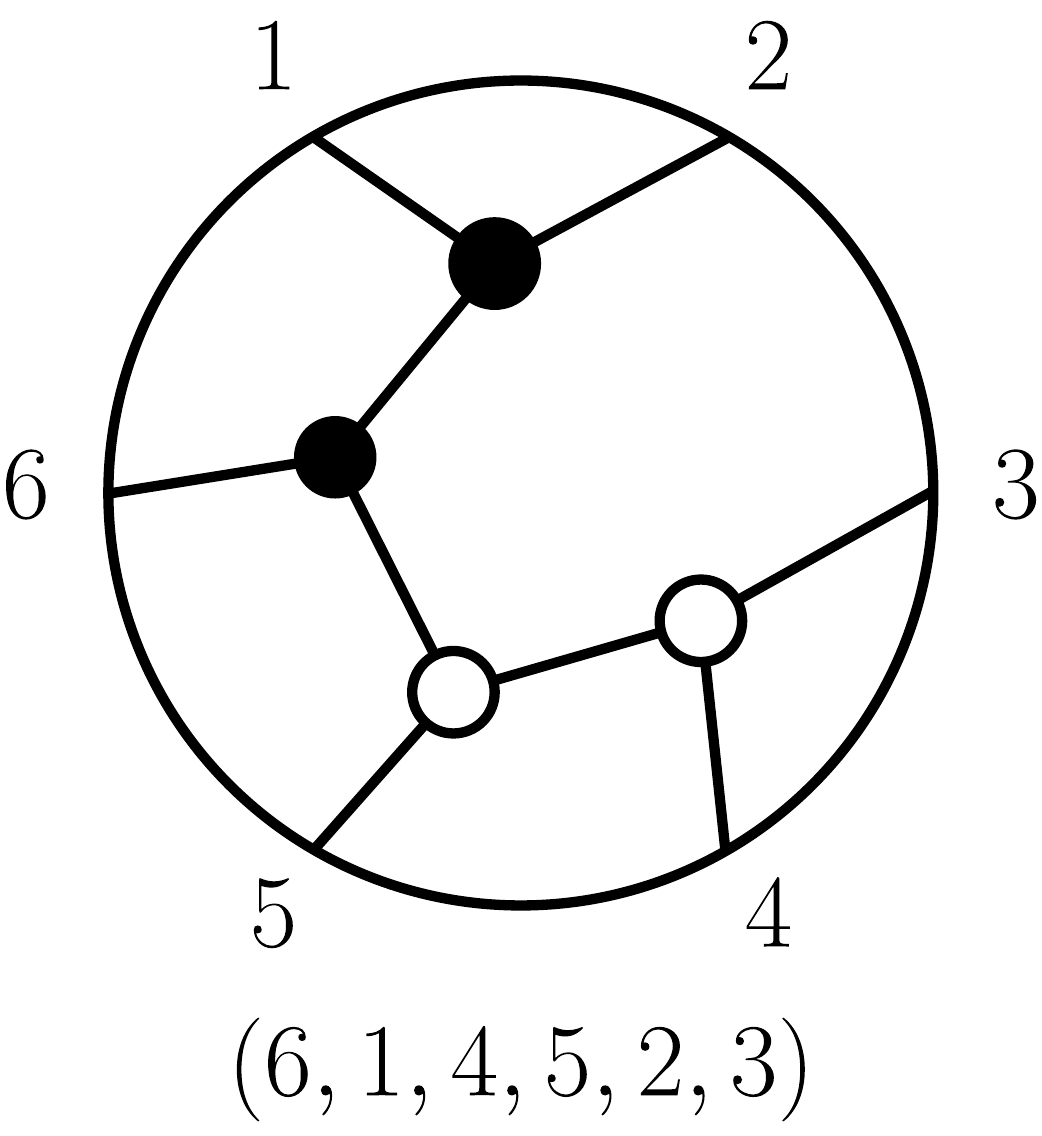}&
\includegraphics[scale=0.14]{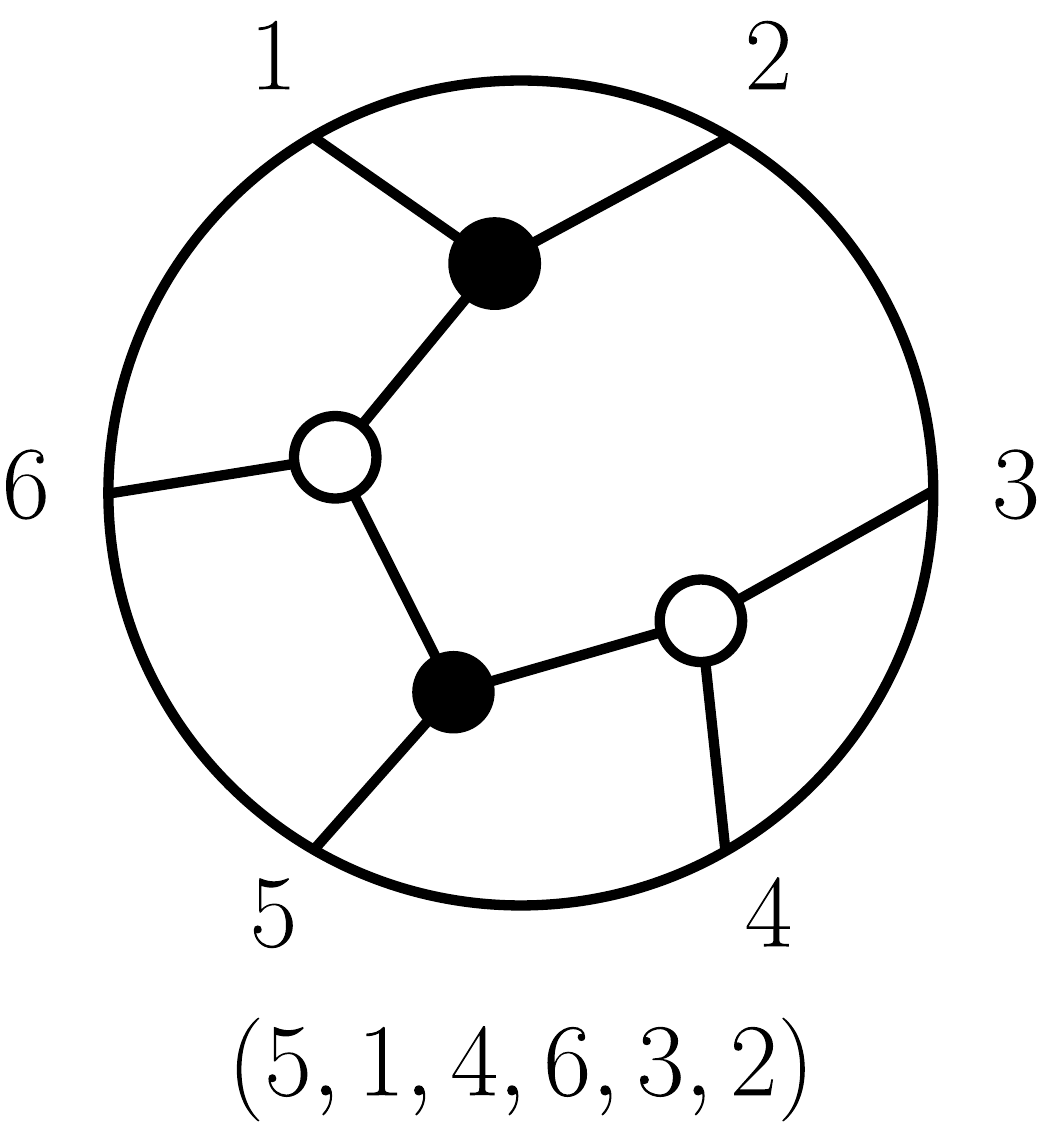}
\end{tabular}
\caption{An example of dissection of $\Delta_{3,6}$ that cannot be obtained from the BCFW-style recursion in \cref{thm:dishyper}.}
\label{fig:hyp2}
\end{figure}
\end{center}
\end{remark}

\subsection{BCFW dissections of the $m=2$ amplituhedron}

We now introduce some maps on plabic graphs, and 
recall a result of Bao and He \cite{BaoHe}.

\begin{definition}
Let $G$ be a reduced plabic graph
with $n-1$ boundary vertices, associated to 
a positroid cell of 
$Gr^{\geq 0}_{k,n-1}$. 
We define $\Ipre$ to be the map which takes $G$ 
 and adds a black lollipop at a new boundary vertex $n$,
as shown in 
 the middle graph of 
\cref{fig:amp}. 
Similarly, we define $\Iinc$ to be the map on a plabic graph $G'$ for $Gr^{\geq 0}_{k-1,n-1}$ which 
  modifies $G'$, changing the graph locally around vertices $1, n, n-1$, as shown at the right 
	of \cref{fig:amp}.
\end{definition}
\begin{remark}
The the resulting graph $\Ipre(G)$
is a reduced plabic graph for a cell of $Gr^{\geq 0}_{k,n}$.
It is not hard to show that, if $G'$ does not have white fixed points
	at vertices $1$ or $n-1$, then
	$\Iinc(G')$ is a reduced plabic graph for a cell of $Gr^{\geq 0}_{k,n}$.
\end{remark}

Abusing notation slightly, we also use 
$\Ipre$ and $\Iinc$ to denote the corresponding maps on 
positroid cells and positroid polytopes, decorated permutations,  etc.
Using \cref{def:rules},
one can also determine the effect of $\Ipre$ and $\Iinc$ on decorated permutations 
(and $\Le$-diagrams).
We leave the proof of the following lemma as an exercise.

\begin{lemma}
Let $\pi = (a_1, a_2,\dots,a_{n-1})$ be a decorated permutation on $n-1$ letters.
Then $\Ipre(\pi) = (a_1, a_2,\dots,a_{n-2},a_{n-1},n)$, where $n$ is a 
black fixed point.

Let $\pi = (a_1, a_2,\dots,a_{n-1})$ be a decorated permutation; assume that 
neither positions $1$ nor $n-1$ are white fixed points.
Let $h = \pi^{-1}(n-1)$.  Then 
$\Iinc(\pi)$ is the permutation such that
 $1 \mapsto n-1$, $h \mapsto n$,  $n \mapsto a_1$, and $j \mapsto a_j$
for all $j \neq 1, h, n$.
\end{lemma}

 The construction below is closely related to 
 the recursion from \cite[Definition 4.4]{Karp:2017ouj}, which is a 
 sort of $m=2$ version of the BCFW recurrence.

\begin{theorem}[BCFW recursions for the $m=2$ amplituhedron] \label{thm:disamp}\cite[Theorem A]{BaoHe}
	Let $\mathcal{C}_{n-1,k,2}$ (respectively $\mathcal{C}_{n-1,k-1,2}$)
	be a collection of Grasstopes which dissects
	the $m=2$ amplituhedron $\mathcal{A}_{n-1,k,2}(Z')$ (resp. $\mathcal{A}_{n-1,k-1,2}(Z'')$).  
	Then $$\mathcal{C}_{n,k,2} = \Ipre(\mathcal{C}_{n-1,k,2}) \cup 
	    \Iinc(\mathcal{C}_{n-1,k-1,2})$$ dissects $\mathcal{A}_{n,k,2}(Z).$
\end{theorem}

We use the term \emph{BCFW dissection}  
(respectively, \emph{BCFW tiling}) to refer to any dissection or 
tiling
that has the form $\mathcal{C}_{k,n}$ from \cref{thm:disamp}.

Diagrammatically, \cref{thm:disamp} reads as follows:
\begin{figure}[h]
\includegraphics[scale=0.35]{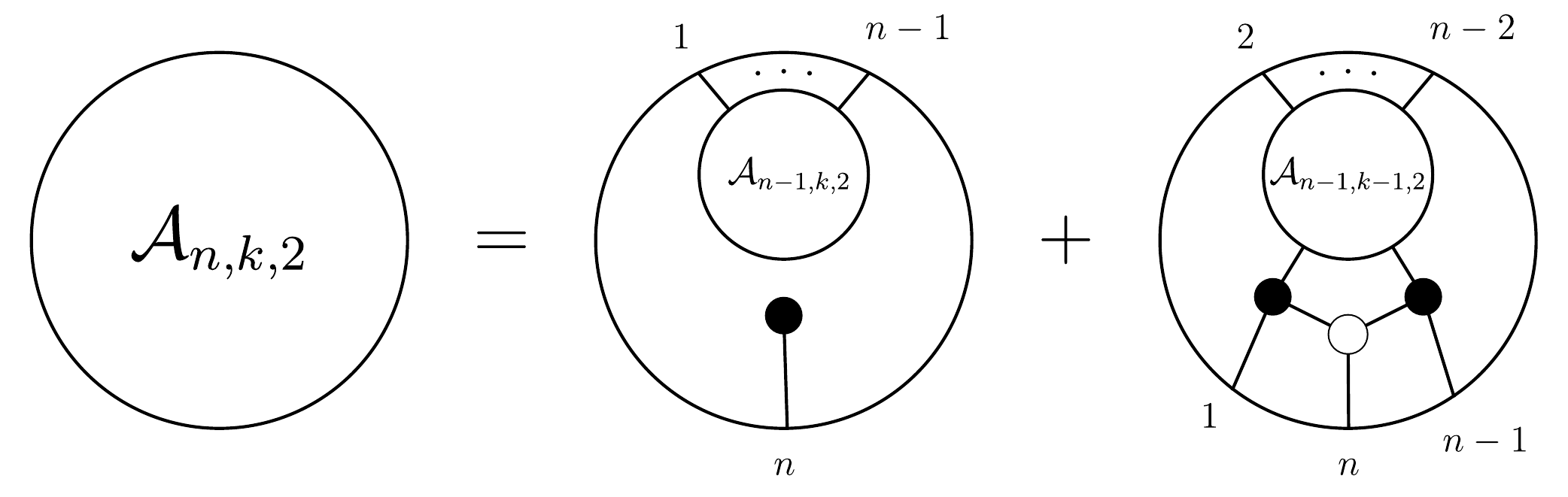}
	\caption{A BCFW-style 
	recursion for dissecting the amplituhedron.  There is a 
	 parallel recursion obtained from this one by cyclically shifting
	all boundary vertices of the plabic graphs by $i$ (modulo $n$).}
	\label{fig:amp}
\end{figure}

\begin{remark}\label{rem:cyc2}
Because of the cyclic symmetry of 
the positive Grassmannian and the amplituhedron 
	(see e.g. \cref{th:cyclicsym})
there are $n-1$ other versions of 
	\cref{thm:dishyper} (and \cref{fig:hyp}) in which all plabic graph labels
	get shifted by $i$  modulo $n$ (for $1 \leq i \leq n-1$).
\end{remark}

Note that \cite{BaoHe} worked in the setting of \emph{positroid tilings} -- i.e. they were only 
considering collections of cells that map injectively from the positive Grassmannian to the 
amplituhedron -- but \cref{thm:disamp} holds in the more general setting of dissections.

\begin{example}\label{ex:amp}
Let $n=5$ and $k=2$.  
We will use \cref{thm:disamp} to obtain a dissection of 
	$\mathcal{A}_{n,k,2}(Z) = \mathcal{A}_{5,2,2}$.
	We start with a dissection of $\mathcal{A}_{4,2,2}$ 
	coming from the plabic graph shown below (corresponding
	to the decorated permutation
	$(3,4,1,2)$),
	and a dissection of $\mathcal{A}_{4,1,2}$ (corresponding 
	to the permutations $(3,\underline{2},4,1)$ and $(2,3,1,\underline{4})$).
Applying the theorem leads to the three plabic graphs in the bottom
line, which correspond to the permutations 
	$(3,4,1,2,\underline{5}), (4,\underline{2},5,1,3), (4, 3,1,5,2)$.
\end{example}
\begin{center}
\begin{tabular}{c}
\raisebox{1.9cm}{$\mathcal{A}_{4,2,2}$:\,\,\,\,\,}\includegraphics[scale=0.15]{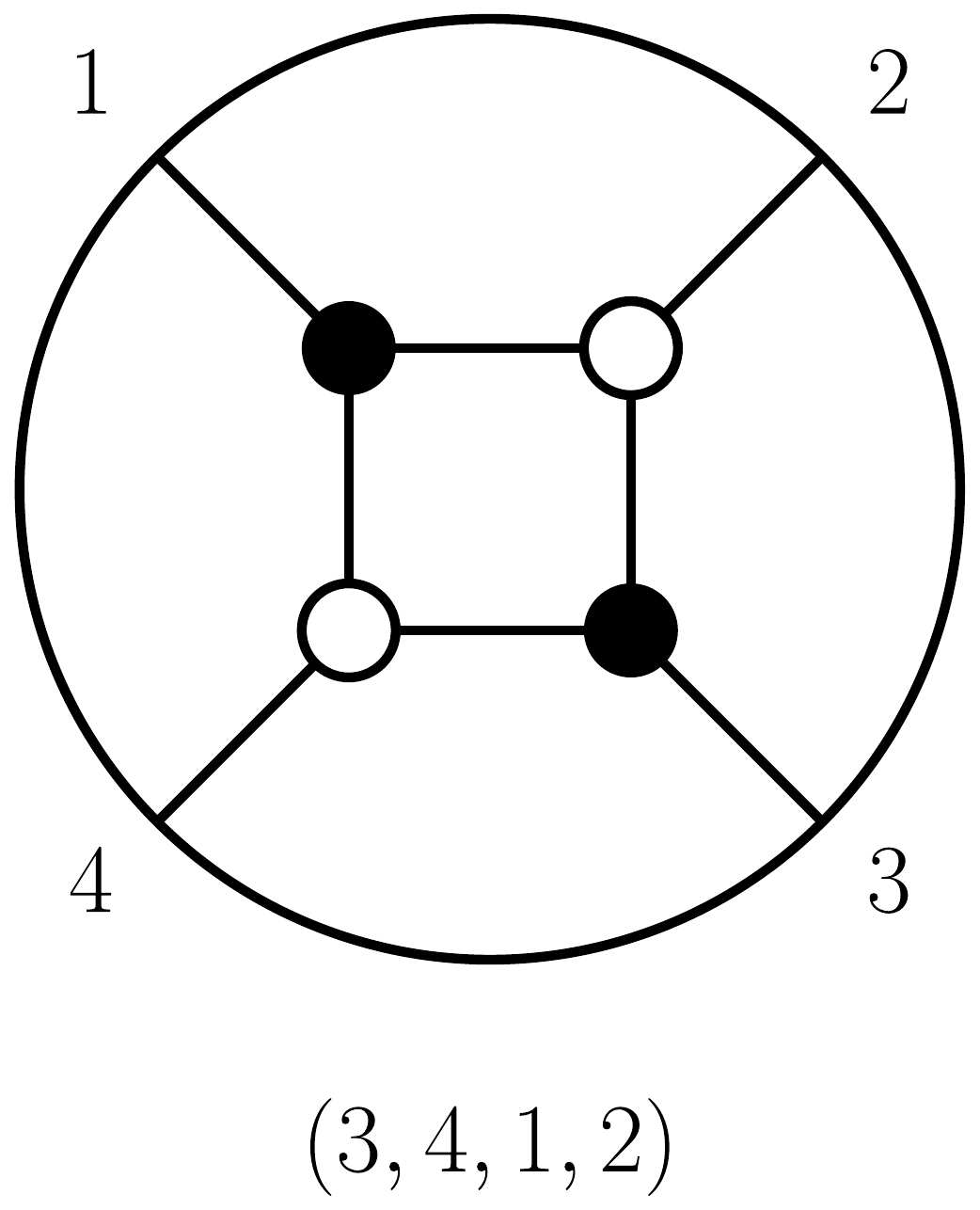}\qquad\qquad
\raisebox{1.9cm}{$\mathcal{A}_{4,1,2}$:\,\,\,\,\,\,}\includegraphics[scale=0.15]{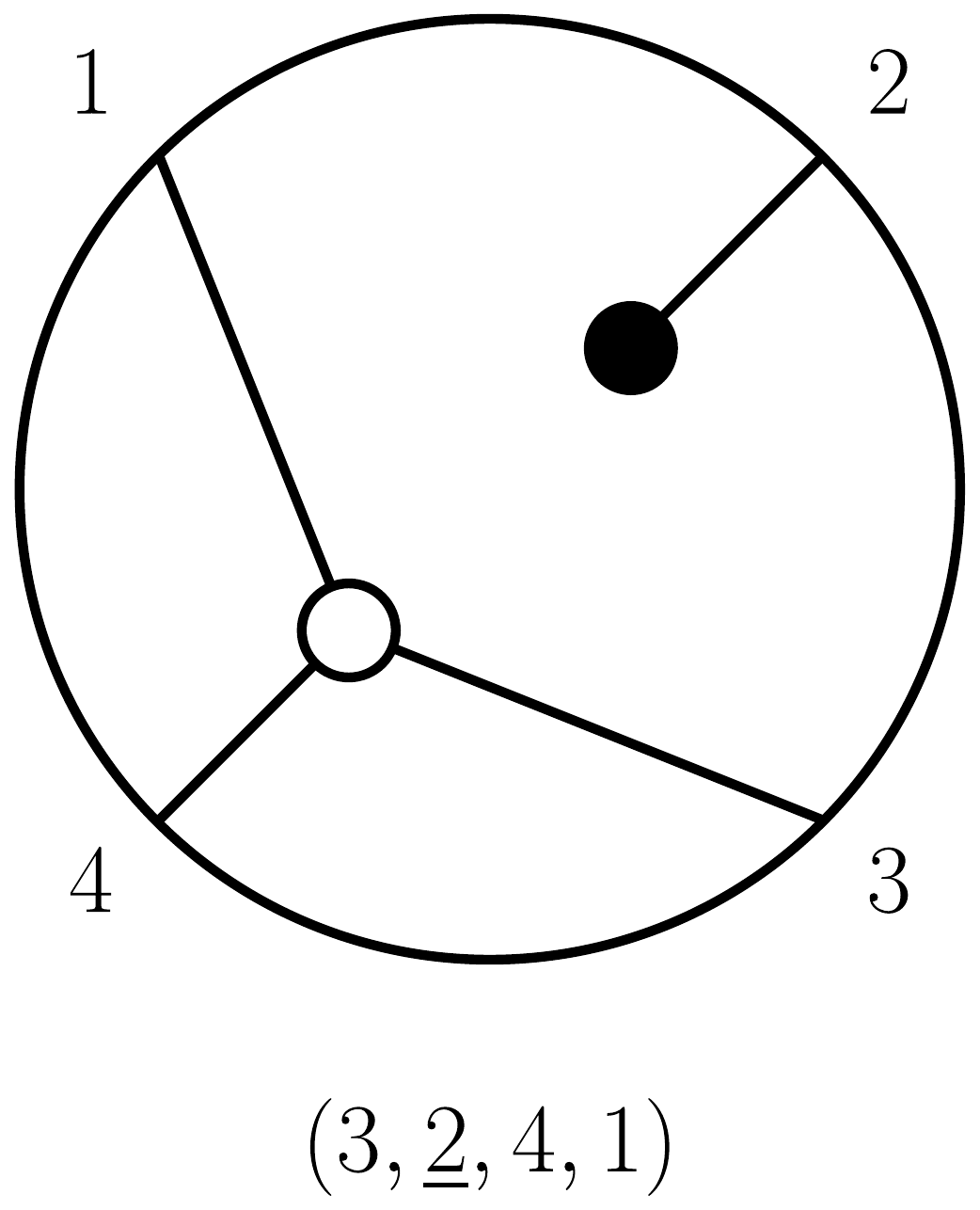}\quad\includegraphics[scale=0.15]{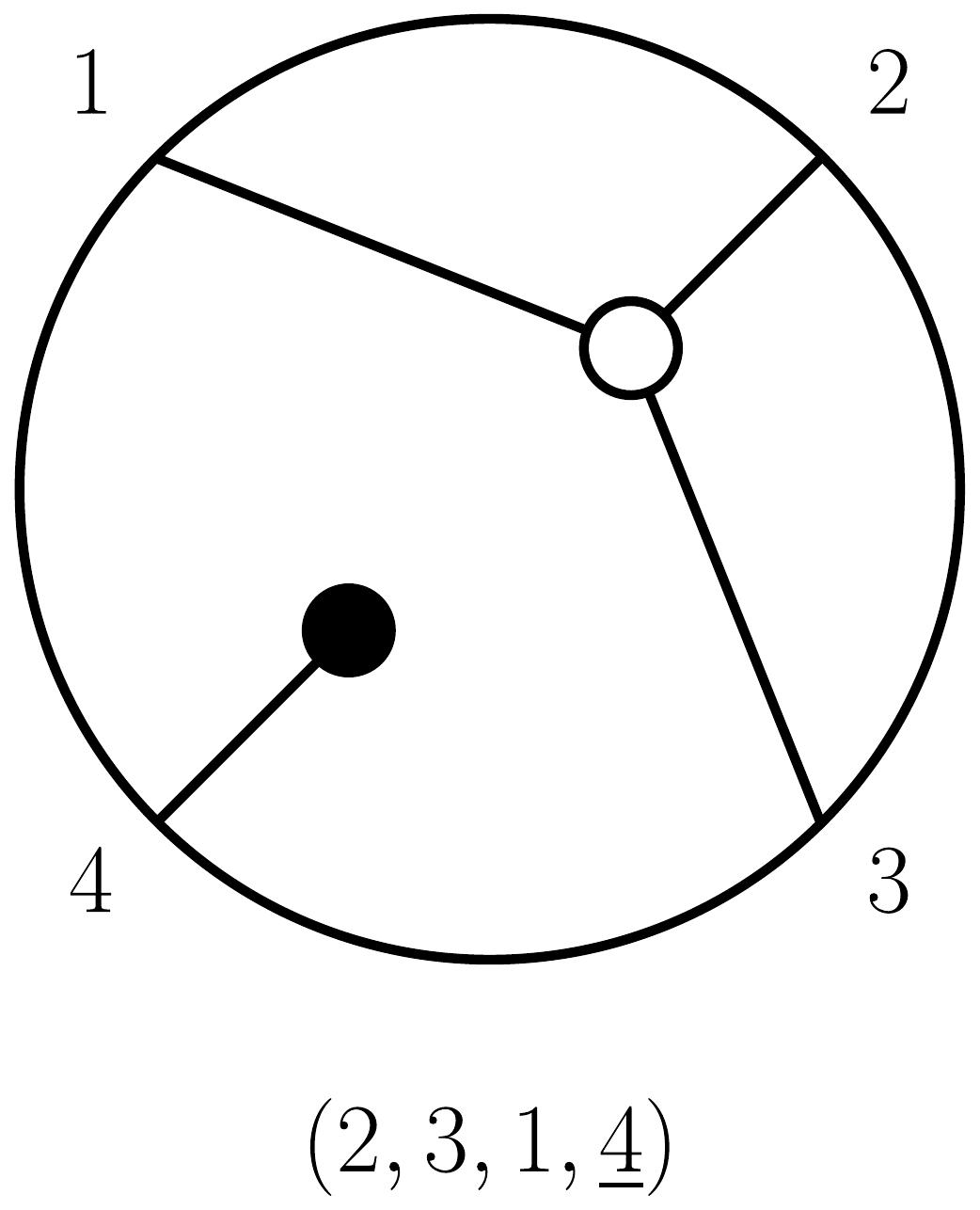}\\
\raisebox{1.9cm}{$\mathcal{A}_{5,2,2}$:\,\,\,\,\,\,}\includegraphics[scale=0.17]{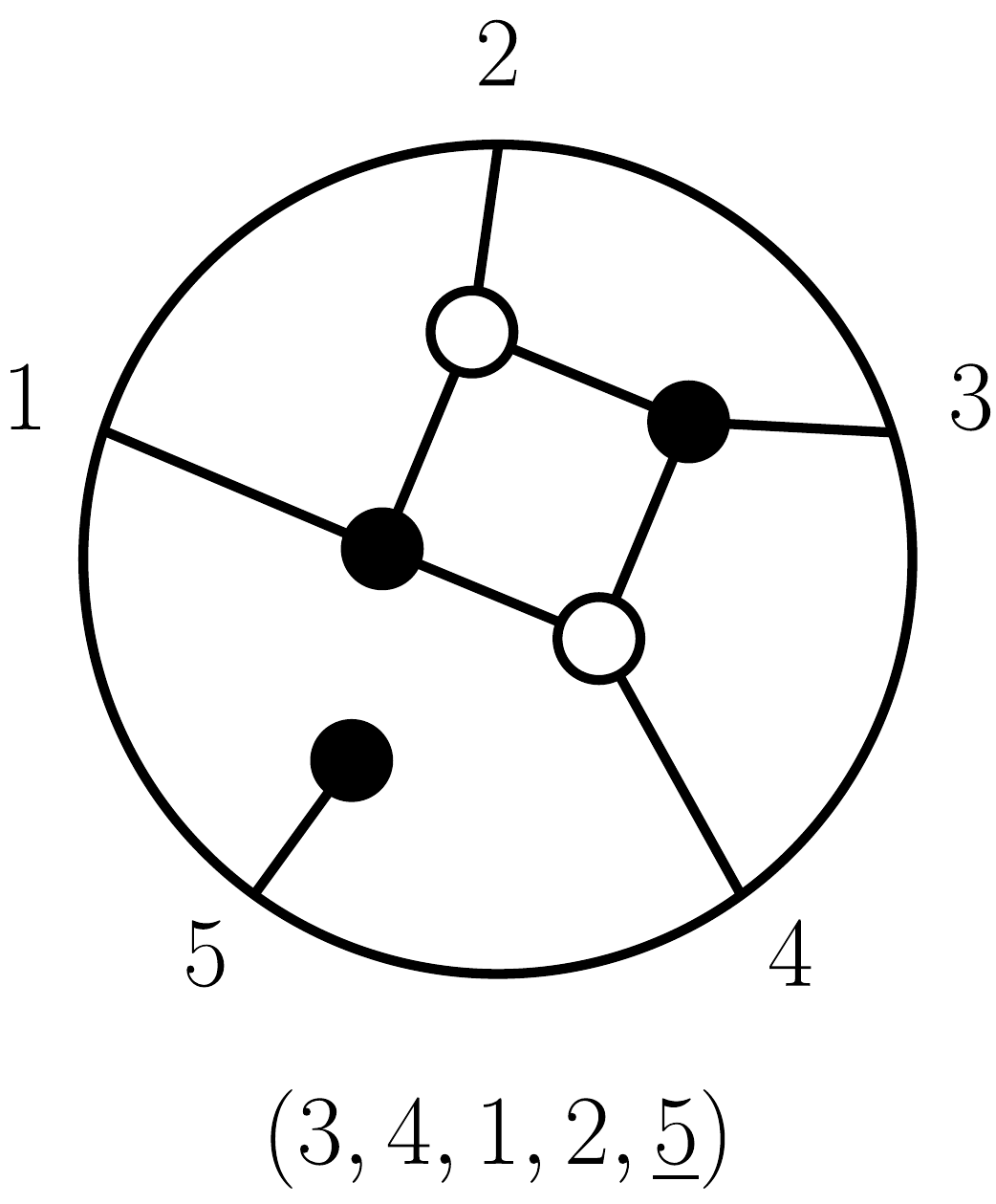}\quad\includegraphics[scale=0.17]{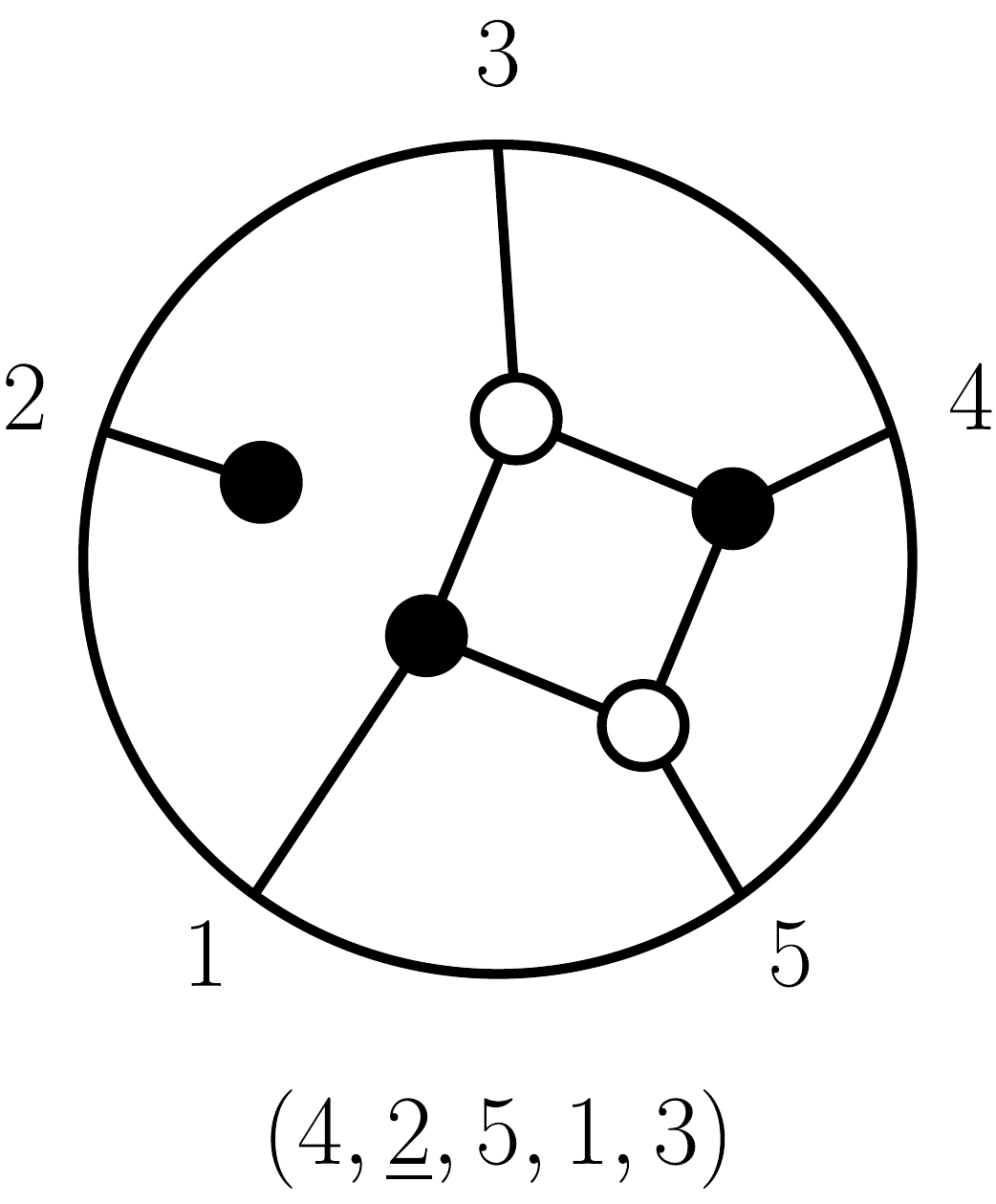}\quad\includegraphics[scale=0.17]{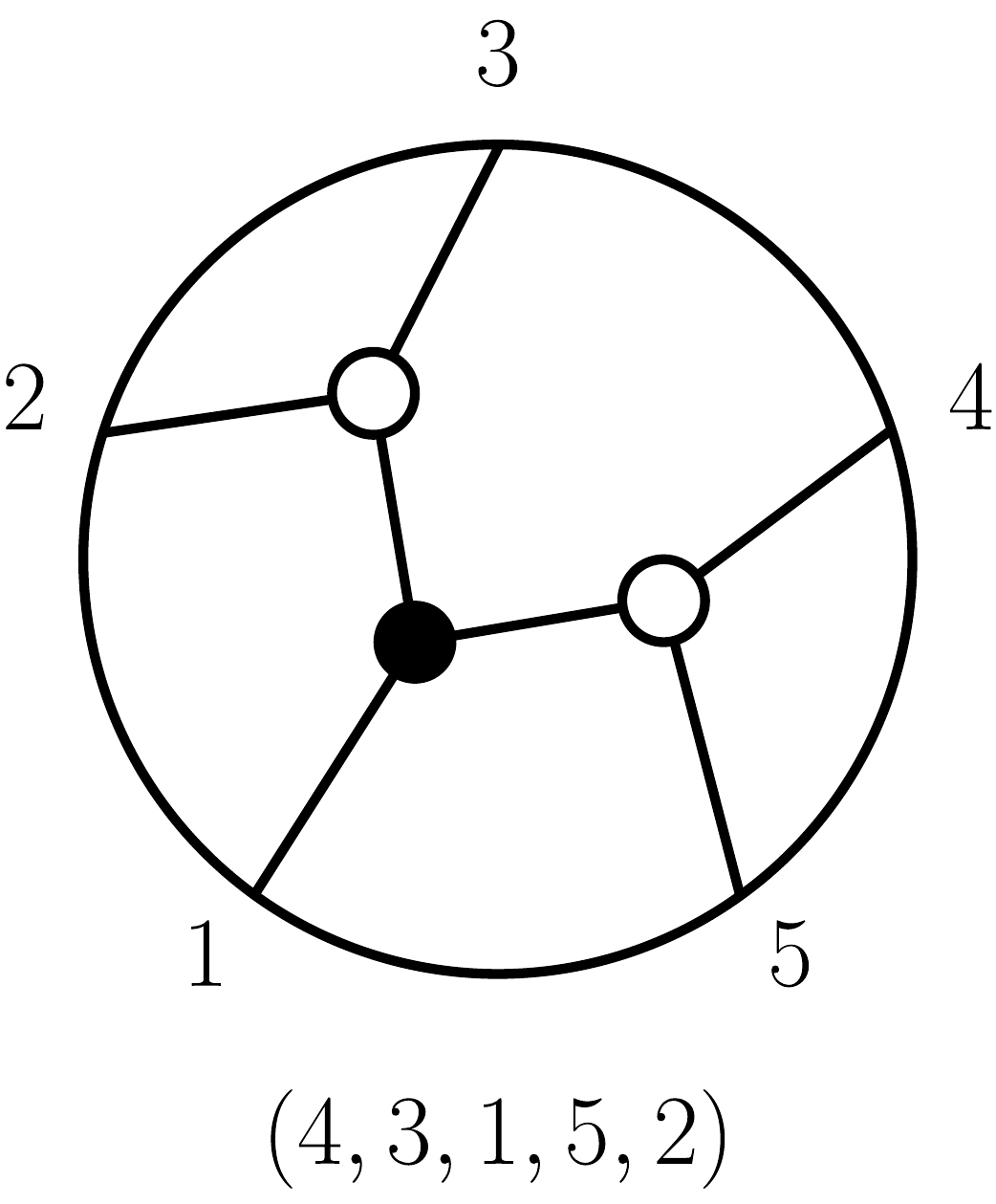}
\end{tabular}
\end{center}

\section{The T-duality map}
\label{T-duality}

In this section we define the \emph{T-duality map} 
(previously defined in \cite[Definition 4.5]{Karp:2017ouj}),
from certain positroid
cells of $Gr^{\geq 0}_{k+1,n}$ to positroid cells of $Gr^{\geq 0}_{k,n}$, and we prove
many remarkable properties of it.
We will subsequently explain, in \cref{thm:shift}, how 
the T-duality map gives a correspondence
between tilings (and more generally dissections)
of the hypersimplex $\Delta_{k+1,n}$ and 
 the amplituhedron $\mathcal{A}_{n,k,2}(Z)$.

To get a preview of the phenomenon we will illustrate, compare the 
decorated permutations labelling the plabic graphs in 
\cref{ex:hyper} and \cref{ex:amp}; can you spot the correspondence?
(This correspondence will be explained in \cref{thm:shift}.)

\subsection{T-duality as a map on permutations}

\begin{definition}\label{hatmap}
	We define the \emph{T-duality} 
	map from loopless decorated permutations on $[n]$
to coloopless decorated permutations on $[n]$ as follows.
Given a loopless decorated permutation $\pi=(a_1,a_2,\dots,a_n)$ (written
in list notation) on $[n]$,
	we define the decorated permutation $\hat{\pi}$ by 
	$\hat{\pi}(i) = \pi(i-1)$, so that 
	$\hat{\pi} = (a_n, a_1, a_2,\dots,a_{n-1})$,
where any fixed points in $\hat{\pi}$ are declared to be loops. 
	Equivalently, $\hat{\pi}$ is obtained from $\pi$ by composing $\pi$ with the 
	permutation $\pi_0=(n,  1,  2 ,\dots ,  n-1)$ in the symmetric group, $\hat\pi=\pi_0\circ \pi$. 
\end{definition}

Recall that an anti-excedance of a decorated permutation is a position
$i$ such that $\pi(i) < i$, or $\pi(i)=i$ and $i$ is a coloop.
Our first result shows that T-duality is a bijection between
loopless cells of 
 $Gr^{\geq 0}_{k+1,n}$
and coloopless cells of $Gr^{\geq 0}_{k,n}$.

\begin{lemma}\label{lem:hat}
The T-duality map $\pi \mapsto \hat{\pi}$ is  a bijection between the 
loopless permutations on $[n]$ with $k+1$ anti-excedances,
and the coloopless permutations on $[n]$ with $k$ anti-excedances.
Equivalently, the T-duality map is bijection between
loopless positroid cells of 
 $Gr^{\geq 0}_{k+1,n}$
and coloopless positroid cells of $Gr^{\geq 0}_{k,n}$.
\end{lemma}

\begin{proof}
	The second statement follows from the first by \cref{app}, so 
	it suffices to prove the first statement.
	Let $\pi=(a_1,\dots,a_n)$ be a loopless permutation on $[n]$ 
	with $k+1$ anti-excedances; 
	then $\hat{\pi}=(a_n,a_1,\dots,a_{n-1})$. 
Consider any $i$ such that $1 \leq i \leq n-1$.  Suppose $i$ is a position of a 
	anti-excedance, i.e. either $a_i<i$ or $a_i = \overline{i}$.
	Then the letter $a_i$ appears in the $(i+1)$st position in $\hat{\pi}$,
	and since $a_i<i+1$, we again have an anti-excedance.
	On the other hand, if $i$ is \emph{not} a position of an anti-excedance, i.e. 
	$a_i>i$ (recall that $\pi$ is loopless), then 
	in the $(i+1)$st position of $\hat{\pi}$ we have $a_i \geq i+1$.
	By \cref{hatmap} if we have a fixed point in position $i+1$ (i.e. $a_i = i+1$)
	this is a loop, and so position $i+1$ of $\hat{\pi}$ will not be 
	a anti-excedance.
	Therefore if $I \subset [n-1]$ is the positions of the anti-excedances
	located in the first $n-1$ positions of $\pi$, then $I+1$ is the positions of the anti-excedances
located in positions $\{2,3,\dots,n\}$ in $\hat{\pi}$.

Now consider position $n$ of $\pi$.  Because $\pi$ is loopless, $n$ will be 
	the position of a anti-excedance in $\pi$.  And because $\hat{\pi}$
	is defined to be coloopless, $1$ will never be the position of a 
	anti-excedance in $\hat{\pi}$.
	Therefore the number of anti-excedances of $\hat{\pi}$ will be
	precisely one less than the number of anti-excedances of $\pi$.

It is easy to reverse this map so it is a bijection.
\end{proof}

\begin{remark}
	Since by \cref{lem:hat} the map $\pi \mapsto \hat{\pi}$ is a
	bijection, we can also talk about the inverse map from 
	coloopless permutations on $[n]$ with $k$ anti-excedances
	to loopless permutations on $[n]$ with $k+1$ anti-excedances.
	We denote this inverse map by $\pi \mapsto \check{\pi}$.
\end{remark}

\begin{remark}
Our map $\pi \mapsto \hat{\pi}$ is in fact a special case of the map
$\rho_A$ introduced by Benedetti-Chavez-Tamayo 
	in \cite[Definition 23]{posquotients} (in the case where 
	$A = \emptyset$).  
\end{remark}

\subsection{T-duality as a map on cells}\label{sec:Tcells}
While we have defined the T-duality map as a map $\pi \mapsto \hat{\pi}$ on the 
permutations labelling positroid cells, 
it can be shown that it is induced from a
map on the corresponding cells. 
We will follow here the derivation in \cite{abcgpt} and define a $Q$-map which 
maps elements of 
the positroid cell $S_{\pi}$ of $Gr^{\geq 0}_{k+1,n}$
 to the positroid cell $S_{\hat{\pi}}$ of $Gr^{\geq 0}_{k,n}$. 
Note that in much of this section we allow $m$ to be any positive even integer.

\begin{definition}
Let $\lambda \in Gr_{\frac{m}{2},n}$. We say that $\lambda$ is  {\it generic} if $p_I(\lambda) \neq 0$ for all $I \in \binom{[n]}{ \frac{m}{2}}$.
\end{definition}

For $m=2$, $\lambda=(\lambda_1,\lambda_2,\ldots,\lambda_n)\in\mathbb{R}^n$ is generic in $\mathbb{R}^n$ if $\lambda_i\neq0$ for all $i=1,\ldots,n$.

\begin{lemma}\label{lem:rankgeneric}
Given $C=(c_1,c_2,\ldots,c_n)$ representing an element of $Gr_{k+\frac{m}{2},n}$ where $c_i$ are columns of $C$, then $C$ contains a generic $\frac{m}{2}$-plane if and only if $\mbox{rank} \left(\{c_i\}_{i \in I}\right)=\frac{m}{2}$ for all $I \in {[n] \choose \frac{m}{2}}$.
\end{lemma}
\begin{proof}
If a generic $\frac{m}{2}$-plane $\lambda \in M(\frac{m}{2},n)$ is contained in $C$, then there is a matrix $h \in M(\frac{m}{2},k+\frac{m}{2})$ such that $\lambda=h \cdot C$. Then  $p_I(\lambda)=\sum_{J \in {[k+\frac{m}{2}] \choose \frac{m}{2}}} p_J(h)C_J^I$, with $I \in {[n] \choose \frac{m}{2}}$. If $\mbox{rank} \left(\{c_i\}_{i \in I}\right)=\frac{m}{2}$ then there exist $J_I \in {[k+\frac{m}{2}] \choose \frac{m}{2}}$ such that $C_{J_I}^I \not =0$, therefore it is enough to choose $h$ such that $p_{J_I}(h) \not =0$ in order to guarantee $\lambda=h \cdot C$ is generic. Vice-versa if we assume $\mbox{rank} \left(\{c_i\}_{i \in I}\right)<\frac{m}{2}$ then $C^I_J=0$ for all $J \in {[k+\frac{m}{2}] \choose \frac{m}{2}}$ and this would imply $p_I(\lambda)=0$.
\end{proof}

If we specialize to the $m=2$ case, we have the following:
\begin{lemma}\label{lem:m2generic}
Let $S_{\pi}$ be a positroid cell in $Gr^{\geq 0}_{k+1,n}$. Then $S_{\pi}$ is loopless 
if and only if every vector space $V \in S_{\pi}$ contains a generic vector.
\end{lemma}

\begin{lemma} \label{lem:condgenm}
Let $S_{\pi}$ be a positroid cell.
If every vector space $V \in S_{\pi}$ 
contains a generic $\frac{m}{2}$-plane then $\pi(i) \geq i+\frac{m}{2}$ (as an affine permutation, see \cref{defn:affperm}) for all $i$.
\end{lemma}
\begin{proof}
Let $C=(c_1,c_2,\ldots,c_n)$ be a matrix representing $V$, listed as a sequence of column
vectors.  
Let us assume that there
exists $a$ such that $\pi(a) \leq a+\frac{m}{2}-1$.  Then $c_a \in \mbox{span}\{c_{a+1},\ldots, c_{a+\frac{m}{2}-1}\}$ and, in particular, $r[a;a+\frac{m}{2}-1]<\frac{m}{2}$. The proof follows immediately from 
\cref{lem:rankgeneric}.
\end{proof}

\begin{definition}
For a positroid cell $S_\pi\subset Gr^{\geq 0}_{k+\frac{m}{2},n}$ and  $\lambda\in Gr_{\frac{m}{2},n}$ a generic vector of an element $V \in S_\pi$, we define $$S_\pi^{(\lambda)}:=\{W\in S_\pi:\lambda\subset W\}.$$
\end{definition}

Let $C^{(\lambda)}_\pi$ be matrix representatives for elements in $S_\pi^{(\lambda)}$. It is always possible to find an invertible row transformation which bring $C_\pi^{(\lambda)}$ into the form
\begin{equation}
C_\pi^{(\lambda)}=\left(\begin{matrix}
\lambda_{1\,1}&\lambda_{1\,2}&\ldots&\lambda_{1\,n}\\
\vdots&\vdots&\ddots&\vdots\\
\lambda_{\frac{m}{2} \,1}&\lambda_{\frac{m}{2}\,2}&\ldots&\lambda_{\frac{m}{2}\,n}\\
c_{\frac{m}{2}+1\,1}&c_{\frac{m}{2}+1\,2}&\ldots&c_{\frac{m}{2}+1\,n}\\
\vdots&\vdots&\ddots&\vdots\\
c_{\frac{m}{2}+k\,1}&c_{\frac{m}{2}+k\, 2}&\ldots&c_{\frac{m}{2}+k\, n}
\end{matrix}\right)
\end{equation}
Let us define a linear transformation $Q^{(\lambda)}: \mathbb{R}^n \mapsto \mathbb{R}^n$ represented by the $n\times n$ matrix $Q^{(\lambda)}$ with elements\footnote{Notice that our definition differs from the one found in \cite{abcgpt} for $m=4$. They are however related to each other by a cyclic shift and rescaling each column of $Q^{(\lambda)}$.}
\begin{equation}
Q^{(\lambda)}_{ab}=\sum_{i=0}^{\frac{m}{2}} (-1)^i \, \delta_{a,b-\frac{m}{2}+i} \, \, p_{ b-\frac{m}{2}, \ldots, b-\frac{m}{2}+i-1, \frac{m}{2}+i+1,\ldots, b }\left(\lambda\right), \quad a,b, \in [n].
\end{equation}
Here we used the notation where $\delta_{ab}=1$ when $a=b$ and $\delta_{ab}=0$ otherwise. 

It is easy to show that $\lambda Q^{(\lambda)}=0$ and that $Q^{(\lambda)}$ has rank $n-\frac{m}{2}$. Let us define $\hat{C}_\pi^{(\lambda)}=C_\pi^{(\lambda)}\cdot Q^{(\lambda)}$, then
\begin{equation}\label{eq:C}
\hat{C}_{\pi}^{(\lambda)}=\left(\begin{matrix}
0&0&\ldots&0\\
\vdots&\vdots&\ddots&\vdots\\
0&0&\ldots&0\\
\hat{c}_{\frac{m}{2}+1\,1}&\hat{c}_{\frac{m}{2}+1\,2}&\ldots&\hat{c}_{\frac{m}{2}+1\,n}\\
\vdots&\vdots&\ddots&\vdots\\
\hat{c}_{\frac{m}{2}+k\,1}&\hat{c}_{\frac{m}{2}+k\, 2}&\ldots&\hat{c}_{\frac{m}{2}+k\, n}
\end{matrix}\right).
\end{equation}
It is easy to check that $\mbox{span}\{\hat{c}_{a},\hat{c}_{a+1},\ldots,\hat{c}_b\}\subset \mbox{span}\{c_{a-\frac{m}{2}},c_{a-\frac{m}{2}+1},\ldots,c_b\}$ and moreover that for consecutive maximal minors we have: $p_{a-\frac{m}{2},\ldots a,\ldots a+k-1}(C)$ is proportional to $p_{a, \ldots, a+k-1}(\hat{C})$.
Then, the matrix $Q^{(\lambda)}$  projects  elements of $S_\pi^{(\lambda)}$ into $S_{\hat{\pi}}$, with 
\begin{equation}\label{Tduality_gen}
\hat{\pi}(i)=\pi(i-\frac{m}{2}).
\end{equation} The proof of this fact closely follows the one found in \cite[page 75]{abcgpt}. 

For $m=2$ we get the explicit form of $Q^{(\lambda)}$ is:
\begin{equation}
Q^{(\lambda)}_{ab}=\delta_{a,b-1}\lambda_b-\delta_{a,b}\lambda_{b-1}\,, \quad a,b \in [n].
\end{equation}
Moreover, we have the following relation between consecutive minors 
\begin{equation}
p_{a,a+1,\ldots, a+k-1}(\hat{C})=(-1)^k {\lambda_{a} \ldots \lambda_{a+k-2}}\, p_{a-1,a,\ldots, a+k-1}(C).
\end{equation}

\begin{remark}
In order for the T-duality map to be a well-defined  (on affine permutations), we require that both $i \leq \pi(i) \leq n+i$ and $i \leq \hat{\pi}(i) \leq n+i$ are satisfied. Given that $\hat{\pi}(i)=\pi(i-\frac{m}{2})$, this implies extra conditions on allowed permutations, i.e. $\pi(i)\geq i+\frac{m}{2}$ and $\hat{\pi}(i) \leq i+n-\frac{m}{2}$. We observe that the operation in \eqref{Tduality_gen} is then well-defined for the cells $S_\pi^\lambda$, by \cref{lem:condgenm}. Finally, for $m=2$ these conditions correspond to lack of loops (resp. coloops) for $\pi$ (reps. $\hat{\pi}$).
\end{remark}

\begin{proposition}[How T-duality affects dimensions of cells]\label{prop:2}
	Let $S_{\pi}$ be a loopless 
	cell of $Gr_{k+1,n}^{\geq 0}$.
	Then $S_{\hat{\pi}}$ is a 
	coloopless cell of $Gr_{k,n}^{\geq 0}$, and 
		$\dim(S_{\hat{\pi}}) -2k = \dim(S_\pi) - (n-1).$
	In particular, if $\dim S_{\pi} = n-1$,
	then $\dim S_{\hat{\pi}} = 2k$. 
\end{proposition}

\begin{proof}
Let us translate \cref{hatmap} into the language of affine permutations. Then T-duality maps a $(k+1,n)$-bounded affine permutation $\pi_a$ into a $(k,n)$-bounded affine permutation $\hat{\pi}_a=\pi_a \circ t$, with $t:\mathbb{Z} \rightarrow \mathbb{Z}$ the map $i \mapsto i-1$.
By \cite[Proposition~17.10]{postnikov} and \cref{app}, the codimension of the positroid cell $S_{\nu_a}$ equals the length $\ell(\nu_a)$ of the associated affine permutation $\nu_a$. Clearly the map $t$ preserves the set of inversions, and hence the length, of affine bounded permutations, i.e. $\ell(\pi_a)=\ell(\hat{\pi}_a)$. Therefore the codimensions of $S_{\pi_a} \subseteq Gr^{\geq 0}_{k+1,n}$ and $S_{\hat{\pi}_a} \subseteq Gr^{\geq 0}_{k,n}$ are equal:
\begin{equation}
(k+1)(n-k-1)-\mbox{dim}(S_{\pi_a})=k(n-k)-\mbox{dim}(S_{\hat{\pi}_a}),
\end{equation}
from which the claim of the proposition follows immediately.
\end{proof}

\begin{remark} 
Alternatively, one may prove the above result by mimicking an argument of 
	a similar statement given in 
	\cite[pages 75-76]{abcgpt}.
\end{remark}

\section{T-duality relates tiles, tilings, and dissections}\label{sec:Ttriangle}

In this section we will compare the 
positroid tiles and tilings (and more generally, dissections)
of 
the hypersimplex $\Delta_{k+1,n}$ with those of the 
amplituhedron $\mathcal{A}_{n,k,2}(Z)$.  
Again, we will see that T-duality connects them!
Our main result of this section is \cref{thm:shift}, which says that 
T-duality provides a bijection between the BCFW tilings/dissections
of the hypersimplex $\Delta_{k+1,n}$, and the BCFW tilings/dissections
of the amplituhedron $\mathcal{A}_{n,k,2}(Z)$.

The $2k$-dimensional cells of $Gr^{\geq 0}_{k,n}$ 
which have full-dimensional
image in $\mathcal{A}_{n,k,2}(Z)$ were studied in 
\cite{Lukowski:2019sxw} and called \emph{generalized triangles}.
In this paper we will refer to the above objects as \emph{positroid tiles} defined as follows.
\begin{definition}[Positroid tiles of $\mathcal{A}_{n,k,2}$]\label{def:gentriangle}
Let $S_{\pi}$ be a $2k$-dimensional cell of $Gr^{\geq 0}_{k,n}$ such that 
$\dim Z_{\pi} = \dim S_{\pi}$, and the restriction of the amplituhedron map $\widetilde{Z}$ to $S_{\pi}$ is an injection. Then we call $Z_{\pi}$ 
a \emph{positroid tile}
of $\mathcal{A}_{n,k,2}(Z)$. 
\end{definition}

A conjectural description of positroid tiles
was given in \cite{Lukowski:2019sxw}: 
\begin{definition}
We say that a collection of convex polygons (which have
$p_1,\dots,p_r$ vertices) inscribed in a given $n$-gon 
is a collection of \emph{$k$ non-intersecting triangles in an $n$-gon}
if each pair of such polygons intersects in at most a vertex 
and if the total number of triangles needed to triangulate all polygons in the collection is $k$, i.e. $(p_1-2)+\ldots+(p_r-2)=k$. 
\end{definition}
It was conjectured and experimentally checked in \cite{Lukowski:2019sxw} that positroid tiles in $\mathcal{A}_{n,k,2}(Z)$ are in bijection with 
 collections of `$k$ non-intersecting triangles in a $n$-gon'.
Moreover, one can read off the cell $S_{\pi}$ of $Gr^{\geq 0}_{k,n}$
corresponding to a positroid tile of $\mathcal{A}_{n,k,2}(Z)$ using
the combinatorics of the collection of $k$ non-intersecting triangles 
in an $n$-gon, see \cite[Section 2.4]{Lukowski:2019sxw}. The basic idea is to associate a row vector to each of the non-intersecting triangles, with generic entries at the positions of the triangle vertices (and zeros everywhere else). This way one constructs a $k\times n$ matrix whose matroid is the matroid for $S_\pi$.

Borrowing the terminology of \cref{def:gentriangle},
we make the following definition.
\begin{definition}[Positroid tiles of $\Delta_{k+1,n}$] \label{def:gentriangle2}
Let $S_{\pi}$ be an $(n-1)$-dimensional cell of $Gr^{\geq 0}_{k+1,n}$ such that 
	the moment map $\mu$ is an injection on $S_{\pi}$.  
Then we say 
	the image $\Gamma_{\pi}:= \overline{\mu(S_{\pi})}$ in $\Delta_{k+1,n}$ is a \emph{positroid tile}
in $\Delta_{k+1,n}$.
\end{definition}

We have already studied the positroid tiles
in $\Delta_{k+1,n}$ in 
 \cref{prop:homeo2}: they  
come from $(n-1)$-dimensional positroid cells whose matroid is 
connected, or equivalently, they come from the 
positroid cells whose reduced plabic graphs are trees.
And since these are positroid cells in $Gr^{\geq 0}_{k+1,n}$, 
each such plabic graph, when
drawn as a trivalent graph, is a tree with $n$ leaves with
 precisely $k$ internal black vertices. By simply taking the planar dual of these tree, we get the following: 

\begin{proposition}\label{labels.map}
There is a bijective map between positroid tiles
	in $\Delta_{k+1,n}$ 
	and collections of $k$ non-intersecting triangles in an $n$-gon.
\end{proposition}

\begin{proof}
Consider a collection of non-intersecting polygons inside an $n$-gon $\mathcal{P}=(P_1,\ldots,P_r)$ and its complement $\overline{\mathcal{P}}=(\overline{P}_1,\ldots \overline{ P}_{\bar{r}})$. Let us choose a triangulation of all polygons into triangles  $\mathcal{P}\to \mathcal{T}=(T_1,\ldots ,T_k)$ and $\overline{\mathcal{P}}\to \overline{\mathcal{T}}=(\overline{T}_1,\ldots,\overline{T}_{n-k-2})$. Associate a black vertex to the middle of each triangle $T$ and a white vertex with to middle of each triangle $\overline{T}$. Finally, connect each pair of vertices corresponding to triangles sharing an edge and draw an edge through each boundary of the $n$-gon. This way we get a tree graph with exactly $k$ black and $n-k-2$ white vertices. Hence it is a plabic graph for the cell $S_{\pi}\subset G_{k+1,n}^{\geq 0}$ corresponding to a plabic tile of $\Delta_{k+1,n}$. 
\end{proof}

\begin{figure}
\begin{tabular}{cccc}\includegraphics[scale=0.5]{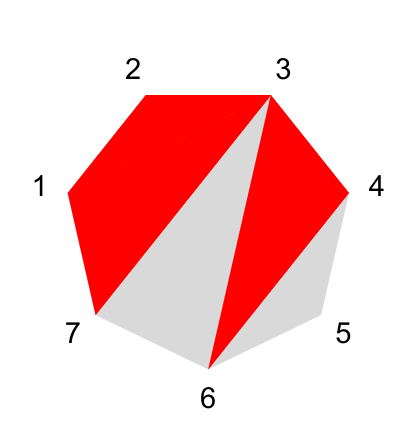}&\includegraphics[scale=0.5]{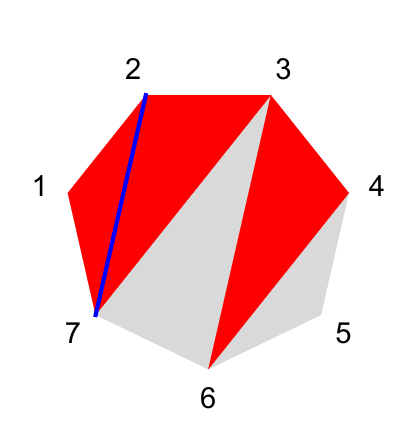}&\includegraphics[scale=0.5]{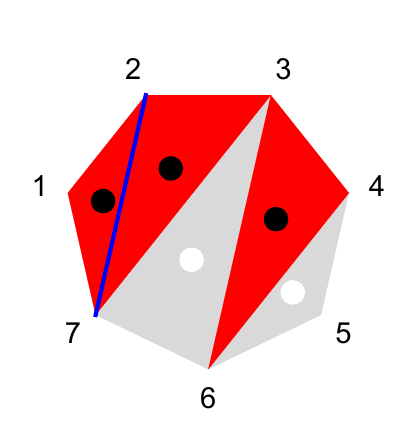}&\includegraphics[scale=0.5]{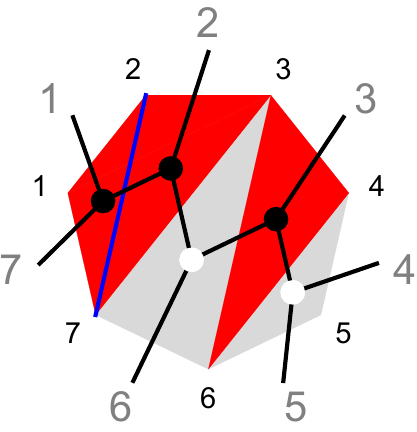}\\(a)&(b)&(c)&(d)
\end{tabular}
\caption{The map in \cref{labels.map} for $\pi=\{4,7,1,6,\underline{5},3,2\}\in Gr_{3,7}^{\geq 0}$:\\ (a) positroid tile label, (b) A triangulation of collections $\mathcal{P}$ and $\bar{\mathcal{P}}$, \\(c) Assigning vertices, (d) Plabic graph of $\check{\pi}=\{7,1,6,5,3,2,4\}\in Gr_{4,7}^{\geq 0}$ }
\end{figure}

In the following theorem we show that T-duality relates 
BCFW tilings and 
dissections of the hypersimplex and amplituhedron.

\begin{theorem}[BCFW tilings of $\Delta_{k+1,n}$ and $\mathcal{A}_{n,k,2}(Z)$ are T-dual]\label{thm:shift}
The T-duality map provides a bijection between 
	the BCFW tilings of the hypersimplex $\Delta_{k+1,n}$
	and the BCFW tilings of the amplituhedron
$\mathcal{A}_{n,k,2}(Z)$.
That is, the collection $\{\Gamma_{\pi}\}$ of positroid polytopes constructed in 
\cref{thm:disamp} is a positroid tiling of $\Delta_{k+1,n}$ if and only if 
 the T-dual collection $\{Z_{\hat{\pi}}\}$ of Grasstopes is a positroid tiling of $\mathcal{A}_{n,k,2}(Z)$. 
The same statement holds if we replace the word ``tiling''
 with ``dissection.''
\end{theorem}

\begin{proof}
We prove this by induction on $k+n$, using 
	\cref{thm:dishyper} and \cref{thm:disamp}.
	It suffices to show:
	\begin{itemize}
		\item 
	if $\{\Gamma_{\pi}\}_{\pi \in \mathcal{C}}$ dissects
	$\Delta_{k+1,n-1}$ and 
	$\{Z'_{\hat{\pi}}\}_{\pi \in \hat{\mathcal{C}}}$ dissects
	$\mathcal{A}_{n-1,k,2}(Z')$ then 
	for any $\pi \in \mathcal{C}$,
	$\widehat{\ipre(\pi)} = \Ipre(\hat{\pi})$.
		\item 
	if $\{\Gamma_{\pi}\}_{\pi \in \mathcal{C}}$ dissects
	$\Delta_{k,n-1}$ and 
	$\{Z''_{\hat{\pi}}\}_{\pi \in \hat{\mathcal{C}}}$ dissects
	$\mathcal{A}_{n-1,k-1,2}(Z'')$ then 
	for any $\pi \in \mathcal{C}$,
	$\widehat{\iinc(\pi)} = \Iinc(\hat{\pi})$.
	\end{itemize}

	Let $\pi = (a_1,\dots,a_{n-1})$
	be a decorated permutation.
We first verify the first statement.  
Then $\ipre(\pi) = (a_1, a_2,\dots,a_{n-2},n,a_{n-1})$,
	so $\widehat{\ipre(\pi)} = (a_{n-1},a_1,a_2,\dots, a_{n-2},n)$,
	where $n$ is a black fixed point.
	Meanwhile, $\hat{\pi} = (a_{n-1}, a_1, a_2,\dots, a_{n-2})$,
	so $\Ipre(\hat{\pi}) = (a_{n-1}, a_1,a_2,\dots,a_{n-2}, n)$, where
	$n$ is a black fixed point.

We now verify the second statement. 
Let $j =\pi^{-1}(n-1)$.
Then we have that $\iinc(\pi) = 
(a_1,a_2,\dots, a_{j-1}, n, a_{j+1},\dots,a_{n-1},n-1)$,
	and 
	$\widehat{\iinc(\pi)} = (n-1, a_1, a_2,\dots, a_{j-1}, n, a_{j+1},\dots, a_{n-1}).$
	Meanwhile $\hat{\pi} = (a_{n-1},a_1, a_2,\dots, a_{n-2})$.  Then  
	it is straightforward to verify that 
	$\Iinc(\hat{\pi})$ is exactly the 
	permutation 
	$\widehat{\iinc(\pi)} = (n-1, a_1, a_2,\dots, a_{j-1}, n, a_{j+1},\dots, a_{n-1}),$ as desired.
\end{proof}

We now see that T-duality relates 
positroid tiles of the hypersimplex and the amplituhedron.
\begin{proposition}\label{prop:injective}
Suppose the positroid polytope $\Gamma_{\pi}$ is a positroid tile of the hypersimplex $\Delta_{k+1,n}$. Then the T-dual Grasstope $Z_{\hat{\pi}}$ is a positroid tile of the amplituhedron $\mathcal{A}_{n,k,2}(Z)$ for all $Z \in \Mat^{>0}_{n,k+2}$.
\end{proposition}
\begin{proof}
By \cref{prop:homeo2}, the fact that $\mu$ is injective implies that a (any) reduced
plabic graph $G$ representing $S_{\pi}$ must be a (planar) tree.  But then by 
\cref{thm:dishyper} (see \cref{fig:hyp}), $G$ has  a black or white
vertex which is incident to two adjacent boundary vertices $i$ and $i+1$ (modulo $n$),
and hence appears in some tiling of the hypersimplex
(and specifically on the right-hand side of 
\cref{fig:hyp}).  

Applying \cref{thm:shift},
we see that $\hat{\pi}$ appears in some tiling of the amplituhedron
$\mathcal{A}_{n,k,2}(Z)$.  It follows that 
$\widetilde{Z}$ is injective on $S_{\hat{\pi}}$.
\end{proof}

By \cref{prop:injective} and \cref{labels.map}, collections of $k$ non-intersecting triangles in an $n$-gon label both positroid tiles of $\Delta_{k+1,n}$ and, via T-duality, positroid tiles of $\mathcal{A}_{n,k,2}(Z)$. We conjecture that this labelling is compatible with the way \cite{Lukowski:2019sxw} associates collections of $k$ non-intersecting triangles in an $n$-gon with positroid tiles of $\mathcal{A}_{n,k,2}(Z)$. 

Using 
\cref{prop:injective},
\cref{prop:conn} and \cref{prop:homeo}, we obtain the following.
\begin{corollary} \label{cor:injective}
The $\widetilde{Z}$-map is an injection on all $2k$-dimensional cells 
of the form $S_{\hat{\pi}} \subset Gr^{\geq 0}_{k,n}$, where 
$\pi$ is a SIF permutation and $\dim S_{\pi} = n-1$.
\end{corollary}

We know from \cref{prop:homeo} that the moment map is an injection
on the cell
${S_{\pi}}$ 
 of $Gr^{\geq 0}_{k,n}$ precisely when 
$\dim S_{\pi} = n-c$, 
 where $c$ is the number of connected
components of the positroid of $\pi$.  We have 
experimentally checked the following statement for these cells.

\begin{conj}\label{prop_cell_relation2}
	Let $S_{\pi}$ be a loopless $(n-c)$-dimensional cell of $Gr_{k+1,n}^{\geq 0}$ with $c$ connected components (for $c$ 
	a positive integer).  	Then 
	$S_{\hat{\pi}}$ is a coloopless 
	$(2k+1-c)$-dimensional cell of $Gr_{k,n}^{\geq 0}$ on which 
	$\widetilde{Z}$ is injective.
\end{conj}
Note that the statement that 
	$S_{\hat{\pi}}$ is  coloopless of dimension $(2k+1-c)$
	follows from 
	\cref{lem:hat} and
	\cref{prop:2}.
Moreover the $c=1$ case of the conjecture is \cref{prop:injective}.

While \cref{thm:shift} shows that T-duality relates 
the large class of BCFW
tilings/dissections of $\Delta_{k+1,n}$
to the corresponding large class of 
BCFW tilings/dissections of $\mathcal{A}_{n,k,2}(Z)$, not 
all tilings/dissections arise from a BCFW-style recursion.
Nevertheless, we conjecture the following.

\begin{conj}[Tilings and dissections of $\Delta_{k+1,n}$ and $\mathcal{A}_{n,k,2}(Z)$ are T-dual]\label{conj:dis}
A collection of positroid polytopes $\{\Gamma_{\pi}\}$ is
a tiling (respectively, dissection) of $\Delta_{k+1,n}$ if and only if for all $Z \in \Mat^{>0}_{n,k+2}$
the collection of T-dual Grasstopes $\{Z_{\hat{\pi}}\}$ is a tiling (respectively, dissection) of $\mathcal{A}_{n,k,2}(Z)$.
\end{conj}
This conjecture is supported by \cref{thm:shift}, \cref{prop:largeclass} and results of \cref{sec:parity} (which relates 
parity duality and T-duality), and 
will be explored in a subsequent work\footnote{Since our paper appeared on arXiv, \cref{conj:dis} has been proved for tilings in \cite{Parisi:2021oql}.}.
We have also checked the conjecture using Mathematica,
see \cref{sec:regularA}.

\section{T-duality, cyclic symmetry and parity duality}\label{sec:parity}

In this section we discuss the relation of T-duality to 
 \emph{parity duality}, which relates dissections of the amplituhedron $\mathcal{A}_{n,k,m}(Z)$ with dissections of $\mathcal{A}_{n,n-m-k,m}(Z')$. The definition of parity duality was originally inspired by the physical operation of parity conjugation in quantum field theory -- more specifically, in the context of scattering amplitudes in $\mathcal{N}=4$ Super-Yang-Mills, where amplitudes can 
 be computed from the geometry of $\mathcal{A}_{n,k,4}(Z)$ 
 \cite{arkani-hamed_trnka}. Furthermore, the conjectural formula of Karp, Williams, and Zhang \cite{Karp:2017ouj} for the number of cells in each tiling of the amplituhedron is invariant under the operation of swapping the parameters $k$ and $n-m-k$ and hence is consistent with parity duality: this motivated further works, see 
 \cite[Section 2.4]{Ferro:2018vpf} and \cite{Galashin:2018fri}.  
 In particular, \cite{Galashin:2018fri} gave an explicit bijection 
 between dissections of  $\mathcal{A}_{n,k,m}(Z)$ and 
 dissections of $\mathcal{A}_{n,n-m-k,m}(Z')$, see 
 \cref{th:GLduality}.

In \cref{thm:pardual}, we will explain how parity duality for $m=2$ amplituhedra is naturally induced by a composition of the usual duality for Grassmannians 
($Gr_{k,n} \simeq Gr_{n-k,n}$) and the T-duality map (between loopless
cells of $Gr^{\geq 0}_{k+1,n}$ and coloopless cells of $Gr^{\geq 0}_{k,n}$).
The usual Grassmannian duality gives rise to a bijection between
dissections of the hypersimplex $\Delta_{k+1,n}$ and dissections of the 
hypersimplex $\Delta_{n-k-1,n}$. By composing this Grassmannian duality
with the T-duality map (on both sides), we obtain the parity duality between dissections of $\mathcal{A}_{n,k,2}(Z)$ and $\mathcal{A}_{n,n-k-2,2}(Z')$!

Recall that our convention on dissections is that the 
images of all positroid cells are of full dimension $n-1$.
Therefore all positroids involved in a dissection must be connected,
and the corresponding decorated permutations will be fixed-point-free.

\begin{thm}[Grassmannian duality for dissections of the hypersimplex]\label{thm:Gr} Let $\lbrace \Gamma_{\pi} \rbrace$ be a collection of positroid polytopes which dissects the hypersimplex $\Delta_{k+1,n}$. Then the collection of positroid polytopes $\lbrace \Gamma_{\pi^{-1}}\rbrace$ dissects the hypersimplex $\Delta_{n-k-1,n}$.  
\end{thm}

\begin{proof}
	If $G$ is a plabic graph representing the positroid cell $S_{\pi}$,
 and if we swap the colors of the black and white
vertices of $G$, we obtain a graph $G'$ representing the positroid
	$S_{\pi^{-1}}.$  It is not hard to see from \cite{ARW}
	that $G'$ and $\pi^{-1}$ represent the \emph{dual positroid} to 
	$G$ and $\pi$.  
	But now the matroid polytopes $\Gamma_{\pi}$ and $\Gamma_{\pi^{-1}}$
	are isomorphic via the map $\dual: \R^n \to \R^n$
	sending $(x_1,\dots,x_n) \mapsto (1-x_1,\dots,1-x_n)$.
	This maps relates the two dissections in the statement of the 
	theorem.
\end{proof}
By composing the inverse map on decorated permutations 
$\pi \mapsto \pi^{-1}$ (which represents the Grassmannian
duality of 
 \cref{thm:Gr})
with T-duality, we obtain the following map.

\begin{definition} \label{def:dualmapm2}
We define $\U$ to be the map between coloopless permutations on $[n]$  with $k$ anti-excedances and coloopless permutations on $[n]$  with $n-k-2$ anti-excedances such that 
	$\U \hat{\pi}=\widehat{\pi^{-1}}$.
Equivalently, we have 
	$(\U\pi)(i)=\pi^{-1}(i-1)-1$, 
	  where 
	 values of the permutation are considered modulo $n$,
	 and any 
	 fixed points which are created are designated to be loops.
\end{definition}

\begin{thm}[Parity duality from T-duality and Grassmannian duality]\label{thm:pardual} Let $\lbrace Z_{\pi} \rbrace$ be a collection of Grasstopes which dissects the amplituhedron $\mathcal{A}_{n,k,2}(Z)$. Then the collection of Grasstopes $\lbrace Z_{\U\pi}\rbrace$ dissects the amplituhedron $\mathcal{A}_{n,n-k-2,2}(Z')$. 
\end{thm}
We will prove \cref{thm:pardual} by using the cyclic symmetry of the 
positive Grassmannian and the amplituhedron,
 and showing (see \cref{lem:shiftlam}) that up to a cyclic shift, 
our map $\U$ agrees with the parity duality map of 
\cite{Galashin:2018fri}.

The totally nonnegative Grassmannian exhibits
a beautiful \emph{cyclic symmetry} \cite{postnikov}.
Let us represent an element of $Gr^{\geq 0}_{k,n}$ by a $k \times n$ matrix,
encoded by the sequence of $n$ columns
$\langle v_1,\ldots, v_n\rangle$. 
We define  the \emph{(left) cyclic shift map} $\sigma$
to be the map which sends
$\langle v_1,\ldots, v_n\rangle$ 
to the point $\langle v_2,\ldots, v_n,(-1)^{k-1} v_1\rangle$, which 
one can easily verify lies in $Gr^{\geq 0}_{k,n}.$ 
Since the cyclic shift maps positroid cells to positroid cells,
for $\pi$ a decorated permutation, we define $\sigma \pi$ to be
the decorated permutation such that 
 $S_{\sigma\pi} = \sigma(S_\pi)$. 
 It is easy to see that $\sigma\pi(i)=\pi(i+1)-1$.
 (Note that under the cyclic shift, a fixed point of $\pi$ at position $i+1$
 gets sent to a fixed point of $\sigma \pi$ at position $i$; we 
 color fixed points accordingly.)
 Meanwhile the inverse operation, the \emph{right cyclic shift}
 $\sigma^{-1}$ satisfies $(\sigma^{-1}\pi)(i) = \pi(i-1)+1$.
We use $\sigma^t$ (respectively, $\sigma^{-t}$) 
to denote the repeated application of $\sigma$  (resp. $\sigma^{-1}$) $t$ 
times,
so that $(\sigma^t\pi)(i):=\pi(i+t)-t$ and  
 $(\sigma^{-t}\pi)(i):=\pi(i-t)+t$.

The next result follows easily from the definitions.
\begin{theorem}[Cyclic symmetry for dissections of 
	the hypersimplex] \label{th:cyclicsymHS}
Let $\lbrace \Gamma_{\pi} \rbrace$ be a collection of positroid polytopes which dissects the hypersimplex $\Delta_{k+1,n}$. Then the collection of positroid polytopes $\lbrace \Gamma_{\sigma\pi}\rbrace$ dissects $\Delta_{k+1,n}$.
\end{theorem}
\begin{proof}
Let $\sigma_\R: \R^n \to \R^n$ be defined by 
$(x_1,\dots,x_n) \mapsto (x_2,\dots,x_n, x_1)$.  
Clearly $\sigma_\R$ is an isomorphism mapping the hypersimplex 
$\Delta_{k+1,n}$ back to itself.
Moreover, 
applying the cyclic shift $\sigma$ to a positroid has the effect
of simply shifting all its bases, so the matroid polytope 
of $\sigma \pi$ satisfies 
$\Gamma_{\sigma \pi} = \sigma_{\R}(\Gamma_{\pi})$.
The result now follows.
\end{proof}

The above cyclic symmetry for dissections of the hypersimplex
also has an analogue for the amplituhedron.

\begin{theorem}[Cyclic symmetry for dissections of 
	the amplituhedron]\cite[Corollary 3.2]{BaoHe}\label{th:cyclicsym}
Let $\lbrace \Gamma_{\pi} \rbrace$ be a collection of Grasstopes which dissects the amplituhedron $\mathcal{A}_{n,k,m}(Z)$, with $m$ even. Then 
	 the collection of Grasstopes $\lbrace Z_{\sigma\pi}\rbrace$ also dissects $\mathcal{A}_{n,k,m}(Z)$.
\end{theorem}

In order to make contact with \cite{Galashin:2018fri}, 
we introduce a map $U_{k,n}$ on (coloopless) decorated permutations as follows.
\begin{definition}
	We define $U_{k,n}$ to be the map from
	coloopless permutations on $[n]$  with $k$ anti-excedances to coloopless permutations on $[n]$  with $n-k-2$ anti-excedances such that 
	$(U_{k,n}\pi)(i) = \pi^{-1}(i+k)+(n-k-2)$,
	  where 
	 values of the permutation are considered modulo $n$,
	 and any 
	 fixed points which are created are designated to be loops.
\end{definition}
It is not hard to see that  this map is equivalent to the parity duality from \cite{Galashin:2018fri} for $m=2$. In particular we have the following theorem:

\begin{theorem}{\cite[Theorem~7.2]{Galashin:2018fri}}\label{th:GLduality}
Let $\lbrace Z_{\pi} \rbrace$ be a collection of Grasstopes which dissects the amplituhedron $\mathcal{A}_{n,k,2}(Z)$. Then the collection of Grasstopes $\lbrace Z_{U_{k,n}\pi}\rbrace$ dissects the amplituhedron $\mathcal{A}_{n,n-k-2,2}(Z')$.
\end{theorem}

\begin{lemma}\label{lem:shiftlam} For fixed $n$ and $k$, the maps $\U$ and $U_{k,n}$ are related by the cyclic shift map
	\begin{equation}
		\U=\sigma^{-(k+1)} \circ U_{k,n}.
\end{equation}
\end{lemma}

\begin{proof}
Since $(U_{k,n}\pi)(i) = \pi^{-1}(i+k)+(n-k-2)$, we have that
	$(\sigma^{-(k+1)}\circ U_{k,n}\pi)(i)=\pi^{-1}((i+k)-(k+1))+(n-k-2)+(k+1)=\pi^{-1}(i-1)+n-1$, which is exactly $\U$ (mod $n$). 
\end{proof}

We now prove \cref{thm:pardual}.
\begin{proof}
This result follows immediately from \cref{th:cyclicsym},      
 \cref{th:GLduality},  
 and \cref{lem:shiftlam}.  
\end{proof}

\begin{remark}
From 
	\cref{th:cyclicsymHS} 
	and \cref{th:cyclicsym} 
	it is clear that if we redefine the T-duality map in \cref{hatmap} by composing it with any cyclic shift $\sigma^a$ (for $a$ an integer),
	the main properties of the map will be preserved. 
	In particular, any statement about dissections of the hypersimplex versus the corresponding ones of the amplituhedron will continue to hold,
	along with the parity duality.   
\end{remark}

\begin{remark} Parity duality has a nice graphical interpretation when we represent positroid tiles of $\mathcal{A}_{n,k,2}(Z)$ as collection of $k$ non-intersecting triangles in an $n$-gon. The Grassmannian duality of $Gr^{\geq 0}_{k+1,n}$ amounts to swapping black and white vertices in the plabic graphs, and when we compose it with the T-duality map, by \cref{labels.map}, results in taking the complementary polygons inside the $n$-gon. We end up with a collection of $n-k-2$ non-intersecting triangles in the $n$-gon.
\end{remark}

\section{Good and bad dissections of the hypersimplex and the amplituhedron}\label{sec:good}

Among all possible positroid dissections, there are some with particularly nice features, which we will call ``good'', as well as others with rather unpleasant properties. We show below examples of 
both a good and a bad dissection.
\begin{example}
Let us study the following tiling of $\mathcal{A}_{6,2,2}(Z)$:
\begin{equation*}
\mathcal{C}_1=\left\{S_{\pi^{(1)}},S_{\pi^{(2)}},S_{\pi^{(3)}},S_{\pi^{(4)}},S_{\pi^{(5)}},S_{\pi^{(6)}},\right\}
\end{equation*}
with
\begin{align*}
\pi^{(1)}=(\underline{1},\underline{2},5,6,3,4),\quad\pi^{(2)}=(\underline{1},3,6,5,2,4),\quad\pi^{(3)}=(\underline{1},4,6,2,\underline{5},3)\,,\\
\pi^{(4)}=(2,6,\underline{3},5,1,4),\quad\pi^{(5)}=(2,6,4,1,\underline{5},3),\quad\pi^{(6)}=(3,6,1,4,\underline{5},2)\,.
\end{align*}
All elements of $\mathcal{C}_1$ are positroid tiles and their images under
	$\tilde Z$ are 4-dimensional.
The tiling $\mathcal{C}_1$ is a refinement of the following dissection
\begin{equation*}
\mathcal{C}_2=\left\{S_{\pi^{(1)}},S_{\pi^{(7)}},S_{\pi^{(8)}},S_{\pi^{(6)}}\right\}
\end{equation*}
with
\begin{equation*}
\pi^{(7)}=(1,4,6,5,2,3)\,,\qquad \pi^{(8)}=(2,6,4,5,1,3)\,.
\end{equation*}
The dissection $\mathcal{C}_2$ has the property that if a pair of cell images under $\tilde Z$-map intersect along a $3$-dimensional surface then this surface is an image of another positroid cell in $Gr^{\geq 0}_{2,6}$:
\begin{align*}
 Z_{\pi^{(1)}}\cap  Z_{\pi^{(7)}}=Z_{(\underline{1},\underline{2},6,5,3,4)}\\
 Z_{\pi^{(7)}}\cap  Z_{\pi^{(8)}}=  Z_{(\underline{1},6,4,5,2,3)}\\
  Z_{\pi^{(8)}}\cap  Z_{\pi^{(6)}}=  Z_{(2,6,1,\underline{4},\underline{5},3)}
\end{align*}
and all remaining pairs of images intersect along lower dimensional surfaces. We consider the dissection $\mathcal{C}_2$ ``good'' because all its elements have compatible codimension one boundaries. However, the dissection $\mathcal{C}_1$ does not have this property. Let us observe that 
\begin{align*}
Z_{\pi^{(2)}}\cup Z_{\pi^{(3)}}=Z_{\pi^{(7)}}\\
 Z_{\pi^{(4)}}\cup  Z_{\pi^{(5)}}= Z_{\pi^{(8)}}
\end{align*}
We expect that, after we subdivide  $Z_{\pi^{(7)}}$ and $ Z_{\pi^{(8)}}$, the boundary $Z_{(\underline{1},6,4,5,2,3)}$ which they share will also get subdivided. This however happens in two different ways and we do not get compatible codimension one faces for the dissection $\mathcal{C}_1$. It is a similar picture to the one we get when we consider polyhedral subdivisions of a double square pyramid: it is possible to subdivide it into two pieces along its equator, and then further subdivide each pyramid into two simplices. However, in order to get a polyhedral triangulation of the double square pyramid, we need to do it in a compatible way, along the same diagonal of the equatorial square.
\end{example}
Therefore, we prefer to work with dissections where the boundaries of the strata
interact nicely.  Towards this end, we introduce the following notion of 
\emph{good dissection}. 

\begin{definition}\label{def:good_dissection}
Let $\mathcal{C}=\{\Gamma_{\pi^{(1)}}, \dots, \Gamma_{\pi^{(\ell)}}\}$ 
be a dissection of $\Delta_{k+1,n}$.
We say that $\mathcal{C}$
is a \emph{good dissection} of $\Delta_{k+1,n}$ if the following condition is satisfied: for $i\neq j$, if
			$\Gamma_{\pi^{(i)}} \cap 
			\Gamma_{\pi^{(j)}}$ has
			codimension one,
			 then 
			$\Gamma_{\pi^{(i)}} \cap 
			\Gamma_{\pi^{(j)}}$ 
			equals 
			$\Gamma_{\pi}$, 
			where 
			$\Gamma_{\pi}$ is a facet of both 
			$\Gamma_{\pi^{(i)}}$ and 
			$\Gamma_{\pi^{(j)}}$.
\end{definition}

Note that the above condition is equivalent to requiring that 
$\mathcal{C}$ is a \emph{polyhedral subdivision} of $\Delta_{k+1,n}$.
To make the analogous notion for amplituhedron, we need to define 
facets.
\begin{definition}
Let $Z_{\pi}
\subset \mathcal{A}_{n,k,m}(Z)$ 
	 be a Grasstope. 
We say that $Z_{\pi'}$ is a \emph{facet} of $Z_{\pi}$ if 
it is maximal by inclusion among the Grasstopes satisfying
	the following properties:  the cell $S_{\pi'}$ is contained
	in $\overline{S_{\pi}}$; 
	$Z_{\pi'}$ is contained in the boundary of $Z_{\pi}$; 
	$Z_{\pi'}$ has codimension $1$ in $Z_{\pi}$.
\end{definition}
\begin{definition}\label{def:good_dissection_amp}
	Let $\mathcal{C}=\{Z_{\pi^{(1)}}, \dots, Z_{\pi^{(\ell)}}\}$ 
	be a collection of Grasstopes of $\mathcal{A}_{n,k,2}(Z)$.
	We say that $\mathcal{C}$
	is a \emph{good dissection} of $\mathcal{A}$ if the following condition is satisfied: for $i\neq j$, if
			$Z_{\pi^{(i)}} \cap 
			Z_{\pi^{(j)}}$ has
			codimension one,
			 then 
			$Z_{\pi^{(i)}} \cap 
			Z_{\pi^{(j)}}$ 
			equals 
			$Z_{\pi}$, 
			where $Z_{\pi}$ is a facet of both 
			$Z_{\pi^{(i)}}$ and 
			$Z_{\pi^{(j)}}$.
\end{definition}

In the following, we will conjecture that good dissections of the
hypersimplex are in one-to-one correspondence with good dissections of the amplituhedron. Towards this goal, we start by providing a characterization of good intersections of positroid polytopes.
\begin{proposition}\label{prop_hyp1}
Let $\Gamma_{\pi^{(1)}}$ and $\Gamma_{\pi^{(2)}}$ be two $(n-1)$-dimensional
positroid polytopes whose 
intersection 
$\Gamma_{\pi^{(1)}}\cap\Gamma_{\pi^{(2)}}$ is a polytope of 
dimension $n-2$.  Then 
$\Gamma_{\pi^{(1)}}\cap\Gamma_{\pi^{(2)}}$ is a positroid
polytope of the form  $\Gamma_{\pi^{(3)}}$,
where 
$\pi^{(3)}$ is a loopless permutation.
\end{proposition}

\begin{proof}
By \cref{thm:faces}, 
$\Gamma_{\pi^{(1)}}\cap\Gamma_{\pi^{(2)}}$ is a positroid polytope and hence has
the form  $\Gamma_{\pi^{(3)}}$, for some decorated permutation
	$\pi^{(3)}$.  (Using \cref{prop:dim}, the fact that 
$\dim(\Gamma_{\pi^{(3)}})=n-2$ implies that the positroid
	associated to $\pi^{(3)}$ has precisely two connected components.)

Now we claim that the positroid associated to $\pi^{(3)}$ is loopless.
In general there is an easy geometric way of recognizing when a matroid $M$
is loopless from the polytope $\Gamma_M$: $M$ is loopless if and only
if $\Gamma_M$ is not contained in any of the $n$ facets of the 
hypersimplex of the type $x_i = 0$ for $1 \leq i \leq n$.
	Since $\Gamma_{\pi^{(3)}}$ arises as the
	codimension $1$ intersection of two full-dimensional
	matroid polytopes contained in $\Delta_{k+1,n}$ it necessarily
	meets the interior of the hypersimplex and hence 
	the matroid must be loopless.
\end{proof}

\begin{remark}
	Recall that the T-duality map is well-defined on positroid
	cells whose matroid is connected, and more generally, loopless.  
	\cref{prop_hyp1} implies that if we consider two 
	cells 
	$S_{\pi^{(1)}}$ and $S_{\pi^{(2)}}$
	of $Gr^{\geq 0}_{k+1,n}$ whose matroid is connected
	and whose moment map images (necessarily top-dimensional) intersect
	in a common facet, then that facet is the moment map 
	image of a loopless cell $S_{\pi^{(3)}}$.  Therefore  
	we can apply the 
	T-duality map to all three cells 
	$S_{\pi^{(1)}}$, $S_{\pi^{(2)}}$,
	and $S_{\pi^{(3)}}$.
\end{remark}

\begin{conj}\label{prop_amp1}
Let $S_{\pi^{(1)}}$ and $S_{\pi^{(2)}}$ be two positroid cells in $Gr^{\geq 0}_{k,n}$ corresponding to coloopless permutations $\pi^{(1)}$ and $\pi^{(2)}$. Let $\dim Z^\circ_{\pi^{(1)}}=\dim Z^\circ_{\pi^{(2)}}=2k$ with $Z_{\pi^{(1)}}\cap Z_{\pi^{(2)}}= Z_{\pi^{(3)}}$, where $S_{\pi^{(3)}}\subset G^{\geq 0}_{k,n}$ is such that $\dim Z^\circ_{\pi^{(3)}}=2k-1$. Then $\pi^{(3)}$ is a coloopless permutation.
\end{conj}

\begin{remark}
 \cref{prop_amp1} guarantees that if we consider two positroid cells with top-dimensional images in the amplituhedron $\mathcal{A}_{n,k,2}(Z)$, which have a facet in common, then the positroid cell corresponding to this facet is coloopless and therefore we can apply the T-duality map to it.
\end{remark}

Finally we arrive at a conjecture connecting good dissections of hypersimplex and amplituhedron, which we confirmed experimentally. 
\begin{conj}\label{prop:subd}
The collection of positroid polytopes $\{\Gamma_{\pi}\}$ is a good tiling 
	(respectively, good dissection) of $\Delta_{k+1,n}$ if and only if, for all $Z \in \Mat^{>0}_{n,k+2}$,
the collection of T-dual Grasstopes $\{Z_{\hat{\pi}}\}$ is a good
	tiling (respectively, good dissection) of $\mathcal{A}_{n,k,2}(Z)$.
\end{conj}

\section{The positive tropical Grassmannian and positroid subdivisions}\label{sec:tropical}

The goal of this section is to use the positive tropical Grassmannian to 
understand the regular positroid
subdivisions of the hypersimplex. 
In \cref{sec:regularA}, we will apply the T-duality map to these regular
positroid subdivisions of the hypersimplex, to obtain subdivisions
of the amplituhedron which have very nice properties.

The \emph{tropical Grassmannian} -- or rather, an outer approximation
of it called the \emph{Dressian} -- 
 controls the regular matroidal subdivisions of the hypersimplex
\cite{Kapranov}, 
 \cite[Proposition 2.2]{Speyer}.
There is a positive subset of the tropical Grassmannian, 
 called the \emph{positive tropical Grassmannian}, which was 
introduced by Speyer and the third author in \cite{troppos}.
The positive tropical Grassmannian equals the positive Dressian,
and as we will show in \cref{prop:positroid},
it controls the regular \emph{positroid} subdivisions of the  
hypersimplex.

\begin{remark} We've learned since circulating 
the first draft of this paper that some of our results
in this section regarding positroid subdivisions of the hypersimplex and 
the positive tropical Grassmannian, though not previously in the literature, 
were known or anticipated by various other experts including 
David Speyer, Nima Arkani-Hamed, Thomas Lam, Marcus Spradlin, Nick Early, Felipe Rincon, Jorge Olarte.
There is some related work in \cite{Early} and the upcoming \cite{AHLS}.
\end{remark}

\subsection{The tropical Grassmannian, the Dressian, and their positive
analogues}

\begin{definition}\label{def:trophyper}
	Given $e=(e_1,\dots,e_N) \in \Z^N_{\geq 0}$, we let 
	$\mathbf{x}^e$ denote $x_1^{e_1} \dots x_N^{e_N}$.  
	Let  $E \subset \Z^N_{\geq 0}$.
	For $f = \sum_{e\in E} f_e \x^e$ a nonzero polynomial, 
	we denote by 
	$\Trop(f) \subset \R^N$ the set of all points 
	$(X_1,\dots, X_N)$ such that, if we form the collection of 
	numbers $\sum_{i=1}^N e_i X_i$ for $e$ ranging over $E$, then
	the minimum of this collection is not unique.
	We say that $\Trop(f)$ is the \emph{tropical hypersurface associated to $f$}.
\end{definition}

In our examples, we always consider polynomials $f$ with real coefficients.
We also have a positive version of \cref{def:trophyper}.

\begin{definition}
	Let  $E=E^+ \sqcup E^- \subset \Z^N_{\geq 0}$, and let
	 $f$ be a nonzero polynomial with real
	coefficients which we write as 
	$f = \sum_{e\in E^+} f_e \x^e - \sum_{e\in E^-} f_e \x^e$,
	where all of the coefficients $f_e$ are nonnegative real numbers.
	We  denote by 
	$\Trop^+(f) \subset \R^N$ the set of all points 
	$(X_1,\dots, X_N)$ such that, if we form the collection of 
	numbers $\sum_{i=1}^N e_i X_i$ for $e$ ranging over $E$, then
	the minimum of this collection is not unique and furthermore
	is achieved for some $e\in E^+$ and some $e\in E^-$.
	We say that $\Trop^+(f)$ is the \emph{positive part of 
	$\Trop(f)$.}
\end{definition}

The Grassmannian $Gr_{k,n}$ is a projective variety which can be embedded
in projective space $\PP^{\binom{[n]}{k}-1}$, and is cut out by the 
\emph{Pl\"ucker ideal}, that is, the ideal of relations satisfied by
the Pl\"ucker coordinates of a generic $k \times n$ matrix.
These relations include
the \emph{three-term Pl\"ucker relations} defined
below.

\begin{definition}\label{def:3}
Let $1<a<b<c<d\leq n$ 
and choose a subset $S \in \binom{[n]}{k-2}$ which is disjoint from $\{a,b,c,d\}$.  
Then $p_{Sac} p_{Sbd} = p_{Sab} p_{Scd}+p_{Sad} p_{Sbc}$ is 
a \emph{three-term Pl\"ucker relations} for the Grassmannian $Gr_{k,n}$.
Here $Sac$ denotes $S \cup \{a,c\}$, etc.
\end{definition}

\begin{definition}\label{rem:tropPlucker}
Given $S, a, b, c, d$ as in \cref{def:3}, 
we say that the \emph{tropical three-term
Pl\"ucker relation holds} if 
\begin{itemize}
	\item $P_{Sac}+P_{Sbd} = P_{Sab}+P_{Scd} \leq P_{Sad}+P_{Sbc}$ 
		or 
	\item $P_{Sac}+P_{Sbd} = 
P_{Sad}+P_{Sbc} \leq 
P_{Sab}+P_{Scd}$
or 
\item $P_{Sab}+P_{Scd} = P_{Sad}+P_{Sbc}
\leq	P_{Sac}+P_{Sbd}$.
	\end{itemize}
	And we say that the \emph{positive tropical three-term
Pl\"ucker relation holds} if either of the first two conditions above holds.
\end{definition}

\begin{definition}
	The \emph{tropical Grassmannian} $\Trop Gr_{k,n} \subset \R^{\binom{[n]}{k}}$ is 
	the intersection of 
	the tropical hypersurfaces $\Trop(f)$, where $f$ ranges over all
	elements of the Pl\"ucker ideal.
	The \emph{Dressian} $\Dr_{k,n} \subset 
	 \R^{\binom{[n]}{k}}$ is the intersection of 
	the tropical hypersurfaces $\Trop(f)$, where $f$ ranges over all
	three-term Pl\"ucker relations.

Similarly, 
	the \emph{positive tropical Grassmannian} $\Trop^+Gr_{k,n} \subset \R^{\binom{[n]}{k}}$ is the intersection of 
	the positive tropical hypersurfaces $\Trop^+(f)$, 
	where $f$ ranges over all
	elements of the Pl\"ucker ideal.
	The \emph{positive Dressian} $\Dr^+_{k,n} \subset 
	 \R^{\binom{[n]}{k}}$ is the intersection of 
	the positive tropical hypersurfaces $\Trop^+(f)$, 
	where $f$ ranges over all
	three-term Pl\"ucker relations.
\end{definition}

Note that the Dressian $Dr_{k,n}$ (respectively, the positive Dressian $Dr^+_{k,n}$) 
is the subset of $\R^{[n]\choose k}$
	where the tropical (respectively, positive tropical) three-term Pl\"ucker relations hold. 

In general, the Dressian $\Dr_{k,n}$ is much larger than the 
tropical Grassmannian $\Trop Gr_{k,n}$ -- for example,
the dimension of the Dressian $\Dr_{3,n}$ grows quadratically is $n$,
while the dimension of the tropical Grassmannian $\Trop Gr_{3,n}$ is 
linear in $n$ \cite{Herrmann2008HowTD}.  
However, the situation for their positive parts is different.

\begin{theorem}\cite{SW}.
The positive tropical Grassmannian $\Trop^+Gr_{k,n}$ equals
the positive Dressian $\Dr^+_{k,n}$.
\end{theorem}

\begin{definition}
	We say that a point $\{P_I\}_{I\in \binom{[n]}{k}}\in \R^{\binom{[n]}{k}}$
	is a \emph{(finite) tropical Pl\"ucker vector} 
	if it lies in the Dressian $\Dr_{k,n}$, i.e. 
	for every three-term Pl\"ucker relation, it lies in 
	the associated tropical hypersurface.
	And we say that $\{P_I\}_{I\in \binom{[n]}{k}}$
	is a \emph{positive tropical Pl\"ucker vector},  
	if it lies in the positive Dressian $\Dr^+_{k,n}$ (equivalently, the 
	positive tropical Grassmannian $\Trop^+Gr_{k,n}$), i.e. 
	 for every three-term Pl\"ucker relation, it lies in 
	the positive part of the associated tropical hypersurface.
\end{definition}

\begin{example}
	For $Gr_{2,4}$, there is only one Pl\"ucker relation, 
	$p_{13} p_{24} = p_{12} p_{34}+p_{14} p_{23}$.  
	The Dressian $\Dr_{2,4} \subset \R^{\binom{[4]}{2}}$ is defined to 
	be the set of points
	$(P_{12}, P_{13}, P_{14}, P_{23}, P_{24}, P_{34})\in \R^6$ such that 
	\begin{itemize}
		\item $P_{13}+P_{24} = P_{12}+P_{34} \leq P_{14}+P_{23}$ 
			or 
		\item $P_{13}+P_{24} = 
	P_{14}+P_{23} \leq 
	P_{12}+P_{34}$
	or 
\item $P_{12}+P_{34} = P_{14}+P_{23}
\leq	P_{13}+P_{24}$.
	\end{itemize}
	And $\Dr^+_{2,4} = \Trop^+Gr_{2,4} \subset \R^{\binom{[4]}{2}}$ is defined to 
	be the set of points
	$(P_{12}, P_{13}, P_{14}, P_{23}, P_{24}, P_{34})\in \R^6$ such that 
	\begin{itemize}
		\item $P_{13}+P_{24} = P_{12}+P_{34} \leq P_{14}+P_{23}$ 
			or 
		\item $P_{13}+P_{24} = 
	P_{14}+P_{23} \leq 
	P_{12}+P_{34}$
	\end{itemize}
\end{example}

\subsection{The positive tropical Grassmannian and positroid subdivisions}\label{sec:DP}

Recall that $\Delta_{k,n}$ denotes the $(k,n)$-hypersimplex, defined as the 
convex hull of the points $e_I$ where $I$ runs over $\binom{[n]}{k}$.
Consider a real-valued function $\{I\} \mapsto P_I$ on the vertices
of $\Delta_{k,n}$.  We define a polyhedral subdivision $\mathcal{D}_P$
of $\Delta_{k,n}$ as follows: consider the points
$(e_I, P_I)\in \Delta_{k,n} \times \R$ and take their convex hull. 
Take the lower faces (those whose outwards normal vector have last component
negative) and project them back down to $\Delta_{k,n}$; this gives us 
the subdivision $\mathcal{D}_P$.  We will omit the subscript $P$ when it is clear from context.  A subdivision obtained in this manner is called \emph{regular}.  

\begin{remark}\label{rem:lower}
A lower face $F$ of the regular subdivision defined above is determined by some 
	vector $\lambda = (\lambda_1,\dots,\lambda_n,-1)$ whose dot product with 
	the vertices of the face $F$ is maximized.
	So if $F$ is the matroid polytope of a matroid $M$ with bases $\mathcal{B}$, 
	this is equivalent to saying that 
	$\lambda_{i_1} + \dots + \lambda_{i_k}-P_I = \lambda_{j_1} + \dots + \lambda_{j_k}-P_J >
	\lambda_{h_1} + \dots + \lambda_{h_k} - P_H$ for any two bases $I,J \in \mathcal{B}$
	and $H\notin \mathcal{B}$.
\end{remark}

Given a subpolytope $\Gamma$ of $\Delta_{k,n}$, we say that $\Gamma$ is \emph{matroidal}
if the vertices of $\Gamma$, considered as elements of $\binom{[n]}{k}$, are the 
bases of a matroid $M$, i.e. $\Gamma = \Gamma_M$.

The following result is originally due to Kapranov \cite{Kapranov}; it was 
also 
proved in \cite[Proposition 2.2]{Speyer}.
\begin{theorem}\label{prop:K}
	The following are equivalent.
	\begin{itemize}
		\item The collection $\{P_I\}_{I \in \binom{[n]}{k}}$ 
			is a tropical Pl\"ucker vector.
		\item The one-skeleta of $\mathcal{D}_P$ and $\Delta_{k,n}$
			are the same.
		\item Every face of $\mathcal{D}_P$ is matroidal.
	\end{itemize}
\end{theorem}

Given a subpolytope $\Gamma$ of $\Delta_{k,n}$, we say that $\Gamma$ is \emph{positroid}
if the vertices of $\Gamma$, considered as elements of $\binom{[n]}{k}$, are the 
bases of a positroid $M$, i.e. $\Gamma = \Gamma_M$.
We now give a positroid version of \cref{prop:K}.

\begin{theorem}\label{prop:positroid}
The following are equivalent.
\begin{itemize}
	\item The collection $\{P_I\}_{I \in \binom{[n]}{k}}$ 
		is a positive tropical Pl\"ucker vector.
	\item Every face of $\mathcal{D}_P$ is positroid.
\end{itemize}
\end{theorem}

\begin{proof}
Suppose that 
the collection $\{P_I\}_{I \in \binom{[n]}{k}}$ 
	are positive tropical Pl\"ucker coordinates.
	Then in particular they are tropical Pl\"ucker coordinates,
	and so by \cref{prop:K}, 
	every face of $\mathcal{D}_P$ is matroidal.  

	Suppose that one of those
	faces $\Gamma_M$ fails to be positroid.  
Then by \cref{thm:char}, 
	$\Gamma_M$ (and hence $\mathcal{D}_P$) has
	 a two-dimensional face with vertices
		$e_{Sab}, e_{Sad},
		e_{Sbc}, e_{Scd}$, for some $1 \leq a<b<c<d\leq n$ and $S$ 
		of size $k-2$ disjoint from $\{a,b,c,d\}$.
		By \cref{rem:lower}, this means that 
		there is a vector $\lambda=(\lambda_1,\dots,\lambda_n,-1)$ whose
		dot product is maximized at the face $F$.  In particular, 
		if we compare the value of the dot product at
		vertices of $F$ versus $e_{Sac}$ and $e_{Sbd}$, we get
	$\lambda_a+\lambda_b-P_{Sab} = \lambda_c+\lambda_d -P_{Scd} = 
	\lambda_a+\lambda_d-P_{Sad} = \lambda_b+\lambda_c-P_{Sbc}$ 
	is greater than either 
	$\lambda_a+\lambda_c-P_{Sac}$ or $\lambda_b+\lambda_d-P_{Sbd}$.
But then 
	$$\lambda_a+\lambda_b-P_{Sab} + \lambda_c+\lambda_d-P_{Scd} = 
	\lambda_a+\lambda_d-P_{Sad} + \lambda_b+\lambda_c-P_{Sbc} >
	\lambda_a+\lambda_c-P_{Sac}+\lambda_b+\lambda_d-P_{Sbd},$$ which 
	implies that 
	$$P_{Sab}+P_{Scd} = P_{Sad}+P_{Sbc} < P_{Sac}+P_{Sbd},$$
	which contradicts the fact that $\{P_I\}$ is a collection of 
	positive tropical Pl\"ucker coordinates.

Suppose that every 
	 face of 
	 $\mathcal{D}_P$ 
	 is positroid.  Then every face is in particular
	 matroidal, and so  by \cref{prop:K},
	the collection $\{P_I\}_{I \in \binom{[n]}{k}}$ 
		are tropical Pl\"ucker coordinates.  Suppose that they fail
		to be positive tropical Pl\"ucker coordinates.
	Then there is some $S\in \binom{[n]}{k-2}$ and $a<b<c<d$
	disjoint from $S$ such that 
	$P_{Sab}+P_{Scd} = P_{Sad}+P_{Sbc} < P_{Sac}+P_{Sbd}$.
	We will obtain a contradiction by showing that $\mathcal{D}_P$ 
	has a two-dimensional (non-positroid) face 
	with vertices
		$e_{Sab}, e_{Sad},
		e_{Sbc}, e_{Scd}$, for some $1 \leq a<b<c<d\leq n$ and $S$ 
		of size $k-2$ disjoint from $\{a,b,c,d\}$.

To show that these vertices form a face, choose some large number 
	$N$ which is greater than the absolute value of 
	any of the tropical Pl\"ucker coordinates, i.e. 
	$N> \max\{|P_I|\}_{I\in \binom{[n]}{k}}$.
	We define a vector $(\lambda_1,\dots,\lambda_n) \in \R^n$ by setting 
	$$\lambda_i = \begin{cases}
		\frac{1}{2}(	P_{Sab}+P_{Sac}+P_{Sad}) &\text{ for i=a}\\
		\frac{1}{2}(P_{Sab}+P_{Sbc}+P_{Sbd}) &\text{ for i=b}\\
		\frac{1}{2} (P_{Sac}+P_{Sbc}+P_{Scd}) &\text{ for i=c}\\
		\frac{1}{2} (P_{Sad}+P_{Sbd}+P_{Scd}) &\text{ for i=d}\\
		\frac{3}{2} N &\text{ for }i\in S\\
		-\frac{3}{2} N &\text{ for }i\notin S\cup \{a,b,c,d\}.
	\end{cases}
	$$
We now compute the lower face of $\mathcal{D}_P$ 
determined by vector $\lambda:=(\lambda_1,\dots,\lambda_n,-1)$,
using \cref{rem:lower}. 
Clearly any point $(e_I, P_I)$ of $\Delta_{k,n} \times \R$ 
maximizing the dot product with 
$\lambda$
must have 
$e_I \in \{e_{Sab}, e_{Sac}, e_{Sad}, e_{Sbc}, e_{Sbd}, e_{Scd}\}$.
The relation 
$P_{Sab}+P_{Scd} = P_{Sad}+P_{Sbc} < P_{Sac}+P_{Sbd}$
implies that the lower face of 
 $\mathcal{D}_P$ determined by $\lambda$ has vertices
		$e_{Sab}, e_{Sad},
		e_{Sbc}, e_{Scd}$.   
\end{proof}

It follows from \cref{prop:positroid}
that the regular subdivisions of $\Delta_{k+1,n}$ consisting of positroid polytopes
are precisely those of the form $\mathcal{D}_P$, where $P=\{P_I\}$ is a positive tropical
Pl\"ucker vector. This motivates the following definition.

\begin{definition}\label{def:regularpos}
We say that a positroid dissection of $\Delta_{k+1,n}$
is a \emph{regular positroid subdivision} if it has the form $\mathcal{D}_P$, 
where $P = \{P_I\} \in \R^{[n] \choose k}$ is a positive tropical Pl\"ucker vector.
\end{definition}

\begin{remark}\label{rem:good}
	Every regular subdivision of a polytope is a polytopal subdivision, 
and so in particular it is a good dissection 
(see \cref{def:good_dissection}).
\end{remark}

\subsection{Fan structures on the Dressian and positive 
Dressian}\label{sec:fan}

As described in \cite{Herrmann2008HowTD},
there are two natural fan structures on the (positive) Dressian: the \emph{Pl\"ucker fan},
and the \emph{secondary fan.} 

We say that two elements of the Dressian, i.e. two 
tropical Pl\"ucker
vectors
$\{P_I\}_{I\in \binom{[n]}{k}}$
and $\{P'_I\}_{I\in \binom{[n]}{k}}
	\in \R^{\binom{[n]}{k}}$, lie in the same cone of the \emph{Plucker fan}
	if for each 
$S, a, b, c, d$ as in \cref{rem:tropPlucker}, 
 the same inequality holds for 
 both $\{P_{Sac}, P_{Sbd}, P_{Sab}, P_{Scd},P_{Sad}, P_{Sbc}\}$ and 
 $\{P'_{Sac}, P'_{Sbd}, P'_{Sab}, P'_{Scd},P'_{Sad}, P'_{Sbc}\}$.
	In particular, the maximal cones in the Pl\"ucker fan structure are the cones where the 
inequalities from \cref{rem:tropPlucker} are all strict.

On the other hand, using  
\cref{prop:K}  and 
\cref{prop:positroid},
we say that 
 two elements of the Dressian, i.e. two 
tropical Pl\"ucker
vectors
$\{P_I\}_{I\in \binom{[n]}{k}}$
and $\{P'_I\}_{I\in \binom{[n]}{k}}
\in \R^{\binom{[n]}{k}}$, lie in the same cone of the \emph{secondary fan}
if 
the matroidal subdivisions $\mathcal{D}_P$ and $\mathcal{D}_{P'}$ coincide.
In particular, the maximal cones in the secondary fan structure are the cones 
corresponding to the unrefinable positroid subdivisions.

In \cite{Herrmann2008HowTD}
it was shown that for the Dressian $Dr_{3,n}$, the Pl\"ucker fan structure 
and the secondary fan structure coincide.  And in \cite[Theorem 14]{Olarte} it was shown 
that the fan structures coincide for general Dressians $Dr_{k,n}$.
We can now just refer to 
the fan structure on $Dr^+_{k,n} = \Trop^+Gr_{k,n}$ without specifying either ``Pl\"ucker fan''
or ``secondary fan.''

  We have the following result.
\begin{corollary}\label{cor:reg0}
A collection $\mathcal{C} = \{S_{\pi}\}$ of positroid cells of $\Gr$
gives a regular positroid tiling 
of $\Delta_{k,n}$ (see \cref{def:dissection1}) if and only if 
this tiling has the form $\mathcal{D}_P$, for 
$P = \{P_I\}_{I \in \binom{[n]}{k}}$  a positive
tropical Pl\"ucker vector from a maximal cone of 
$\Trop^+Gr_{k,n}$.
\end{corollary}

\begin{proof}
Suppose that a collection
$\{S_{\pi}\}$ of positroid cells of $\Gr$
is a regular positroid tiling; in other words,
the images of the cells $\{S_{\pi}\}$ under the moment map are 
the top-dimensional 
positroid polytopes in the subdivision
$\mathcal{D}_P$
of $\Delta_{k,n}$, and the moment map is an 
 injection
on each $S_{\pi}$.  Therefore
by \cref{prop:homeo} and \cref{prop:homeo2}, 
$\dim S_{\pi}=n-1$, each positroid
$M_{\pi}$ is connected, and the reduced plabic graph associated
	to $\pi$ is a (planar) tree.

We claim that  the collection
$\{S_{\pi}\}$ gives an unrefineable possible positroid subdivision
of the hypersimplex.
That is, there is no nontrival way to subdivide one of the 
positroid polytopes $\Gamma_{\pi}$ into two full dimensional
	positroid polytopes.  If we \emph{can} subdivide $\Gamma_{\pi}$
	as above, and there is another full-dimensional positroid polytope
	$\Gamma_{\pi'}$ strictly contained in $\Gamma_{\pi}$, then 
	the bases of $M_{\pi'}$ are a subset of the bases of $\Gamma_{\pi}$,
	and hence the cell $S_{\pi'}$ lies in the closure of 
	$S_{\pi}$.  But then a reduced plabic graph $G'$ for $S_{\pi'}$ can 
	be obtained by deleting some edges from a reduced plabic
	graph $G$ for $S_{\pi}$; this means that $G'$ has fewer faces
	than $G$ and hence has the corresponding cell has smaller dimension,
	which is a contradiction, so the claim is true.

But now 	
the fact that $\{S_{\pi}\}$ gives an unrefineable positroid subdivision
means that it came from a maximal cone of $\Trop^+Gr_{k,n}$. 

Conversely, consider a regular positroid subdivision 
$\mathcal{D}_P$
coming from a maximal cone of 
$\Trop^+Gr_{k,n}$. 
Then the subdivision 
	$\mathcal{D}_P$ (which we identify with its top-dimensional
	pieces
	$\{S_{\pi}\}$) is an unrefineable positroid subdivision.
	In other words,
	none of the positroid polytopes $\Gamma_\pi$ can be subdivided
	into two full-dimensional positroid polytopes, which in turn 
	means that the reduced plabic graph corresponding to $\pi$
	must be a tree.  This implies that the moment
	map is an injection on each $S_{\pi}$ and hence 
	$\{S_{\pi}\}$ gives a regular positroid tiling of 
	$\Delta_{k,n}$.
\end{proof}

\begin{corollary}\label{cor:reg}
The number of regular positroid tilings of the
hypersimplex $\Delta_{k,n}$ equals
the number of maximal cones in the positive tropical Grassmannian
	$\Trop^+Gr_{k,n}$.
	\end{corollary}

The fact that the Pl\"ucker fan structure and the secondary fan structure on 
$\Trop^+Gr_{k,n}$ coincide also implies that the $f$-vector of $\Trop^+Gr_{k,n}$ reflects
the number of positroid subdivisions of $\Delta_{k,n}$ (with maximal cones corresponding
to unrefineable subdivisions and rays corresponding to coarsest subdivisions).

\section{Regular positroid subdivisions 
of \texorpdfstring{$\Delta_{k+1,n}$}{} and 
\texorpdfstring{$\mathcal{A}_{n,k,2}(Z)$}{}
from \texorpdfstring{$\Trop^+Gr_{k+1,n}$}{}}
\label{sec:regularA}

In \cref{sec:good}, we discussed the fact that arbitrary dissections of the hypersimplex
and the amplituhedron can have rather unpleasant properties, with their maximal cells
intersecting badly at their boundaries.  We 
introduced the notion of \emph{good dissections} for the hypersimplex and amplituhedron 
in \cref{def:good_dissection} and 
\cref{def:good_dissection_amp}.  Our goal in this section is to introduce a large class
of good dissections for the amplituhedron -- these are the 
\emph{regular positroid subdivisions}.

\subsection{Regular positroid subdivisions of $\mathcal{A}_{n,k,2}(Z)$}
Recall from \cref{def:regularpos} that 
the \emph{regular positroid subdivisions} 
 of $\Delta_{k+1,n}$
are precisely the dissections of the form 
 $\mathcal{D}_P$ 
(see \cref{sec:DP})
where $P = \{P_I\} \in \R^{[n] \choose k}$.
We know from \cref{T-duality}, \cref{sec:Ttriangle}, and \cref{sec:parity} that 
T-duality maps dissections of $\Delta_{k+1,n}$ to the amplituhedron $\mathcal{A}_{n,k,2}(Z)$
and preserves various nice properties along the way.  
We therefore use the T-duality map 
from \cref{hatmap}
to define \emph{regular (positroid) subdivisions}
of the $m=2$ amplituhedron $\mathcal{A}_{n,k,2}(Z)$.

\begin{definition}\label{def:regular}
We say that a positroid dissection of $\mathcal{A}_{n,k,2}(Z)$
is a \emph{regular positroid subdivision} if it has the form 
$\{S_{\hat{\pi}}\}$, where 
$\{S_{{\pi}}\}$ is a regular positroid subdivision of 
$\Delta_{k+1,n}$.
\end{definition}

As every regular positroid subdivision of $\Delta_{k+1,n}$ is a polyhedral subdivision
(and hence is good), 
 \cref{prop:subd} implies the following.
\begin{conj}\label{conj:regular}
Every regular positroid subdivision of $\mathcal{A}_{n,k,2}(Z)$ is a good dissection.
\end{conj}

In \cref{sec:experiments} we provide some computational evidence for \cref{conj:regular}.
For example, for $\mathcal{A}_{6,2,2}(Z)$ and $\mathcal{A}_{7,2,2}(Z)$,  every
regular positroid subdivision is good, and moreover, all good dissections are 
regular positroid subdivisions.  (This appears to also be the case for $\mathcal{A}_{8,2,2}(Z)$; 
but we were only able to compute the number of \emph{tilings} in this case.)
One might hope to strengthen \cref{conj:regular} and conjecture that the 
regular positroid subdivisions are precisely the good dissections.  
However, the notion of regularity is rather subtle (as usual in polyhedral geometry), and 
starting from $\mathcal{A}_{9,2,2}(Z)$, there are some good dissections which are not regular. 

\subsection{A large class of regular positroid tilings of $\Delta_{k+1,n}$ and $\mathcal{A}_{n,k,2}(Z)$}

\begin{definition}
Let $T$ be any planar trivalent tree with $n$ leaves (which will necessarily have $n-2$ internal
vertices), embedded in a disk with the leaves labelled from $1$ to $n$ in clockwise order.
Let $\mathcal{T}_{n,k}$ be the set of ${n-2 \choose k}$ plabic graphs obtained from $T$ by 
	colouring precisely $k$ of the internal vertices black, as in \cref{fig:trees}.
\end{definition}
\begin{figure}
\begin{center}
\begin{tabular}{c}
	\includegraphics[scale=0.18]{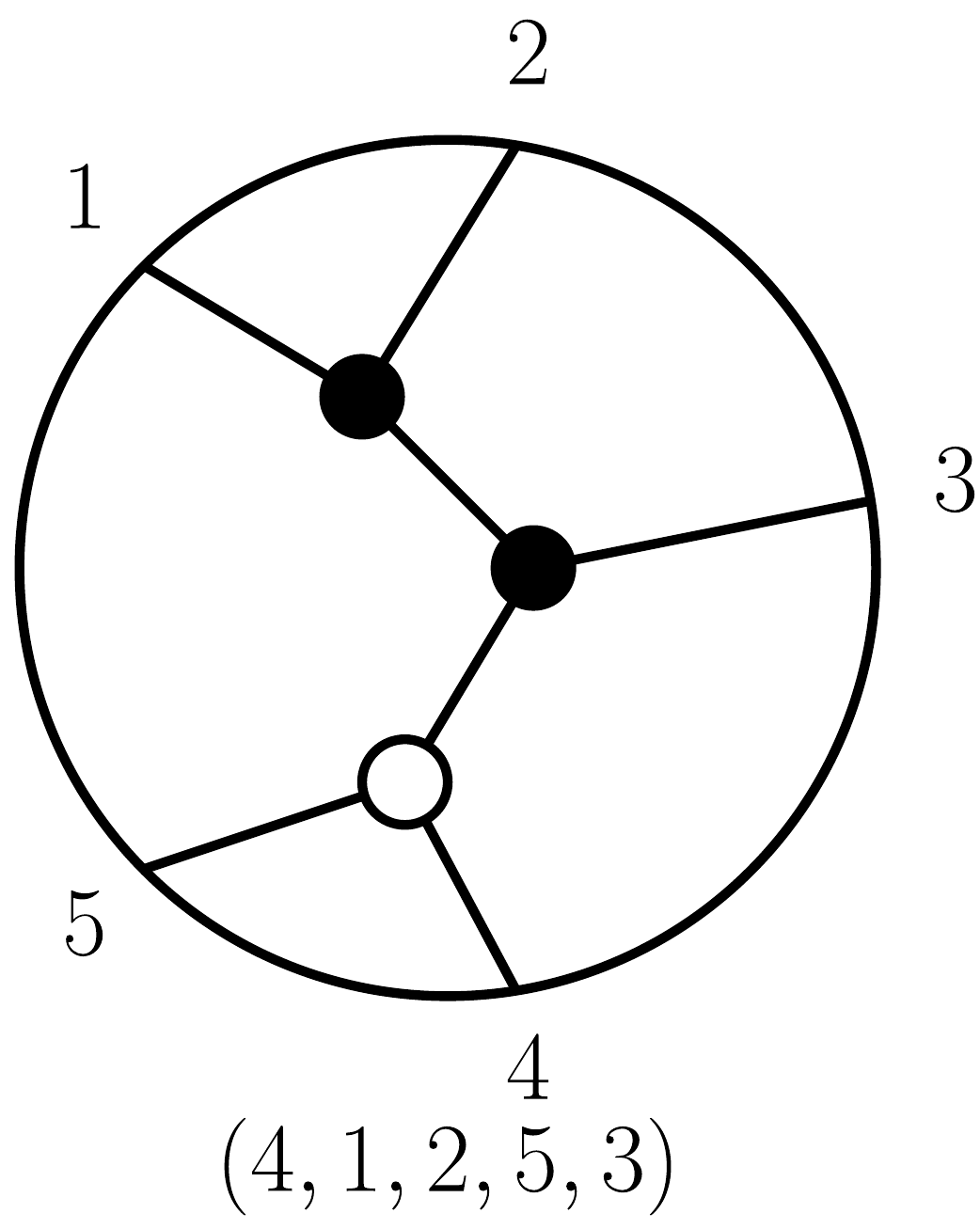}\quad\includegraphics[scale=0.18]{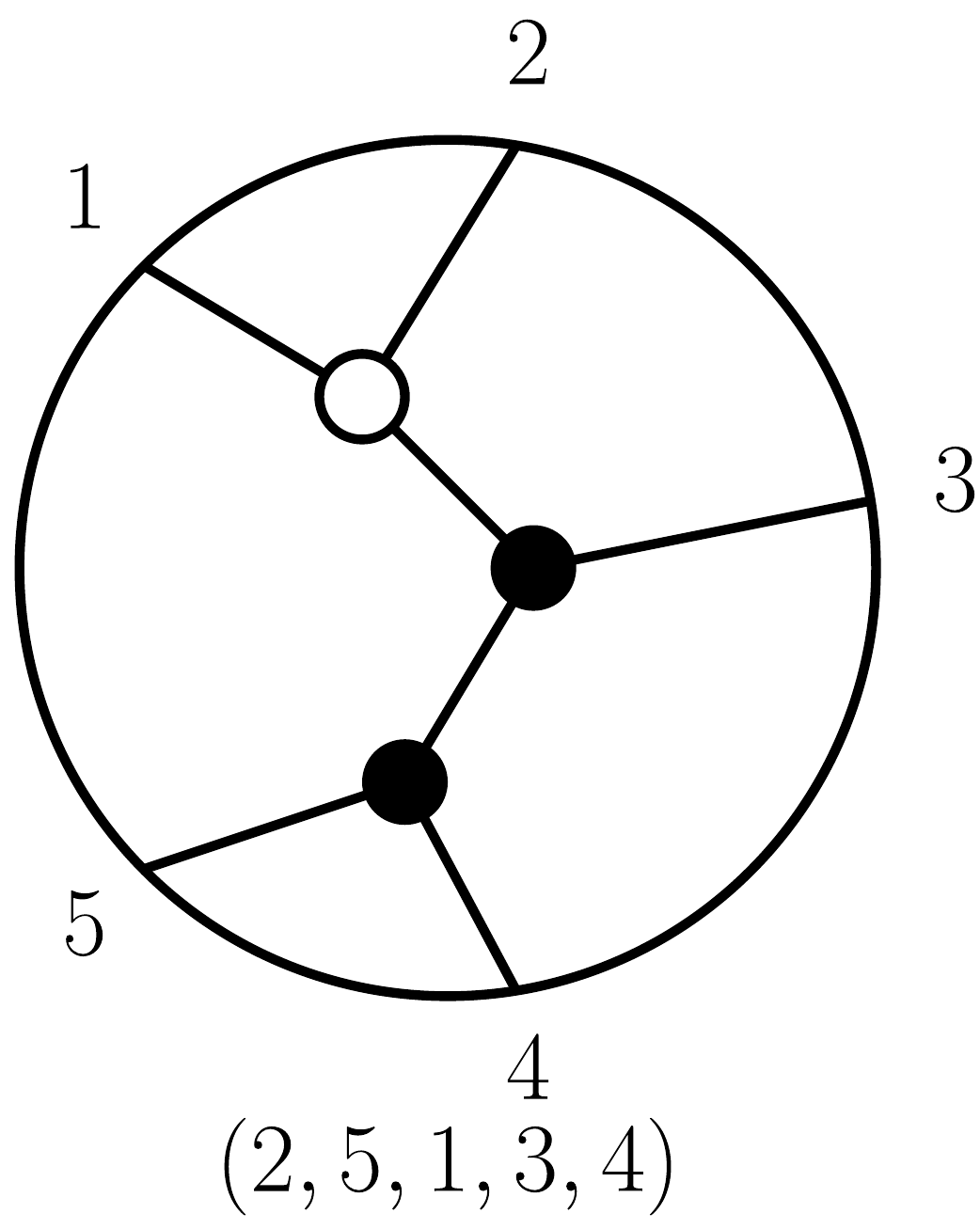}\quad\includegraphics[scale=0.18]{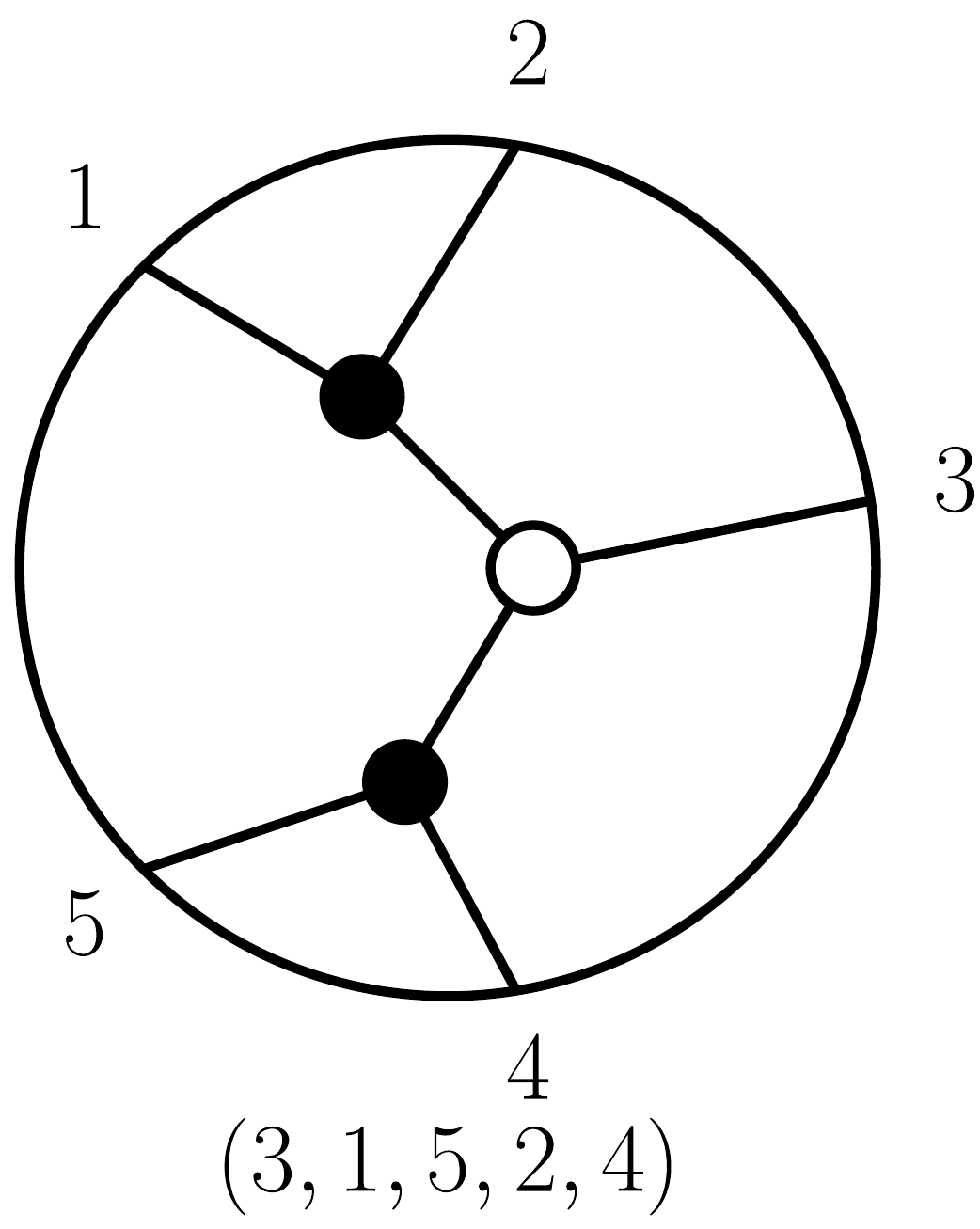}
\end{tabular}
\end{center}
	\label{fig:trees}
	\caption{The collection $\mathcal{T}_{5,2}$ of plabic graphs giving a regular subdivision of $\Delta_{3,5}$}
\end{figure}

\begin{proposition}\label{prop:largeclass}
The 
cells of $Gr^{\geq 0}_{k+1,n}$ corresponding to the plabic graphs in $\mathcal{T}_{n,k}$ give a regular
tiling of $\Delta_{k+1,n}$.  Therefore the images of these cells under the T-duality map
give a regular tiling of $\mathcal{A}_{n,k,2}(Z)$.  
\end{proposition}
\begin{proof}
We can use \cref{thm:dishyper} (see \cref{fig:hyp}) to inductively
prove that the cells corresponding
to $\mathcal{T}_{n,k}$ give a tiling of $\Delta_{k+1,n}$.
The fact that the cells corresponding to the plabic graphs 
in $\mathcal{T}_{n,k}$ give a \emph{regular} tiling
of $\Delta_{k+1,n}$ follows from \cite[Theorem 8.4]{Speyer}.  
Now using 
\cref{thm:shift}, it follows that the images of these cells under the T-duality map
	give a tiling of $\mathcal{A}_{n,k,2}(Z)$.  
The fact that this tiling is regular now follows from
\cref{def:regular}.
\end{proof}

\begin{remark}
The above construction gives us $C_{n-2}$ regular tilings of $\mathcal{A}_{n,k,2}(Z)$,
	where $C_n = \frac{1}{n+1} {2n \choose n}$ is the Catalan number.
\end{remark}

\subsection{The fan structure for regular positroid subdivisions}
We now discuss the fan structure for regular positroid 
subdivisions of the hypersimplex and amplituhedron.
\begin{definition}
	Given two subdivisions $\{\Gamma_{\pi}\}$ and $\{\Gamma_{\pi'}\}$ of $\Delta_{k+1,n}$, 
	we say that \emph{
		$\{\Gamma_{\pi}\}$ refines $\{\Gamma_{\pi'}\}$} and write
	$\{\Gamma_{\pi}\} \preceq \{\Gamma_{\pi'}\}$ if every $\Gamma_{\pi}$ is contained
	in some $\Gamma_{\pi'}$.

	Similarly, given two subdivisions $\{Z_{\pi}\}$ and $\{Z_{\pi'}\}$ of $\mathcal{A}_{n,k,2}(Z)$, 
	we say that \emph{
		$\{Z_{\pi}\}$ refines $\{Z_{\pi'}\}$} and write
	$\{Z_{\pi}\} \preceq \{Z_{\pi'}\}$ if every $Z_{\pi}$ is contained
	in some $Z_{\pi'}$.
\end{definition}

Recall from \cref{sec:fan} that we have a fan structure
on $\Trop^+Gr_{k+1,n}$ 
(the \emph{secondary fan}, which coincides with the \emph{Pl\"ucker fan})
which describes the regular positroid subdivisions
of $\Delta_{k+1,n}$, ordered by refinement.
We expect that this fan structure on 
$\Trop^+Gr_{k+1,n}$ also 
describes the regular positroid subdivisions of $\mathcal{A}_{n,k,2}(Z)$.

\begin{conj}\label{conj:fan}
The regular positroid subdivisions of $\mathcal{A}_{n,k,2}(Z)$
are parametrized by the cones of 
$\Trop^+Gr_{k+1,n}$, with the natural partial order on the cones
reflecting the refinement order on positroid subdivisions.
\end{conj}

\cref{conj:fan} is consistent with the following conjecture.

\begin{conj}
Consider two regular positroid 	
 subdivisions $\{\Gamma_{\pi}\}$ and $\{\Gamma_{\pi'}\}$ of $\Delta_{k+1,n}$, 
 and two corresponding positroid subdivisions 
$\{Z_{\hat{\pi}}\}$ and $\{Z_{\hat{\pi'}}\}$ of $\mathcal{A}_{n,k,2}(Z)$.
Then we have that $\{\Gamma_{\pi}\} \preceq \{\Gamma_{\pi'}\}$ 
if and only if 
$\{Z_{\hat{\pi}}\} \preceq \{Z_{\hat{\pi'}}\}$ 
\end{conj}

In particular, the regular positroid tilings of $\mathcal{A}_{n,k,2}(Z)$ should 
come precisely from the maximal cones of $\Trop^+Gr_{k+1,n}$.  More specifically,
if $\{P_I\}$ lies in a 
 maximal cone of $\Trop^+Gr_{k+1,n}$,
and $\{S_{\pi}\}$ is the regular positroid tiling corresponding to $\mathcal{D}_P$,
then $\{S_{\hat{\pi}}\}$ should be a regular positroid tiling of $\mathcal{A}_{n,k,2}(Z)$.
(Moreover, all regular positroid tilings of $\mathcal{A}_{n,k,2}(Z)$ should arise in this way.)

\subsection{The $f$-vector of $\Trop^+Gr_{k+1,n}$}
In light of \cref{conj:fan}, it is useful to compute the $f$-vector of the positive 
tropical Grassmannian.  
This is the vector $(f_0,f_1,\dots, f_{d})$ whose components compute the number of 
cones of fixed dimension.

As shown in \cite{troppos}, the positive tropical Grassmannian has an $n$-dimensional 
lineality space coming from the torus action.  However, one may mod out by this torus action
and study the resulting fan.
The method used in \cite{troppos} was to show that 
$\Trop^+Gr_{k,n}$ (a polyhedral subcomplex of $\R^{\binom{[n]}{k}}$)
is combinatorially equivalent to an $(n-k-1)(k-1)$-dimensional
fan $F_{k,n}$, obtained by using an ``$X$-cluster'' or ``web'' parametrization
of the positive Grassmannian, and modding out by the torus action.
As explained in \cite[Section 6]{troppos}, 
$F_{k,n}$ is the dual fan to the
Minkowski sum of the $\binom{n}{k}$ Newton
polytopes obtained by writing down each Pl\"ucker coordinate 
in the $X$-cluster parametrization.

Using this technique, \cite{troppos} computed 
the $f$-vector 
of $\Trop^+Gr_{2,n}$ (which is the $f$-vector of the associahedron, with maximal cones corresponding
to tilings of a polygon)
$\Trop^+Gr_{3,6}$, and 
$\Trop^+Gr_{3,7}$.
The above $f$-vector computations were recently extended 
in \cite{Arkani-Hamed:2019rds} using the notion of ``stringy canonical forms'' and in \cite{Borges:2019csl,Cachazo:2019xjx} using planar arrays and matrices of Feynman diagrams. 
 See also \cite{Henke:2019hve,Drummond:2019cxm, Early:2019eun} for recent, physics-inspired developments in this direction. 
We list all known results about maximal cones in the positive tropical  Grassmannian $\Trop^+Gr_{k+1,n}$ and their relation to tilings of hypersimplex $\Delta_{k+1,n}$  in \cref{tab:table}.

Apart from the $f$-vector of $\Trop^+Gr_{2,n}$, 
the known $f$-vectors of positive tropical Grassmannians $\Trop^+Gr_{k,n}$ (with $k \leq \frac{n}{2}$)
are the following:
\begin{align*}
	\Trop^+Gr_{3,6}: & (1, 48, 98, 66, 16, 1)\\
	\Trop^+Gr_{3,7}: & (1, 693, 2163, 2583, 1463, 392, 42, 1)\\
	\Trop^+Gr_{3,8}: & (1, 13612, 57768, 100852, 93104, 48544, 14088, 2072, 120,1)\\
	\Trop^+Gr_{4,8}: & (1, 90608, 444930, 922314, 1047200, 706042, 285948, 66740, 7984, 360, 1)
\end{align*}
For $\Trop^+Gr_{4,9}$ it is also known that the second component of the $f$-vector is 
30659424 \cite{Cachazo:2019xjx}.

\begin{remark} \label{clusteralgebras}
The coordinate ring of the Grassmannian has the structure of a \emph{cluster algebra} 
\cite{Scott}.  In particular, $Gr_{2,n}$, $Gr_{3,6}$, $Gr_{3,7}$, $Gr_{3,8}$ have 
cluster structures of finite types $A_n$, $D_4$, $E_6$, and $E_8$, respectively.
As discussed in \cite{troppos}, there is an intriguing connection between
 $\Trop^+Gr_{k,n}$ and the cluster structure.  In particular, 
 $F_{2,n}$ is the fan to the type $A_n$ associahedron, while 
  $F_{3,6}$ and $F_{3,7}$ are refinements of the fans associated to the $D_4$ and $E_6$ associahedra.
Via our correspondence between $\Trop^+Gr_{k+1,n}$ and the amplituhedron $\mathcal{A}_{n,k,2}(Z)$,
the Grassmannian cluster structure on $Gr_{k+1,n}$ should be reflected in good subdivisions
of $\mathcal{A}_{n,k,2}(Z)$.  In particular the type $A_n$ cluster structure should control 
	$\mathcal{A}_{n,1,2}(Z)$ (this is apparent, since $\mathcal{A}_{n,1,2}(Z)$ is a projective polygon),
	while the type $D_4$, $E_6$, and $E_8$ cluster structures should be closely related to 
	$\mathcal{A}_{6,2,2}(Z)$,
	$\mathcal{A}_{7,2,2}(Z)$, and 
	$\mathcal{A}_{8,2,2}(Z)$.
\end{remark}

\subsection{Experimental Data.}\label{sec:experiments}
Checks for this section\footnote{A more detailed discussion of these checks can be found in the arXiv version of this paper (v3).} for small values of $n$ and $k$ have been performed using Wolfram Mathematica. In particular, we used the packages \texttt{`positroid'}  \cite{Bourjaily:2012gy} and  \texttt{`amplituhedronBoundaries'} \cite{LM}. 
This allowed us to find the complete poset of good dissections of $\mathcal{A}_{6,2,2}$ and $\mathcal{A}_{7,2,2}$, whose $f$-vectors read:
\begin{align*}
&\mathcal{A}_{6,2,2}: (1,48,98,66,16,1)\\
&\mathcal{A}_{7,2,2}: (1, 693,2163,2583,1463,392,42, 1)\,.
\end{align*}
These are exactly the $f$-vectors of the positive tropical Grassmannian $\Trop^+Gr_{3,6}$ and $\Trop^+Gr_{3,7}$, respectively. 
For higher values of $n$ and $k$, we have been able to find all (good) tilings, and our findings\footnote{We also included there the results for $Gr^{\geq 0}_{3,9}$ which, by using our conjectures, can be derived from \cite{Cachazo:2019xjx}.} are summarized in \cref{tab:table}. 
\begin{table}[h!]
\begin{center}
\begin{tabular}{c|c|c|c|c}
{$(k,n)$}{}&Tilings&Good tilings&$\Trop^+Gr_{k+1,n}$&Non-regular good tilings\\
\hline\hline
$(1,n)$&$C_{n-2}$&$C_{n-2}$&$C_{n-2}$&0\\
\hline
\hline
$(2,5)$&5&5&5&0\\
\hline
$(2,6)$&120
&48&48&0\\
\hline
$(2,7)$&3073&693&693&0\\
\hline
$(2,8)$&6 443 460&13 612&13 612&0\\
\hline
$(2,9)$&?&346 806&346 710&96\\
\hline\hline
$(3,6)$&14&14&14&0\\
\hline
$(3,7)$&3073&693&693&0\\
\hline
$(3,8)$&?&91 496&90 608&888\\
\hline
$(3,9)$
&?&33 182 763& 30 659 424& 2 523 339
\end{tabular}
\end{center}
	\caption{New results about the tilings of the amplituhedron $\mathcal{A}_{n,k,2}(Z)$ in relation to known results about the number of maximal cones of the positive tropical Grassmannian $\Trop^+Gr_{k+1,n}$.}
	\label{tab:table}
\end{table}
In particular, we observe that for $\mathcal{A}_{8,2,2}(Z)$ the number of good tilings agrees with the number of maximal cones in $\Trop^+Gr_{3,8}$. Starting from $n=9$, the number of good tilings is larger than the number of maximal cones in positive tropical Grassmannian. It is indeed the first example where one can find good tilings which are \emph{not regular}. In particular, out of $346806$ good tilings, $96$ are not regular. Similarly, for $k=3$ and $n=8$, $888$ good tilings of $\mathcal{A}_{8,3,2}(Z)$ are not regular. We note that these correspond exactly to degenerate matrices found in \cite{Cachazo:2019xjx}.

\section{T-duality and the momentum amplituhedron for general (even) m}\label{sec:twistor}

Throughout the paper we have explored the remarkable connection between the hypersimplex and the $m=2$ amplituhedron. This was established via the T-duality map which allowed to relate positroid tiles, tilings, and dissections of both objects. 
It is then a natural question to wonder whether the story  generalizes for any (even) $m$. 

For $m=4$, we know that the amplituhedron $\mathcal{A}_{n,k,4}(Z)$ encodes the geometry of scattering amplitudes in $\mathcal{N}=4$ SYM, expressed in momentum twistor space. 
Physicists have already observed a beautiful connection between this and the formulation of scattering amplitudes of the same theory in 
momentum space\footnote{More precisely it is `spinor helicity' space, or, equivalently (related by half-Fourier transform), in twistor space. See \cite[Section 8]{abcgpt}.}.
At the core of this connection lies the Amplitude-Wilson Loop Duality \cite{Alday:2008yw}, which was shown to arise from a more fundamental duality in String Theory called `T-duality' \cite{Berkovits:2008ic}. 
For both formulations a Grassmannian representation has been found  \cite{Bullimore:2009cb,ArkaniHamed:2009dn}: scattering amplitudes (at tree level) are computed by performing a contour integral around specific cycles inside the positive Grassmannian (what in physics is referred to as a `BCFW contour'). 
If we are in momentum space, then one has to integrate over cycles corresponding to collections of $(2n-4)$-dimensional positroid cells of $Gr^{\geq 0}_{k+2,n}$. Whereas, if we are in momentum twistor space, the integral is over collections of $4k$-dimensional positroid cells of $Gr^{\geq 0}_{k,n}$. 
The two integrals compute the same scattering amplitude, and it was indeed shown that that formulas are related by a change of variables. In particular, this implied the existence of a map between certain $(2n-4)$-dimensional positroid cells of $Gr^{\geq 0}_{k+2,n}$ and certain $4k$-dimensional positroid cells of $Gr^{\geq 0}_{k,n}$ (called `BCFW'), which was defined in \cite[Formula (8.25)]{abcgpt}.
It is easy to see that this map is exactly our T-duality map for the case $m=4$ in \eqref{Tduality_gen}, up to a cyclic shift:
\begin{equation}
\sigma \hat{\pi}(i)=\pi(i-\frac{m}{2}+1)-1=\pi(i-1)-1.
\end{equation}

Collections of $4k$-dimensional `BCFW' positroid cells of $Gr^{\geq 0}_{k,n}$ defined from physics were conjectured to triangulate $\mathcal{A}_{n,k,4}(Z)$. The main results in the literature towards proving this conjecture can be found in \cite{Karp:2017ouj}. On the other hand, the corresponding collections of $(2n-4)$-dimensional `BCFW' positroid cells of $Gr^{\geq 0}_{k+2,n}$ were conjectured to triangulate an object $\mathcal{M}_{n,k,4}(\Lambda,\widetilde\Lambda)$ called `momentum amplituhedron', introduced recently by two of the authors in \cite{mamp}\footnote{In the paper, the momentum amplituhedron was denoted as $\mathcal{M}_{n,k}$, without the subscript `$4$'.}. 

The story aligns with the philosophy of the rest of this paper. In particular, one aims to seek for an object and a map which relates its tiles, tilings (and, more generally, dissections) to the ones of $\mathcal{A}_{n,k,m}(Z)$, for general (even) $m$.
There is a natural candidate for such a map: we have already seen that the T-duality map defined in \eqref{Tduality_gen} does indeed the job in the case of $m=2$ and $m=4$. 
Moreover, some of the statements which has been proven throughout the paper for $m=2$, as \cref{prop:2} and \cref{thm:pardual}, can be generalized for general (even) $m$.

\begin{proposition}\label{prop:2genm}
	Let $S_{\pi}$ be a 
	cell of $Gr^{\geq 0}_{k+\frac{m}{2},n}$ such that, as affine permutation, $\pi(i)\geq i+\frac{m}{2}$.
	Then $S_{\hat{\pi}}$ is a 
	 cell of $Gr_{k,n}^{\geq 0}$ such that $\hat{\pi}(i) \leq i+n-\frac{m}{2}$. Moreover, 
		$\dim(S_{\hat{\pi}}) -mk = \dim(S_\pi) - \frac{m}{2}(n-\frac{m}{2}).$
	In particular, if $\dim S_{\pi} = \frac{m}{2}(n-\frac{m}{2})$,
	then $\dim S_{\hat{\pi}} = mk$.
\end{proposition}
 \begin{proof}
This is a straightforward generalization of the proof of  \cref{prop:2}. It is enough to observe that, in the language of affine permutations, T-duality maps a $(k+m/2,n)$-bounded affine permutation $\pi_a$ into a $(k,n)$-bounded affine permutation $\hat{\pi}_a=\pi_a \circ t^{m/2}$, with $t^{m/2}:\mathbb{Z} \rightarrow \mathbb{Z}$ the map $i \mapsto i-m/2$. Clearly, $t^{m/2}$ preserve the length of affine permutations. Hence the codimensions of $S_{\pi_a} \subseteq Gr^{\geq 0}_{k+\frac{m}{2},n}$ and $S_{\hat{\pi}_a} \subseteq Gr^{\geq 0}_{k,n}$ are equal.
\end{proof}

It is also natural to think of
parity duality between $\mathcal{A}_{n,k,m}(Z)$ and $\mathcal{A}_{n,n-k-m,m}(Z')$
as a composition of the Grassmannian duality and T-duality
(plus cyclic shifts).
Imitating \cref{def:dualmapm2}, let us define
	$\widetilde{U}_{k,n,m}(\hat{\pi}):=\widehat{{\pi}^{-1}}$. Then we have the following theorem:

\begin{thm}[Parity duality from T-duality and Grassmannian duality]\label{thm:pardualgenm} Let $\lbrace Z_{\pi} \rbrace$ be a collection of Grasstopes which dissects the amplituhedron $\mathcal{A}_{n,k,m}(Z)$. Then the collection of Grasstopes $\lbrace Z_{\widetilde{U}_{k,n,m}\pi}\rbrace$ dissects the amplituhedron $\mathcal{A}_{n,n-k-m,m}(Z')$. 
\end{thm}
\begin{proof}
The parity duality $U_{k,n,m}$ in \cite{Galashin:2018fri}
was defined for any (even) $m$ as: $U_{k,n,m}(\pi):=(\pi-k)^{-1}+(n-k-m)$. 
	Then it easy to show that ${U}_{k,n,m}=\sigma^{k+\frac{m}{2}} \circ \widetilde{U}_{k,n,m}$. Using \cref{th:cyclicsym}, the prove follows immediately. 
\end{proof}

Since we found a natural candidate map, we now introduce a candidate object, which would speculatively relate to $\mathcal{A}_{n,k,m}(Z)$ via the T-duality map. This is a generalization of the momentum amplituhedron $\mathcal{M}_{n,k,4}(\Lambda,\widetilde\Lambda)$ and it is defined below.

\begin{definition}
For $k,n$ such that $k \leq n$, define the \emph{twisted positive part} of $Gr_{k,n}$ as:
\begin{equation}
Gr^{+,\tau}_{k,n} := \lbrace X \in Gr_{k,n}: (-1)^{\mbox{inv}(I, [n] \setminus I)}\Delta_{[n] \setminus I}(X) \geq 0  \rbrace
\end{equation}
where $\mbox{inv}(A,B):=\# \lbrace a \in A, b \in B | a > b \rbrace $ denotes the inversion number.
\end{definition}

The lemma below can be found in \cite[Lemma 1.11]{karp},
which  sketched a
proof and attributed it to Hochster and Hilbert.
\begin{lem}
Suppose $\Delta_I(V)$ are the Pl\"ucker coordinates of a point $V \in Gr_{k,n}$. Then the kernel $V^\perp \in Gr_{n-k,n}$ of $V$ is represented by the point with Pl\"ucker coordinates $\Delta_J (V^\perp)=(-1)^{\mbox{\emph{inv}}(J, [n] \setminus J)}\Delta_{[n] \setminus J}(V)$ for $J \in \binom{[n]}{n-k}$.
\end{lem}

\begin{definition}
For $a,b$ such that $a \leq b$, define $\mbox{Mat}^{>0}_{a,b}$ the set of real $a \times b$ matrices whose $a \times a$ minors are all positive and its \emph{twisted positive part} as
\begin{equation}
\mbox{Mat}^{>0,\tau}_{a,b} := \lbrace A \in \mbox{Mat}_{a,b}: (-1)^{\mbox{inv}(I, [b] \setminus I)}\Delta_{[b] \setminus I}(A) > 0  \rbrace
\end{equation}
\end{definition}

\begin{definition}[The momentum amplituhedron]\label{deftwistroampl}
 Let  $\tilde{\Lambda} \in \mbox{Mat}^{>0}_{n,k'+\frac{m}{2}}, \Lambda \in \mbox{Mat}^{>0,\tau}_{n,n-k'+\frac{m}{2}}$, $k' +m/2 \leq n$. The \emph{momentum amplituhedron map} $\Phi_{\tilde{\Lambda},\Lambda}: Gr^{\geq 0}_{k',n} \rightarrow Gr_{k',k'+\frac{m}{2}} \times Gr_{n-k',n-k'+\frac{m}{2}}$ is defined by $\Phi_{\tilde{\Lambda},\Lambda}(C):=(C \tilde{\Lambda},C^\perp \Lambda)$, where $C$ and $C^\perp$ are matrices representing an element of $Gr_{k',n}^{\ge 0}$ and its orthogonal in $Gr_{k',n}^{\ge 0,\tau}$ respectively, and $C \tilde{\Lambda}$ and $C^\perp \Lambda$ matrices representing an element of $Gr_{k',k'+\frac{m}{2}}$ and $ Gr_{n-k',n-k'+\frac{m}{2}}$ respectively.
The \emph{momentum amplituhedron} $\mathcal{M}_{n,k',m}(\Lambda,\tilde{\Lambda}) \subseteq Gr_{k',k'+\frac{m}{2}} \times Gr_{n-k',n-k'+\frac{m}{2}}$ is the image $\Phi_{\tilde{\Lambda},\Lambda}(Gr^{\geq 0}_{k',n})$.
\end{definition}

\begin{proposition}[Momentum conservation]\label{momcons}
Let $(\tilde{Y},Y) $ represent a point in $Gr_{k',k'+\frac{m}{2}} \times Gr_{n-k',n-k'+\frac{m}{2}}$ and let $\tilde{Y}^\perp$ and $Y^\perp$ be matrices representing the orthogonal complements of $Y$ and $\tilde{Y}$, respectively. 
If $(\tilde{Y},Y)$ is in the momentum amplituhedron $\mathcal{M}_{n,k',m}(\Lambda,\widetilde\Lambda)$, then
\begin{equation}
(Y^\perp \Lambda^T) \cdot (\tilde{Y}^\perp \tilde{\Lambda}^T)^T =0
\end{equation}
\end{proposition}

\begin{proof}
From the identity 
\begin{equation}
0=Y^\perp  Y^T=Y^\perp \Lambda^T (C^\perp)^T
\end{equation}
we deduce that the row-span of $Y^\perp \Lambda^T$ is included in the row-span of the orthogonal of $C^\perp$, i.e. $C$. Analogously, from 
\begin{equation}
0=\tilde{Y}^\perp  \tilde{Y}^T=\tilde{Y}^\perp \tilde{\Lambda}^T C
\end{equation}
we deduce that the row-span of $\tilde{Y}^\perp \tilde{\Lambda}^T$ is included in the row-span of the $C^\perp$. Therefore $Y^\perp \Lambda^T$ and $\tilde{Y}^\perp \tilde{\Lambda}^T$ belong to orthogonal subspaces and satisfy
\begin{eqnarray}
(Y^\perp \Lambda^T) \cdot (\tilde{Y}^\perp \tilde{\Lambda}^T)^T=0.
\end{eqnarray}
\end{proof}

\begin{remark}
In reference to \cref{deftwistroampl}, we observe that: 
\begin{equation}
\mbox{dim} \left( Gr_{k',k'+m/2} \times Gr_{n-k',n-k'+m/2} \right)=\frac{m}{2}k'+\frac{m}{2}\left(n-k'\right)=\frac{m}{2}n.
\end{equation}
Moreover, \cref{momcons} implies that the momentum amplituhedron $\mathcal{M}_{n,k,m}$ is included in a codimension $(\frac{m}{2})^2$ sub-variety of $Gr_{k',k'+m/2} \times Gr_{n-k',n-k'+m/2}$.
Therefore, the dimension of $\mathcal{M}_{n,k',m}$ is at most (and conjectured to be exactly):
\begin{equation}
\frac{m}{2}n-\left(\frac{m}{2}\right)^2=\frac{m}{2}\left(n-\frac{m}{2} \right).
\end{equation}
We observe that, for $m=2$, this dimension is exactly $n-1$, which is the dimension of the hypersimplex $\Delta_{k+1,n}$; whereas, for $m=4$, the dimension is $2n-4$, which is the one of BCFW cells in momentum space.
\end{remark}

\begin{remark} \label{m2ampl}
For $m=2$,  \cref{deftwistroampl} reads:
\begin{equation}
\Phi_{\tilde{\Lambda},\Lambda}: Gr^{\geq 0}_{k',n} \rightarrow Gr_{k',k+1} \times Gr_{n-k',n-k'+1} \cong \mathbb{P}^{k'} \times \mathbb{P}^{n-k'}.
\end{equation}
Moreover, the conditions in \cref{momcons} are equivalent to:
\begin{equation} \lambda \cdot \tilde{\lambda}=0
\end{equation}
where we used the dot product in $\mathbb{R}^n$ of the vectors $\lambda:= {\Lambda}(Y^\perp)^T $ and  $\tilde{\lambda}:= {\tilde{\Lambda}}(\tilde{Y}^\perp)^T $.

Note that the $m=2$ momentum amplituhedron is not equal to the hypersimplex, as pointed out in \cite{Lukowski:2021amu}. 
\end{remark}
\begin{remark} For $m=4$, \cref{deftwistroampl} coincides with the one in \cite{mamp}. This is the positive geometry relevant for scattering amplitudes for $\mathcal{N}=4$ SYM in spinor helicity space.
\end{remark}

Many properties of $\mathcal{M}_{n,k,4}(\Lambda,\widetilde\Lambda)$ have still to be explored and proven. 
Let $\Phi_{\pi}$ denote the image under the amplituhedron map $\Phi_{\Lambda,\tilde{\Lambda}}(\bar{S}_{\pi})$ of (the closure of) a positroid cell $S_{\pi}$ in $Gr_{k',n}$.  Analogously to the amplituhedron, we call $\Phi_{\pi}$ a \emph{positroid tile} of $\mathcal{M}_{n,k,4}(\Lambda,\widetilde\Lambda)$ if it is full-dimensional and if the momentum amplituhedron map is injective on $S_{\pi}$. We also define \emph{positroid tilings} of $\mathcal{M}_{n,k,4}(\Lambda,\widetilde\Lambda)$ collections $\{\Phi_{\pi}\}$ of positroid tiles whose interior is disjoint and cover $\mathcal{M}_{n,k,4}(\Lambda,\widetilde\Lambda)$.
Then the conjecture in \cite{mamp} can be stated as:

\begin{conj}\label{conj:momamplm4} \cite{mamp}
There exists an open subset $\mathcal{P} \subset \Mat^{>0,\tau}_{n,k'+2} \times \Mat^{>0}_{n,n-k'+2}$ such that for all $(\Lambda,\widetilde\Lambda) \in \mathcal{P}$ a collection of positroid tiles $\{\Phi_{\pi}\}$ is a positroid tiling 
	(respectively, dissection) of $\mathcal{M}_{n,k+2,4}(\Lambda,\widetilde\Lambda)$  if and only if 
for all $Z \in \Mat^{>0}_{n,k+4} $ the collection of T-dual Grasstopes $\{Z_{\hat{\pi}}\}$ is a tiling (respectively, dissection) of $\mathcal{A}_{n,k,4}(Z)$.
\end{conj}
\begin{rmk}
\cite{mamp} provided experimental evidence that a subset $\mathcal{P}$ with the properties above can be obtained by imposing positivity of \emph{planar Mandelstam variables}. In particular, choosing the rows of $\Lambda^\perp$ and $\widetilde\Lambda$ on the moment curve as $(\Lambda^\perp)_{i,a}=i^a, \widetilde\Lambda_{i,\dot a}=i^{\dot a}$, with $ i \in [n], a \in [k'-2], \dot a \in [k'+2]$ would give a point in $\mathcal{P}$.  
\end{rmk}

Finally, we speculate that:
\begin{conj}\label{conj:momamplgenm}
Let $m$ be a multiple of $4$ and $k'=k+m/2$.
There exists an open subset $\mathcal{P} \subset \Mat^{>0,\tau}_{n,k'+\frac{m}{2}} \times \Mat^{>0}_{n,n-k'+\frac{m}{2}}$ such that for all $(\Lambda,\widetilde\Lambda) \in \mathcal{P}$ a collection $\{\Phi_{\pi}\}$ of positroid tiles is
a tiling 
	(respectively, dissection) of $\mathcal{M}_{n,k',m}(\Lambda,\widetilde\Lambda)$ if and only if 
the collection of T-dual Grasstopes $\{Z_{\hat{\pi}}\}$ is a
	tiling (respectively, dissection) of $\mathcal{A}_{n,k,m}(Z)$.
\end{conj}

\section{Appendix. Combinatorics of cells of the positive Grassmannian.}\label{app}

In \cite{postnikov}, Postnikov 
classified the cells of the positive Grassmannian, showing that 
the \emph{positroid cells} could be indexed by 
\emph{decorated permutations}
 and also equivalence classes of \emph{reduced plabic graphs}. 
We review these objects here.  
This will give us a canonical way to label each positroid by a decorated permutation or
an equivalence class of plabic graphs.  
We refer to reader to \cite{postnikov} or 
\cite[Section 2]{Karp:2017ouj} for more details.

\begin{defn}\label{defn:decperm}
A \emph{decorated permutation} on $[n]$ is a bijection $\pi : [n] \to [n]$ whose fixed points are each coloured either black (loop) or white (coloop). We denote a black fixed point $i$ by $\pi(i) = \underline{i}$, and a white fixed point $i$ by $\pi(i) = \overline{i}$.
An \emph{anti-excedance} of the decorated permutation $\pi$ is an element $i \in [n]$ such that either $\pi^{-1}(i) > i$ or $\pi(i)=\overline{i}$.  We say that a decorated permutation
	on $[n]$ is of
	\emph{type $(k,n)$} if it has $k$ anti-excedances.
\end{defn}

For example, $\pi = (3,\underline{2},5,1,6,8,\overline{7},4)$ has 
a loop in position $2$, and a coloop in position $7$.
It has three anti-excedances, in positions $4, 7, 8$.

\begin{definition}
Given a $k\times n$ matrix $C = (c_1,\dots,c_n)$ written as a list of its columns,
we associate a decorated  permutation $\pi:=\pi_C$ as follows.
We set $\pi(i):=j$ to be the label of the first column $j$ such that 
$c_i \in \spn \{c_{i+1}, c_{i+2},\dots, c_j\}$.
If $c_i$ is the all-zero vector, we call $i$ a loop or black fixed point
	and if $c_i$ is not in the span of the other column
vectors, we call $i$ a coloop or white fixed point.
	We let $$S_{\pi} = \{C\in Gr^{\geq 0}_{k,n} \ \vert \ \pi_C=\pi\}.$$
\end{definition}

Postnikov showed that  $S_{\pi}$ is a cell, and that  the positive
Grassmannian 
$Gr_{k,n}^{\ge 0}$ is the union of cells $S_{\pi}$ 
where $\pi$ ranges over decorated permutations of type $(k,n)$
\cite[Section 16]{postnikov}.

Decorated permutations can be equivalently thought of as affine permutations \cite{KLS}.
\begin{defn}\label{defn:affperm}
	An \emph{affine permutation} on $[n]$ is a bijection $\pi : \mathbb{Z} \to \mathbb{Z}$ such that for all $i \in \mathbb{Z}$, 
	$\pi(i+n)=\pi(i)+n$ and $i \leq \pi(i) \leq i+n$.
	If 
	$\sum_{i=1}^n (\pi(i)-i) = kn$
	we say $\pi$ is  \emph{$(k,n)$-bounded}. 
\end{defn}

There is a bijection between decorated permutations of type $(k,n)$ and $(k,n)$-bounded 
affine permutations. Given a decorated permutation $\pi_d$ we can define an affine permutation $\pi_a$ by the following procedure: if $\pi_d(i)>i$, then define $\pi_a(i):=\pi_d(i)$; if $\pi_d(i)<i$, then define $\pi_a(i):=\pi_d(i)+n$; if $\pi_d(i)$ is a loop then define $\pi_a(i):=i$; if $\pi_d(i)$ is a coloop then define $\pi_a(i):=i+n$. 
For example, under this map, the decorated permutation $\pi_d=(3,\underline{2},5,1,6,8,\overline{7},4)$ in the previous example gives rise to $\pi_a=(3,2,5,9,6,8,15,12)$.

Let a pair $(i,j)$ be an \emph{inversion} of $\pi_a$ if $i,j\in \Z, i<j$, and $\pi_a(i)>\pi_a(j)$. Two inversions $(i,j)$ and $(i',j')$ are \emph{equivalent} if $i'-i=j'-j \in n \mathbb{Z}$. Then the \emph{length} $\ell(\pi_a)$ of $\pi_a$ is defined to be the number of equivalence classes of inversions. We note that $\ell(\pi_a)$ equals the number of \emph{alignments} of the associated decorated permutation $\pi_d$ (see \cite[Section~5]{postnikov}).

Positroid cells can also be represented by \emph{plabic graphs}.
\begin{defn}
A {\it plabic graph}\footnote{``Plabic'' stands for {\itshape planar bi-colored}.}  is an undirected planar graph $G$ drawn inside a disk
(considered modulo homotopy)
with $n$ {\it boundary vertices} on the boundary of the disk,
	labeled $1,\dots,n$ in clockwise order, as well as some {\it internal vertices}. Each boundary vertex is incident to a single edge, and each internal vertex is colored either black or white. If a boundary vertex is incident to a leaf (a vertex of degree $1$), we refer to that leaf as a \emph{lollipop}. {We will assume that $G$
	has no internal leaves except for lollipops.}
\end{defn}

\begin{definition}
A {\it perfect orientation\/} $\O$ of a plabic graph $G$ is a
choice of orientation of each of its edges such that each
black internal vertex $u$ is incident to exactly one edge
directed away from $u$; and each white internal vertex $v$ is incident
to exactly one edge directed towards $v$.
A plabic graph is called {\it perfectly orientable\/} if it admits a perfect orientation.
Let $G_\O$ denote the directed graph associated with a perfect orientation $\O$ of $G$.
The {\it source set\/} $I_\O \subset [n]$ of a perfect orientation $\O$ is the set of $i$ which 
	are sources
 of the directed graph $G_\O$. Similarly, if $j \in \overline{I}_{\O} := [n] - I_{\O}$, then $j$ is a sink of $\O$.
\end{definition}

\cref{fig:orientation} shows a plabic graph with a perfect orientation. In that example, $I_{\O} = \{2,3,6,8\}$.
 \begin{figure}[ht]
   \includegraphics[scale=0.9]{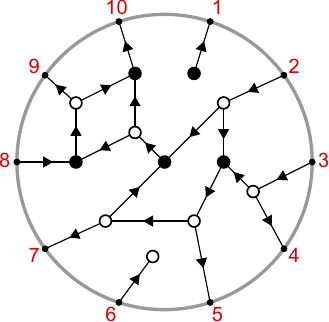}
	 \caption{A plabic graph with a perfect orientation.}
   \label{fig:orientation}
 \end{figure}

All perfect orientations of a fixed plabic graph
$G$ have source sets of the same size $k$, where
$k-(n-k) = \sum \mathrm{color}(v)\cdot(\deg(v)-2)$.
Here the sum is over all internal vertices $v$, $\mathrm{color}(v) = 1$ for a black vertex $v$,
and $\mathrm{color}(v) = -1$ for a white vertex;
see~\cite{postnikov}.  In this case we say that $G$ is of {\it type\/} $(k,n)$.
	
Now let us connect plabic graphs to 
the positroids and positroid cells from \cref{def:positroid}.
\begin{theorem}
	[{\cite[Section 11]{postnikov}}]
	\label{prop:perf}
Let $G$ be a plabic graph of type $(k,n)$.
Then we have a positroid
$M_G$ on $[n]$ defined by
$$M_G = \{I_{\O} \mid \O \text{ is a perfect orientation of }G\},$$
where $I_\O$ is the set of sources of $\O$.  Moreover, every 
	positroid cell has the form $S_{M_G}$ for some plabic graph $G$.
\end{theorem}

One can also read off the positroid from $G$ using \emph{flows} \cite{Talaska}
or  perfect matchings.

If a plabic graph $G$ is \emph{reduced} (see \cite[Section 12]{postnikov})
or \cite[Chapter 7]{FWZ}),
we have that $S_{M_G} = S_{\pi_G}$, where $\pi_G$ is the decorated permutation
defined as follows.
\begin{defn}\label{def:rules}
Let $G$ be a reduced plabic graph  with boundary vertices $1,\dots, n$. For each boundary vertex $i\in [n]$, we follow a path along the edges of $G$ starting at $i$, turning (maximally) right at every internal black vertex, and (maximally) left at every internal white vertex. This path ends at some boundary vertex $\pi(i)$. By \cite[Section 13]{postnikov}, the fact that $G$ is reduced implies that each fixed point of $\pi$ is attached to a lollipop; we color each fixed point by the color of its lollipop. This 
	defines a decorated permutation, called the \emph{decorated trip permutation} $\pi_G = \pi$ of $G$. 
\end{defn}

\bibliographystyle{alpha}
\bibliography{bibliography}

\end{document}